\documentclass[a4paper]{article}

\usepackage{amsfonts}
\usepackage{amsmath}
\usepackage{amssymb}
\usepackage{amsthm}
\usepackage[UKenglish]{babel}
\usepackage{Baskervaldx}
\usepackage{bm}
\usepackage{booktabs}
\usepackage{caption}
\usepackage[noadjust]{cite}
\usepackage{enumitem}
    \setlist{topsep=0pt,itemsep=0pt}
\usepackage{faktor}
\usepackage{float}
\usepackage[T1]{fontenc}
\usepackage[margin=25mm]{geometry}
\usepackage{mathscinet}
\usepackage{mathtools}
\usepackage{microtype}
\usepackage{sectsty}
    \sectionfont{\large}
    \subsectionfont{\normalsize}
\usepackage{tabularx}
\usepackage{tikz}
    \usetikzlibrary{arrows.meta}
\usepackage{titlesec}
    \titlespacing{\section}{0pt}{12pt}{0pt}
    \titlespacing{\subsection}{0pt}{6pt}{0pt}
\usepackage{titling}
\usepackage[nottoc]{tocbibind}
\usepackage{tocloft}
\usepackage[normalem]{ulem}
\usepackage{url}
\usepackage{ytableau}

\usepackage[colorlinks=true, linkcolor=blue, citecolor=blue, urlcolor=magenta, breaklinks=true]{hyperref}
\usepackage[capitalise,noabbrev]{cleveref}
    \crefname{equation}{equation}{equations}
    \crefname{conjecture}{Conjecture}{Conjectures}

\theoremstyle{plain}
    \newtheorem{theorem}{Theorem}
    \newtheorem{proposition}[theorem]{Proposition}
    \newtheorem{corollary}[theorem]{Corollary}
    \newtheorem{conjecture}[theorem]{Conjecture}
    \newtheorem{theoremx}{Theorem}
        
    \newtheorem{conjecturex}[theoremx]{Conjecture}
        
    \numberwithin{theorem}{section}
\theoremstyle{definition}
    \newtheorem{definition}[theorem]{Definition}
    \newtheorem{example}[theorem]{Example}
    \newtheorem{remark}[theorem]{Remark}

\renewcommand{\leq}{\leqslant}
\renewcommand{\geq}{\geqslant}

\newcommand{\dd}{\mathrm{d}}
\newcommand{\h}{\hbar}
\newcommand{\e}{\mathfrak{e}}
\newcommand{\m}{\mathfrak{m}}
\newcommand{\n}{\mathfrak{n}}
\newcommand{\p}{\mathfrak{p}}
\newcommand{\w}{\mathfrak{w}}
\renewcommand{\O}{\mathbf{O}}
\newcommand{\J}{\mathcal{J}}
\renewcommand{\P}{\mathcal{P}}
\newcommand{\V}{\mathcal{V}}
\newcommand{\X}{\mathcal{X}}
\newcommand{\Z}{\mathcal{Z}}
\newcommand{\hb}{\vec{h}^{(b)}}
\newcommand{\Hb}{\vec{H}^{(b)}}
\newcommand{\hbt}{\vec{h}^{(bt)}}
\newcommand{\Hbt}{\vec{H}^{(bt)}}
\newcommand{\wga}{\mathrm{Wg}^{\mathbf{A}}}
\newcommand{\wgb}{\mathrm{Wg}^{(b)}}
\newcommand{\wgo}{\mathrm{Wg}^{\mathbf{O}}}
\newcommand{\wgbt}{\mathrm{Wg}^{(bt)}}
\newcommand{\Tab}{\mathsf{Tab}}
\newcommand{\mmid}{\,|\,}

\title{From Weingarten calculus for real Grassmannians to deformations of monotone Hurwitz numbers and Jucys--Murphy elements}
\author{Xavier Coulter \and Norman Do}

\begin{document}

\makeatletter
\textbf{\large \thetitle}

\textbf{\theauthor}
\makeatother

Department of Mathematics, The University of Auckland, Auckland 1142 New Zealand \\
School of Mathematics, Monash University, VIC 3800 Australia \\
Email: \href{mailto:xavier.coulter@auckland.ac.nz}{xavier.coulter@auckland.ac.nz}, 
\href{mailto:norm.do@monash.edu}{norm.do@monash.edu}

\emph{Abstract.} The present work is inspired by three interrelated themes: Weingarten calculus for integration over unitary groups, monotone Hurwitz numbers which enumerate certain factorisations of permutations into transpositions, and Jucys--Murphy elements in the symmetric group algebra. The authors and Moskovsky recently extended this picture to integration on complex Grassmannians, leading to a deformation of the monotone Hurwitz numbers to polynomials that are conjectured to satisfy remarkable interlacing phenomena.

In this paper, we consider integration on the real Grassmannian $\mathrm{Gr}_\mathbb{R}(M,N)$, interpreted as the space of $N \times N$ idempotent real symmetric matrices of rank $M$. We show that in the regime of large $N$ and fixed~$\frac{M}{N}$, such integrals have expansions whose coefficients are variants of monotone Hurwitz numbers that are polynomials in the parameter $t = 1 - \frac{N}{M}$.

We define a ``$b$-Weingarten calculus'', without reference to underlying matrix integrals, that recovers the unitary case at $b = 0$ and the orthogonal case at $b = 1$. The $b$-monotone Hurwitz numbers, previously introduced by Bonzom, Chapuy and Do\l\polhk{e}ga, arise naturally in this context as monotone factorisations of pair partitions. The $b$- and $t$-deformations can be combined to form a common generalisation, leading to the notion of $bt$-monotone Hurwitz numbers, for which we state several results and conjectures.

Finally, we introduce certain linear operators inspired by the aforementioned $b$-Weingarten calculus that can be considered as $b$-deformations of the Jucys--Murphy elements in the symmetric group algebra. We make several conjectures regarding these operators that generalise known properties of the Jucys--Murphy elements and make a connection to the family of Jack symmetric functions.

\emph{Acknowledgements.} XC was supported by a University of Auckland Doctoral Scholarship. ND was supported by the Australian Research Council grant FT240100795. Both authors thank Brice Arrigo for useful discussions as well as Valentin Bonzom and Guillaume Chapuy for insightful questions that inspired the present work.


\vspace{3mm} \hrule \vspace{2mm}

\tableofcontents

\vspace{3mm} \hrule

\section{Introduction}

\subsection*{Background and context}

The present work is inspired by the following three themes and the rich connections between them.

\begin{itemize}
\item {\em Weingarten calculus} was originally developed to compute integrals of products of matrix elements over unitary groups, with respect to the Haar measure~\cite{col03,wei78}. 
\begin{equation} 
\int_{\mathbf{U}(N)} U_{i(1)j(1)} U_{i(2)j(2)} \cdots U_{i(k)j(k)} U^*_{\overline{i}(1)\overline{j}(1)} U^*_{\overline{i}(2)\overline{j}(2)} \cdots U^*_{\overline{i}(k)\overline{j}(k)} \, \dd \mu
\end{equation}
In recent decades, Weingarten calculus has developed into a rich theory broadly concerned with integration on compact groups and related objects, with respect to the Haar measure~\cite{col-mat-nov22}. Modern accounts of Weingarten calculus often rely on elegant algebraic approaches via Schur--Weyl duality~\cite{col-sni06}. In this paper, we instead follow the approach via orthogonality relations utilised by Collins and Matsumoto~\cite{col-mat17}, which is in turn inspired by the ideas contained in the seminal paper of Weingarten~\cite{wei78}.

\item {\em Monotone Hurwitz numbers} enumerate certain factorisations of permutations into transpositions 
\begin{equation} \label{eq:factorisation}
p = (a_1~b_1) \circ (a_2~b_2) \circ \cdots \circ (a_r~b_r)
\end{equation}
that satisfy the monotonicity constraint $b_1 \leq b_2 \leq \cdots \leq b_r$, where we always write $a_i < b_i$. They were introduced by Goulden, Guay-Paquet and Novak, who recognised their appearance as coefficients in the large~$N$ expansion of the HCIZ matrix integral over the unitary groups~\cite{gou-gua-nov14}. They are now known to appear in various contexts, such as enumeration of maps~\cite{BCDG19}, free probability theory~\cite{BCGLS21}, and topological recursion~\cite{do-dye-mat17}.

\item The {\em Jucys--Murphy elements} $J_1, J_2, \ldots, J_k$ are elements of the symmetric group algebra defined by
\begin{equation} \label{eq:jucysmurphy}
J_i = (1~i) + (2~i) + \cdots + (i-1~i) \in \mathbb{C}[S_k],
\end{equation}
where we interpret the formula for $i = 1$ as $J_1 = 0$. They were introduced independently by Jucys~\cite{juc74} and Murphy~\cite{mur81}, and their seemingly simple definition belies their remarkable properties. For example, they commute with each other and indeed, generate a maximal commutative subalgebra of $\mathbb{C}[S_k]$. Any symmetric function of the Jucys--Murphy elements lies in the centre of the symmetric group algebra $Z\mathbb{C}[S_k]$ and the class expansions of such expressions are of significant interest, appearing in various contexts~\cite{las09}. Furthermore, the Jucys--Murphy elements are essential components of the Okounkov--Vershik approach to the representation theory of symmetric groups~\cite{oko-ver96}.
\end{itemize}

In recent work, the authors and Moskovsky developed a Weingarten calculus for integration over the space of $N \times N$ idempotent Hermitian matrices of rank~$M$~\cite{cou-do-mos23}. This space is naturally homeomorphic to the Grassmannian $\mathrm{Gr}_\mathbb{C}(M,N)$ of $M$-dimensional subspaces of an $N$-dimensional complex vector space. The coefficients of the large $N$ fixed $\frac{M}{N}$ expansion of these integrals gives rise to a deformation of the monotone Hurwitz numbers that promotes them to polynomials in the parameter $t = 1 - \frac{N}{M}$. These are obtained by counting with the weight $t^{|\{b_1, b_2, \ldots, b_r\}|}$ assigned to the factorisation of \cref{eq:factorisation}, where duplicates are removed from the set $\{b_1, b_2, \ldots, b_r\}$. In the same work, it was empirically observed that the polynomials arising from these so-called $t$-monotone Hurwitz numbers are real-rooted and satisfy remarkable interlacing phenomena, which remain conjectural at present~\cite{cou-do-mos23}.

\subsection*{Weingarten calculus for real Grassmannians}

In the present work, our starting point is the natural analogue of the previous work of the authors and Moskovsky~\cite{cou-do-mos23}, in the setting of integration on real Grassmannians. In particular, we consider the space $\mathbf{A}(M,N)$ of $N \times N$ idempotent real symmetric matrices of rank $M$, which is naturally homeomorphic to the Grassmannian $\mathrm{Gr}_{\mathbb{R}}(M,N)$ of $M$-dimensional subspaces of an $N$-dimensional real vector space.

We develop a Weingarten calculus for integration of polynomials in the matrix elements over $\mathbf{A}(M,N)$, with respect to the Haar measure $\dd \mu$. Whereas the original Weingarten calculus for unitary groups involves permutations, symmetric groups, and Jucys--Murphy elements~\cite{col03,nov10,mat-nov13}, our work relies on the Weingarten calculus for orthogonal groups, which analogously involves pair partitions, hyperoctahedral groups, and odd Jucys--Murphy elements~\cite{col-sni06,mat11}. A minimal introduction to all of these notions is provided in \cref{sec:prelim}. For the present purposes, it suffices to define a pair partition of a set to be a partition of its elements into unordered pairs and to note that the symmetric group $S_{2k}$ acts on the set of pair partitions of $\{1, 2, \ldots, 2k\}$, which we denote by $\P_k$.

Define the Weingarten function $\wga$ on a pair partition $\m$ of the set $\{1, 2, \ldots, 2k\}$ via the equation
\begin{equation} \label{eq:wgafunction}
\wga(\m) = \int_{\mathbf{A}(M,N)} A_{i(1)i(2)} A_{i(3)i(4)} \cdots A_{i(2k-1)i(2k)} \, \dd \mu.
\end{equation}
Here, $i: \{1, 2, \ldots, 2k\} \to \{1, 2, \ldots, N\}$ is any function such that $i(a) = i(b)$ if and only if $\{a,b\} \in \m$. We prove that any integral of a polynomial in the matrix elements over $\mathbf{A}(M,N)$ can be expressed as a sum of values of the Weingarten function. Furthermore, we show that these values can be calculated recursively, have large $N$ fixed $\frac{M}{N}$ expansions whose coefficients count monotone factorisations of pair partitions, and are related to the odd Jucys--Murphy elements $J_1, J_3, \ldots, J_{2k-1} \in \mathbb{C}[S_{2k}]$.

\begin{theoremx}[Weingarten calculus for $\mathbf{A}(M,N)$] ~ \\
The Weingarten function $\wga$ defined by \cref{eq:wgafunction} satisfies the following.
\begin{itemize}
\item {\em Convolution formula (\cref{prop:convolution_A})} \\
For any function $i: \{1, 2, \ldots, 2k\} \to \{1, 2, \ldots, N\}$,
\[
\int_{\mathbf{A}(M,N)} A_{i(1) i(2)} A_{i(3) i(4)} \cdots A_{i(2k-1) i(2k)} \, \dd \mu = \sum_{\m \in \P_{k}} \Delta_{\m}(i) \, \wga(\m).
\]
Here, we set $\Delta_\m(i) = 1$ if $i$ is a function such that $\{a,b\} \in \m$ implies $i(a) = i(b)$ and $\Delta_\m(i) = 0$ otherwise.

\item {\em Orthogonality relations (\cref{thm:orthogonality_A})} \\
For each non-empty pair partition $\m \in \P_k$, the Weingarten function $\wga$ satisfies the relation
\begin{align*}
\wga(\m) = &-\frac{1}{N} \sum_{i=1}^{2k-2} \wga((i~2k-1) \cdot \m) + \frac{M}{N} \delta_{\{2k-1,2k\} \in \m} \wga(\m^\downarrow) \\
&+ \frac{1}{N} \sum_{i=1}^{2k-2} \delta_{\{i,2k\}\in\m} \wga([(i~2k-1)\cdot\m]^\downarrow).
\end{align*}
Here, we set $\delta_{\{i,j\} \in \m} = 1$ if $\{i,j\} \in \m$ and $\delta_{\{i,j\} \in \m} = 0$ otherwise. If $\{2k-1,2k\} \in \m$, we define $\m^\downarrow \in \P_{k-1}$ to be the pair partition obtained by removing the pair $\{2k-1, 2k\}$ from $\m$.

\item {\em Large $N$ expansion (\cref{thm:expansion_A})} \\
For each pair partition $\m \in \P_k$, we have the large $N$ fixed $\frac{M}{N}$ expansion 
\[
\wga(\m) = \frac{1}{(1-t)^k} \sum_{r=0}^\infty \vec{h}^{(t)}_r(\m) \left(-\frac{1}{N}\right)^r,
\]
where $t = 1 - \frac{N}{M}$. The coefficient $\vec{h}^{(t)}_r(\m) \in \mathbb{Z}[t]$ denotes the weighted enumeration of monotone factorisations of $\m$ with length $r$, where the weight of a monotone factorisation $\bm{\tau}$ is given by $t^{\mathrm{hive}(\bm{\tau})}$, terminology that is explained in \cref{subsec:expansion_A}.

\item {\em Representation-theoretic interpretation (\cref{prop:jucysmurphy_A})} \\
For each positive integer $k$, we have the following equality in the vector space $\mathbb{C}[\P_k]$, where $\e_k$ denotes the identity pair partition $(1 ~ 2 \mmid 3 ~ 4 \mmid \cdots \mid 2k-1 ~ 2k)$.
\[
\sum_{\m \in \P_k} \wga(\m) \, \m = \prod_{i=1}^k \frac{M + J_{2i-1}}{N + J_{2i-1}} \cdot \e_k
\]
\end{itemize}
\end{theoremx}

Our basis for understanding the Weingarten calculus for $\mathbf{A}(M,N)$ is by forming the Weingarten graph~$\mathcal{G}^{\mathbf{A}}$, whose vertices are pair partitions and whose edges represent the terms of the orthogonality relations above. By construction, the large $N$ fixed $\frac{M}{N}$ expansion of the Weingarten function $\wga$ has coefficients that are weighted enumerations of paths in the Weingarten graph $\mathcal{G}^{\mathbf{A}}$. Equivalently, these coefficients can be expressed as a weighted enumeration of monotone factorisations of pair partitions
\begin{equation} \label{eq:oddfactorisation}
(a_1~b_1) \circ (a_2~b_2) \circ \cdots \circ (a_r~b_r) \cdot \m = \e_k,
\end{equation}
where $b_1 \leq b_2 \leq \cdots \leq b_r$ are odd integers. In the following, we show that enumerating monotone factorisations of permutations (see \cref{eq:factorisation}) and enumerating monotone factorisations of pair partitions (see \cref{eq:oddfactorisation}) can be encompassed by a single weighted enumeration that carries a parameter $b$.

\subsection*{The $bt$-monotone Hurwitz numbers}

Inspired by the $b$-deformation arising from Jack functions, we introduce a $b$-Weingarten calculus that carries a parameter $b$ and is constructed to recover the unitary Weingarten calculus at $b = 0$ and the orthogonal Weingarten calculus at $b = 1$. Although there are no underlying matrix integrals for general values of $b$, we construct a $b$-Weingarten graph and a $b$-Weingarten function that enumerates paths within the graph. The fundamental object underlying these constructions is a weight function that takes as input two pair partitions and returns as output an element of the set $\{0, 1, b\}$. As a consequence, we show that the \mbox{$b$-Weingarten} function has a large $N$ expansion whose coefficients are $b$-monotone Hurwitz numbers, which were first introduced by Bonzom, Chapuy and Do\l{}\k{e}ga via Jack functions~\cite{bon-cha-dol23}. A consequence of our work is a combinatorial interpretation for $b$-monotone Hurwitz numbers as an enumeration of monotone factorisations of pair partitions, weighted by $b$ to the power of a certain ``flip number''.

It is natural to ask whether the $b$-deformation of the previous paragraph is compatible with the $t$-deformation arising from the Weingarten calculus for Grassmannians, defined previously in the complex/unitary case~\cite{cou-do-mos23} and presently in the real/orthogonal case. The two deformations are indeed compatible and this leads to the notion of a $bt$-Weingarten calculus and $bt$-monotone Hurwitz numbers.

Following the definition of $b$-monotone Hurwitz numbers of Bonzom, Chapuy and Do\l\polhk{e}ga~\cite{bon-cha-dol23}, we may define $bt$-monotone Hurwitz numbers via the following equivalence of generating functions. The Jack functions and related constructions that appear in the first line are explained in \cref{subsec:jack}, while the sum over $g$ that appears in the second line is actually over $\frac{1}{2}\mathbb{N} = \{0, \frac{1}{2}, 1, \frac{3}{2}, \ldots\}$.
\begin{align*}
Z^{(bt)}(p_1, p_2, \ldots; \h, z) &= \sum_{k \geq 0} \frac{z^k}{\h^k} \sum_{\lambda \vdash k} \Bigg( \prod_{\Box \in \lambda} \frac{1- (1-t)
\h c_b(\Box)}{1 - \h c_b(\Box)} \Bigg) \, \frac{J^{(b)}_\lambda(p_1, p_2, \ldots)}{\mathrm{hook}_b(\lambda) \, \mathrm{hook}_b'(\lambda)} \\
&= \exp \Bigg[ \sum_{k \geq 0} z^k \sum_{g \geq 0} \sum_{n \geq 1} \frac{\h^{2g-2+n}}{n!} \sum_{\mu_1, \ldots, \mu_n \geq 1} \Hbt_{g,n}(\mu_1, \ldots, \mu_n) \, \frac{p_{\mu_1} \cdots p_{\mu_n}}{(b+1)^n} \Bigg]
\end{align*}

One can ask whether known properties of monotone Hurwitz numbers carry over to this two-parameter deformation.

\begin{theoremx}[$bt$-monotone Hurwitz numbers] ~ \\
The $bt$-monotone Hurwitz numbers satisfy the following.
\begin{itemize}
\item {\em Virasoro constraints} (\cref{thm:virasoro_bt}) \\
The partition function $Z^{(bt)}(p_1, p_2, \ldots; \h, z)$ satisfies $L_m^{(bt)} \cdot Z^{(bt)} = 0$, for $m = 1, 2, 3, \ldots$, where the differential operator $L_m^{(bt)}$ is defined in \cref{thm:virasoro_bt}. Moreover, the operators $L_1^{(bt)}, L_2^{(bt)}, L_3^{(bt)}, \ldots$ satisfy the Virasoro commutation relations $[L_m^{(bt)}, L_n^{(bt)}] = (m-n) \, L_{m+n}^{(bt)}$ for $m, n \geq 1$.

\item {\em Cut-join-flip recursion} (\cref{cor:cjfrecursion}) \\
The $bt$-monotone Hurwitz numbers satisfy an effective recursion appearing in \cref{cor:cjfrecursion}, which generalises the cut-join recursion for usual monotone Hurwitz numbers~\cite{gou-gua-nov13}.
\item {\em Combinatorial interpretation} (\cref{prop:btmonotonecombinatorics}) \\
The $bt$-monotone Hurwitz numbers satisfy
\[
\Hbt_{g,n}(\lambda) = \frac{1}{\prod \lambda_i} \sum_{\substack{\bm{\tau}\in \mathrm{CMono}(\m) \\ \ell(\bm{\tau}) = |\lambda|+2g-2+n}} b^{\mathrm{flip}(\bm{\tau})} t^{\mathrm{hive}(\bm{\tau})},
\]
where $\m$ is any pair partition of coset-type $\lambda$. The summation is over connected monotone factorisations $\bm{\tau}$ of $\m$ with length $|\lambda|+2g-2+n$. Connected monotone factorisations and hive number are defined in \cref{def:factorisation}, while flip is defined in \cref{def:flip}.
\end{itemize}
\end{theoremx}

When considered as polynomials in $t$, the coefficients and roots of the $bt$-monotone Hurwitz numbers possess interesting structure. For example, their coefficients form a sequence of polynomials in $b$ that is symmetric. At the level of roots, it appears that the $bt$-monotone Hurwitz numbers are real-rooted and that they exhibit interlacing phenomena. Two real-rooted polynomials are said to {\em interlace} if their degrees differ by one and their roots weakly alternate on the real number line. We have gathered overwhelming numerical
evidence to support the following conjectures.

\begin{conjecturex} [Real-rootedness and interlacing -- \cref{con:btrealrooted,con:btinterlacing}]
For any value of $b \in \mathbb{R} \setminus \{0, -1\}$, 
\begin{itemize}
\item the $bt$-monotone Hurwitz number $\Hbt_{g,n}(\mu_1, \ldots, \mu_n)$ is a real-rooted polynomial in $t$;

\item the polynomial $\Hbt_{g,n}(\mu_1, \ldots, \mu_n)$ interlaces each of the $n$ polynomials
\[
\Hbt_{g,n}(\mu_1+1, \mu_2, \ldots, \mu_n), \quad \Hbt_{g,n}(\mu_1, \mu_2+1, \ldots, \mu_n), \quad \ldots, \quad \Hbt_{g,n}(\mu_1, \mu_2, \ldots, \mu_n+1).
\]
\end{itemize}
\end{conjecturex}

\cref{tab:variants} summarises the context of the present work in the general landscape of monotone Hurwitz numbers. In each cell, we name the particular variant of the monotone Hurwitz numbers, the corresponding matrix integrals, and the references in which these were introduced. Note that the notion of orthogonal monotone Hurwitz numbers had not been previously introduced explicitly in the literature.

\begin{table}[pht!]
\caption{Variants of the monotone Hurwitz numbers and their corresponding matrix integrals.}
\label{tab:variants}
\begin{tabularx}{\textwidth}{lXXX} \toprule
 & $b = 0$ & $b = 1$ & general $b$ \\ \midrule
$t = 1$ & monotone & ``orthogonal monotone'' & $b$-monotone \\
 & unitary groups & orthogonal groups & {\em no matrix integral} \\
 & Goulden et al.~\cite{gou-gua-nov14} & implicit in other works & Bonzom et al.~\cite{bon-cha-dol23} \\ \midrule
general $t$ & $t$-monotone & ``$t$-orthogonal monotone'' & $bt$-monotone \\
 & complex Grassmannian & real Grassmannian & {\em no matrix integral} \\
 & Coulter et al.~\cite{cou-do-mos23} & [present work, \cref{sec:weingarten-A}] & [present work, \cref{sec:hurwitz}] \\ \bottomrule
\end{tabularx}
\end{table}

\subsection*{A deformation of the Jucys--Murphy elements}

The Jucys--Murphy elements appear naturally in the Weingarten calculus for integration over unitary groups~\cite{nov10,mat-nov13}. In a certain sense, they encapsulate the adjacency structure of the unitary Weingarten graph. By analogy, our construction of the $b$-Weingarten graph motivates the definition of {\em $b$-deformed Jucys--Murphy operators} $\J_1, \J_2, \ldots, \J_k$ that depend on a parameter $b$ and act on the vector space with basis the set of pair partitions of $\{1, 2, \ldots, 2k\}$.

It is natural to ask whether known properties of the Jucys--Murphy elements lift to the setting of these $\J$-operators. This leads to the following series of conjectures, which are discussed in further detail in \cref{sec:jucys-murphy}.

\begin{conjecturex}[$b$-deformed Jucys--Murphy operators -- \cref{con:jucys-murphy,con:bJMsymmetric,con:b-laplace}] \label{con:jmmeta} ~ \\
For $k$ a positive integer, let $\X(k) = \langle \J_1, \J_2, \ldots, \J_k \rangle \cdot \e_k$. That is, $\X(k)$ is the orbit of the identity pair partition under the action of the algebra of $\J$-operators.
\begin{itemize}
\item The $\J$-operators commute when restricted to $\X(k)$.

\item There exists a basis $\{\w_\mathsf{T}: \mathsf{T} \in \Tab(k)\}$ of $\X(k)$, where $\Tab(k)$ denotes the set of standard Young tableaux with $k$ boxes. The $\J$-operators $\J_1, \J_2, \ldots, \J_k$ act diagonally on $\X(k)$, according to the formula $\J_i \cdot \w_\mathsf{T} = c_b(\mathsf{T}_i) \, \w_\mathsf{T}$, where $\mathsf{T}_i$ is the box labelled $i$ in the tableau $\mathsf{T}$ and $c_b$ denotes the $b$-content, defined in \cref{subsec:jack}.

\item For any tableau $\mathsf{S} \in \Tab(j)$ with $j \leq k$, the vector $\w_\mathsf{S} \in \X(j) \subseteq \X(k)$ satisfies $
\w_\mathsf{S} = \displaystyle\sum_{\substack{\mathsf{T} \in \Tab(k) \\ \mathsf{S} \subseteq \mathsf{T}}} \w_\mathsf{T}$.

\item For any tableau $\mathsf{T} \in \Tab(k)$, the vector $\w_\mathsf{T} \in \X(k)$ satisfies the recursive equation
\[
\w_\mathsf{T} = \Bigg(\prod_{\substack{\mathsf{S} \in \Tab(k) \\ \mathsf{S} \neq \mathsf{T}, \overline{\mathsf{S}} = \overline{\mathsf{T}}}} \frac{\J_k - c_b(\mathsf{S}_k)}{c_b(\mathsf{T}_k) - c_b(\mathsf{S}_k)} \Bigg) \cdot \w_{\overline{\mathsf{T}}},
\]
where $\overline{\mathsf{T}}$ denotes the tableau obtained from $\mathsf{T}$ by removing the box labelled $k$.
\end{itemize}

In \cref{eq:pwbases}, we define $b$-deformed analogues of the conjugacy classes $\{C_\lambda: \lambda \vdash k\}$ and orthogonal idempotents $\{\epsilon_\lambda: \lambda \vdash k\}$ in the centre $Z\mathbb{C}[S_k]$ of the symmetric group algebra, respectively denoted $\{\p_\lambda: \lambda \vdash k\}$ and $\{\w_\lambda: \lambda \vdash k\}$.
\begin{itemize}
\item For each partition $\lambda \vdash k$, $\w_\lambda = \displaystyle\sum_{\mathsf{T} \in \Tab(\lambda)} \w_\mathsf{T}$.

\item For any symmetric function $f$ in $k$ variables and any partition $\lambda \vdash k$, $f(\J_1, \J_2, \ldots, \J_k) \cdot \w_\lambda = f(\mathrm{cont}_b(\lambda)) \, \w_\lambda$, where $\mathrm{cont}_b(\lambda)$ is the multiset of $b$-contents of boxes in $\lambda$.
\end{itemize}

In \cref{eq:char_b}, we define a $b$-deformed analogue $\mathrm{ch}^{(b)}$ of the Frobenius characteristic map that sends class functions on $S_k$ to homogeneous symmetric functions of degree $k$.
\begin{itemize}
\item For all $\lambda \vdash k$, $\mathrm{ch}^{(b)} \big( (\J_1 + \J_2 + \cdots + \J_k) \cdot \p_\lambda \big) = D(b) \cdot \mathrm{ch}^{(b)}(\p_\lambda)$, where $D(b)$ is the Laplace--Beltrami operator, defined in \cref{subsec:jack}, whose eigenfunctions are the Jack symmetric functions.
\end{itemize}
\end{conjecturex}

Each part of \cref{con:jmmeta} generalises a known statement about the usual Jucys--Murphy elements in the symmetric group algebra, which is recovered by setting $b = 0$. In that case, one can prove these statements by referring to the representation theory of the symmetric groups. However, the main obstruction to progress on \cref{con:jmmeta} is a lack of such representation-theoretic foundations for general values of $b$.

These conjectures have all been verified computationally for $1 \leq k \leq 5$. For $k = 5$, the vector space on which the $\J$-operators act has dimension 945 while the subspace $\X(5)$ has dimension 26, so the verification in this case is already rather substantial. Furthermore, we prove in \cref{prop:bJucysMurphy} that the penultimate part of \cref{con:jmmeta} holds whenever $f$ is a complete homogeneous symmetric polynomial. We believe that the motivation behind our construction of the $\J$-operators, our computations, and our supporting results together constitute compelling evidence towards \cref{con:jmmeta}.

By virtue of their construction and the putative properties of \cref{con:jmmeta}, the $b$-deformed Jucys--Murphy operators should be closely related to Jack functions. Our work supports the notion that the $b$-deformation obtained by passing from Schur functions to Jack functions is aligned with the $b$-deformation of the Jucys--Murphy elements that we introduce in the present work. Such deformations are currently of great interest, due to the multitude of conjectures concerning Jack functions~\cite{han88,sta89,gou-jac96,ale-hag-wan21} and connections to the more general notion of refinement in mathematical physics, particularly with regards to topological recursion~\cite{CDO24b,osu24}. One can speculate that our conjectures hint at a deeper underlying algebraic theory that may help to unlock some of the unresolved problems concerning Jack functions, $b$-deformations, and refinement.

\subsection*{Structure of the paper}

In \cref{sec:prelim}, we provide a brief introduction to Weingarten calculus for integration over orthogonal groups, beginning with some combinatorial preliminaries concerning the hyperoctahedral groups and pair partitions. We then define the family of Jack symmetric functions and show how they are used by Bonzom, Chapuy and Do\l{}\k{e}ga to construct the $b$-monotone Hurwitz numbers~\cite{bon-cha-dol23}.

In \cref{sec:weingarten-A}, we introduce and develop the Weingarten calculus for integration over the space $\mathbf{A}(M,N)$ of $N \times N$ idempotent real symmetric matrices of rank $M$. This space provides a matrix realisation of the real Grassmannian $\mathrm{Gr}_\mathbb{R}(M,N)$. We define a Weingarten function in this context, and prove a convolution formula and orthogonality relations for it. These orthogonality relations lead to the Weingarten function being expressed in various ways: as a weighted enumeration of walks in an appropriate Weingarten graph, as a weighted enumeration of monotone factorisations of pair partitions, and as a particular expression involving odd Jucys--Murphy elements.

In \cref{sec:hurwitz}, we introduce a $b$-Weingarten calculus that is constructed to recover the unitary Weingarten calculus at $b = 0$ and the orthogonal Weingarten calculus at $b = 1$. The theory relies fundamentally on the construction of a weight function that takes as input two pair partitions and returns as output an element of the set $\{0, 1, b\}$. We show that the $b$-deformation of this section can be combined with the $t$-deformation of the previous section to produce a $bt$-Weingarten calculus and corresponding $bt$-monotone Hurwitz numbers. On the basis of overwhelming numerical evidence, we conjecture that these $bt$-monotone Hurwitz numbers are real-rooted and display remarkable interlacing phenomena when considered as polynomials in $t$, generalising previous conjectures of the authors with Moskovsky~\cite{cou-do-mos23}.

In \cref{sec:jucys-murphy}, we use the construction of the weight function and the $b$-Weingarten graph from the previous section to motivate the definition of the $b$-deformed Jucys--Murphy operators. These act on the vector space with basis the set of pair partitions and naturally generalise the Jucys--Murphy elements in the symmetric group algebra. We state several conjectures concerning these operators that generalise known properties of the Jucys--Murphy elements. These conjectures hint at a deeper underlying algebraic theory that is currently unknown, as well as strong connections to the family of Jack symmetric functions.

\section{Preliminaries} \label{sec:prelim}

In this section, we provide a minimal introduction to orthogonal Weingarten calculus and its algebro-combinatorial foundations, the family of Jack symmetric functions, and the notion of $b$-monotone Hurwitz numbers. The reader already acquainted with any of these concepts is welcome to skip the corresponding subsections.

\subsection{Hyperoctahedral groups and pair partitions} \label{subsec:hyperoctahedral}

Weingarten calculus was originally developed to compute integrals of products of matrix elements and their conjugates over unitary groups, with respect to the Haar measure~\cite{wei78,col03}. The theory is closely related to the representation theory of symmetric groups via Schur--Weyl duality~\cite{col-sni06}. The original notions of Weingarten calculus can be transported to various other settings, perhaps the most natural being integration over orthogonal groups~\cite{col-mat09}. The algebraic underpinnings of orthogonal Weingarten calculus involve the hyperoctahedral groups~\cite{mac15}.

As an abstract group, the {\em hyperoctahedral group} $H_k$ can be defined as the wreath product $S_2 \wr S_k$. It has order $2^k \times k!$ and is naturally the symmetry group of a $k$-dimensional hypercube. We realise $H_k$ as the centraliser of the involution $(1~2) (3~4) \cdots (2k-1 ~ 2k) \in S_{2k}$.

\begin{example}
The centraliser of the involution $(1~2) (3~4) \in S_4$ is
\[
H_2 = \{ (1), \, (1~2), \, (3~4), \, (1~2)(3~4), \, (1~3)(2~4), \, (1~4)(2~3), \, (1~3~2~4), \, (1~4~2~3) \}.
\]
\end{example}

The cosets of $H_k \leq S_{2k}$ are naturally indexed by combinatorial objects known as pair partitions.

\begin{definition}
A {\em pair partition} is a way to express the set $\{1, 2, \ldots, 2k\}$ as a disjoint union of two-element sets, where $k$ is a positive integer. We denote a pair partition by
\[
(\m(1) \, \m(2) \mmid \m(3) \, \m(4) \mmid \cdots \mmid \m(2k-1) \, \m(2k)),
\]
where $\{\m(1), \m(2), \ldots, \m(2k)\} = \{1, 2, \ldots, 2k\}$. Furthermore, we adopt the standard convention that $\m(2i-1) < \m(2i)$ for $1 \leq i \leq k$ and $\m(1) < \m(3) < \cdots < \m(2k-1)$. We write the {\em trivial pair partition} as
\[
\e_k = (1 ~ 2 \mmid 3 ~ 4 \mmid \cdots \mid 2k-1 ~ 2k).
\]
Let $\P_k$ denote the set of pair partitions on the set $\{1, 2, \ldots, 2k\}$ and, for completeness, define $\P_0$ to be the set with only the empty pair partition $(~)$. Finally, let $\P_\bullet$ denote the set $\P_0 \sqcup \P_1 \sqcup \P_2 \sqcup \cdots$ of all pair partitions.
\end{definition}

\begin{example}
The set $\P_2$ comprises the pair partitions of $\{1, 2, 3, 4\}$, which are as follows.
\[
(1\,2 \mmid 3\,4) \qquad \qquad (1 \, 3 \mmid 2 \, 4) \qquad \qquad (1\,4 \mmid 2\,3) 
\]
The set $\P_3$ comprises the pair partitions of $\{1, 2, 3, 4, 5, 6\}$, which are as follows.
\begin{align*}
(1\,2 \mmid 3\,4 \mmid 5\,6) && (1\,3 \mmid 2\,4 \mmid 5\,6) && (1\,4 \mmid 2\,3 \mmid 5\,6) && (1\,5 \mmid 2\,3 \mmid 4\,6) && (1\,6 \mmid 2\,3 \mmid 4\,5) \\
(1\,2 \mmid 3\,5 \mmid 4\,6) && (1\,3 \mmid 2\,5 \mmid 4\,6) && (1\,4 \mmid 2\,5 \mmid 3\,6) && (1\,5 \mmid 2\,4 \mmid 3\,6) && (1\,6 \mmid 2\,4 \mmid 3\,5) \\
(1\,2 \mmid 3\,6 \mmid 4\,5) && (1\,3 \mmid 2\,6 \mmid 4\,5) && (1\,4 \mmid 2\,6 \mmid 3\,5) && (1\,5 \mmid 2\,6 \mmid 3\,4) && (1\,6 \mmid 2\,5 \mmid 3\,4)
\end{align*}
\end{example}

The symmetric group $S_{2k}$ naturally acts on the set of pair partitions $\P_k$ by permuting the elements of the pairs --- that is, the pair $\{a, b\}$ appears in the pair partition $\m$ if and only if the pair $\{\sigma(a), \sigma(b)\}$ appears in the pair partition $\sigma \cdot \m$.

The left cosets of $H_k \leq S_{2k}$ are naturally indexed by pair partitions of $\{1, 2, \ldots, 2k\}$. Under this correspondence, a pair partition $\m = (\m(1)\,\m(2) \mmid \m(3)\,\m(4) \mmid \cdots \mmid \m(2k-1)\,\m(2k))$ corresponds to the left coset $\m H_k$, where we abuse notation by using $\m$ to also denote the permutation that sends $i \mapsto \m(i)$ for $i = 1, 2, \ldots, 2k$.

\begin{example}
The left cosets of $H_2 \leq S_4$ are as follows, corresponding respectively to $\m = (1\,2 \mmid 3\,4) = (1)$, $\m = (1 \, 3 \mmid 2 \, 4) = (2~3)$ and $\m = (1\,4 \mmid 2\,3) = (2~4~3)$. Here and throughout, we compose permutations from right to left.
\begin{align*}
(1) \circ H_2 &= \{ (1), (1~2), (3~4), (1~2)(3~4), (1~3)(2~4), (1~4)(2~3), (1~3~2~4), (1~4~2~3) \} \\
(2~3) \circ H_2 &= \{ (2~3), (1~3~2), (2~3~4), (1~3~4~2), (1~2~4~3), (1~4), (1~2~4), (1~4~3) \} \\
(2~4~3) \circ H_2 &= \{ (2~4~3), (1~4~3~2), (2~4), (1~4~2), (1~2~3), (1~3~4), (1~2~3~4), (1~3) \}
\end{align*}
\end{example}

We associate to each permutation $\sigma \in S_{2k}$ a partition of $k$ in the following way. Consider the graph on the vertex set $\{1, 2, \ldots, 2k\}$, with edges joining $2i-1$ to $2i$ and $\sigma(2i-1)$ to $\sigma(2i)$ for $i = 1, 2, \ldots, k$. This graph necessarily comprises disjoint cycles of even lengths that sum to $2k$. The corresponding partition is simply the one that records these lengths divided by two. This construction extends to pair partitions by interpreting a pair partition as a permutation in the sense described above. We refer to the resulting partition as the {\em coset-type} of the permutation or pair partition.

Let $\lambda(\m)$ and $\Gamma(\m)$ respectively denote the coset-type and the graph corresponding to a pair partition $\m$, as described in the previous paragraph. We use the standard notations $|\lambda|$ to denote the sum of the parts of $\lambda$ and $\lambda \vdash k$ to denote the fact that $\lambda$ is a partition of the non-negative integer $k$. For $\lambda \vdash k$, we denote by $H_\lambda$ the set of permutations of coset-type $\lambda$ in $S_{2k}$.

\begin{proposition}
The double cosets of $H_k \leq S_{2k}$ are given by the sets $H_\lambda$, for $\lambda \vdash k$. The left coset $\m H_k$ is a subset of the double coset $H_\lambda$ if and only if $\lambda(\m) = \lambda$.
\end{proposition}

\begin{example}
The double cosets of $H_2$ in $S_4$ are as follows.
\begin{align*}
H_{(1,1)} = \{ &(1), (1~2), (3~4), (1~2)(3~4), (1~3)(2~4), (1~4)(2~3), (1~3~2~4), (1~4~2~3) \} \\
H_{(2)} = \{ &(2~3), (1~3~2), (2~3~4), (1~3~4~2), (1~2~4~3), (1~4), (1~2~4), (1~4~3), \\
& (2~4~3), (1~4~3~2), (2~4), (1~4~2), (1~2~3), (1~3~4), (1~2~3~4), (1~3) \}
\end{align*}
\end{example}

The pair $(S_{2k}, H_k)$ is a Gelfand pair, meaning that the permutation representation of $S_{2k}$ on $S_{2k} / H_k$ is multiplicity-free~\cite[Section~VII.2]{mac15}. In general, a Gelfand pair $(G, K)$ has an associated Hecke algebra comprising the functions $f: G \to \mathbb{C}$ that are constant on double cosets of $K$, under the convolution product. In the case of the Gelfand pair $(S_{2k}, H_k)$, a natural basis for the Hecke algebra is the {\em double coset basis} $\{\phi_\lambda \mmid \lambda \vdash k\}$, where $\phi_\lambda: S_{2k} \to \mathbb{C}$ is the indicator function for the double coset $H_\lambda$. Another important basis is given by the {\em zonal spherical functions} $\{\omega^\lambda \mmid \lambda \vdash k\}$, where $\omega^\lambda: S_{2k} \to \mathbb{C}$ is defined by
\[
\omega^\lambda(\sigma) = \frac{1}{2^k \, k!} \sum_{\rho \in H_k} \chi^{2\lambda}(\sigma \rho),
\]
where $\chi^{2\lambda}$ is the irreducible character of $S_{2k}$ corresponding to the partition $2\lambda$, in which each part of $\lambda$ is multiplied by two.

\subsection{Weingarten calculus for orthogonal groups} \label{subsec:weingarten_O}

Weingarten calculus for orthogonal groups was originally developed using the Schur--Weyl duality between orthogonal groups and Brauer algebras~\cite{col-sni06}. The theory has been further refined and expanded in subsequent works~\cite{col-mat09,col-mat17,mat11}. Orthogonal Weingarten calculus is concerned with integrals of the form
\[
\int_{\mathbf{O}(N)} O_{i(1)j(1)} O_{i(2)j(2)} \cdots O_{i(2k)j(2k)} \, \dd \nu,
\]
for functions $i, j: \{1, 2, \ldots, 2k\} \to \{1, 2, \ldots, N\}$. Here, $O_{ij}$ denotes a matrix element and $\dd \nu$ denotes the normalised Haar measure on the orthogonal group $\mathbf{O}(N)$. Observe that such an integral with a product of an odd number of terms in the integrand is necessarily zero, by exploiting the symmetry $A \mapsto -A$.

Let $\m \in \P_k$ be a pair partition. We say that a function $i: \{1, 2, \ldots, 2k\} \to \{1, 2, \ldots, N\}$ is
\begin{itemize}
\item {\em admissible} for $\m$ if $\{a, b\} \in \m \Rightarrow i(a) = i(b)$ for all $a \neq b$; and
\item {\em strongly admissible} for $\m$ if $\{a, b\} \in \m \Leftrightarrow i(a) = i(b)$ for all $a \neq b$.
\end{itemize}

\begin{definition} \label{def:weingartenfn}
Define the {\em orthogonal Weingarten function} $\wgo: \P_\bullet \to \mathbb{R}(N)$ by
\[
\wgo(\m) = \int_{\mathbf{O}(N)} O_{1,i(1)} O_{1,i(2)} O_{2,i(3)} O_{2,i(4)} \cdots O_{k,i(2k-1)} O_{k,i(2k)} \, \dd \nu,
\]
where $i:\{1, 2, \ldots, 2k\} \to \{1, 2, \ldots, N\}$ is a strongly admissible function for $\m$. We extend the orthogonal Weingarten function to $\wgo: S_0 \sqcup S_2 \sqcup S_4 \sqcup \cdots \to \mathbb{R}(N)$ by considering a permutation $\sigma \in S_{2k}$ as the pair partition $(\sigma(1) \, \sigma(2) \mmid \sigma(3) \, \sigma(4) \mmid \cdots \mmid \sigma(2k-1) \, \sigma(2k))$.
\end{definition}

Admittedly, \cref{def:weingartenfn} is not particularly satisfactory, since one does not know a priori that the integral is independent of the choice of $\m$ or $\sigma$, nor that the result is a rational function of $N$. However, these are both indeed true, so the orthogonal Weingarten function is well-defined.

The following result asserts that integrals over $\O(N)$ of arbitrary monomials in the matrix elements are explicit sums of values of the orthogonal Weingarten function.

\begin{theorem}[Convolution formula for $\mathbf{O}(N)$~\cite{col-sni06}] \label{thm:convolution_O}
For functions $i, j: \{1, 2, \ldots, 2k\} \to \{1, 2, \ldots, N\}$,
\[
\int_{\mathbf{O}(N)} O_{i(1)j(1)} O_{i(2)j(2)} \cdots O_{i(2k)j(2k)} \, \dd \nu = \sum_{\m, \n \in \P_k} \Delta_\m(i) \, \Delta_\n(j) \, \wgo(\m^{-1}\n).
\]
Here, we set $\Delta_\m(i) = 1$ if $i$ is admissible for $\m$ and $\Delta_\m(i) = 0$ otherwise. The expression $\m^{-1} \n$ is understood by interpreting $\m$ and $\n$ as permutations in $S_{2k}$.
\end{theorem}

The values of the orthogonal Weingarten function are governed by the following key result.

\begin{theorem}[Orthogonality relations for $\wgo$~{\cite[Lemma~4.3]{col-mat17}}] \label{thm:orthogonality_O}
For each non-empty pair partition $\m \in \P_k$, the orthogonal Weingarten function satisfies the relation
\begin{align} \label{eq:orthogonality_O}
\wgo(\m) = - \frac{1}{N} \sum_{i=1}^{2k-2} \wgo((i~2k-1) \cdot \m) + \delta_{\{2k-1,2k\} \in \m} \frac{1}{N} \wgo(\m^\downarrow).
\end{align}
Here, we set $\delta_{\{i,j\} \in \m} = 1$ if $\{i,j\} \in \m$ and $\delta_{\{i,j\} \in \m} = 0$ otherwise. If $\{2k-1,2k\} \in \m$, we define $\m^\downarrow \in \P_{k-1}$ to be the pair partition obtained by removing the pair $\{2k-1, 2k\}$ from $\m$.
\end{theorem}

The orthogonality relations are linear equations that uniquely define the orthogonal Weingarten function from the base case $\wgo((~)) = 1$. Their repeated application leads to an interpretation of the orthogonal Weingarten function as a weighted enumeration of walks in a particular graph~\cite{col-mat17}. The two terms appearing on the right side of \cref{eq:orthogonality_O} lead to two different types of edges, which receive different weights, as per the following definition.

\begin{definition}
The {\em orthogonal Weingarten graph} $\mathcal{G}^{\mathbf{O}}$ is the infinite directed graph with vertex set $\P_\bullet$ and edge set $E_A \sqcup E_B$, where
\begin{itemize}
\item the set $E_A$ comprises ``type $A$'' edges, which are of the form $\m \longrightarrow (i~2k-1) \cdot \m$ for $\m \in \P_k$ and $1 \leq i \leq 2k-2$;
\item the set $E_B$ comprises ``type $B$'' edges, which are of the form $\m \longrightarrow \m^\downarrow$ for $\m \in \P_k$ with $\{2k-1, 2k\} \in \m$.
\end{itemize}

Consistent with usual graph-theoretic terminology, we define a {\em path} in $\mathcal{G}^{\mathbf{O}}$ to be a tuple of pair partitions 
$(\m_0, \m_1, \m_2, \ldots, \m_\ell)$ such that $\m_{i-1} \longrightarrow \m_i$ is a directed edge of $\mathcal{G}^{\mathbf{O}}$ for $i = 1, 2, \ldots, \ell$.

Let $\mathrm{Path}^{\mathbf{O}}(\m)$ denote the set of paths from the pair partition $\m$ to the empty pair partition $(~)$ in the Weingarten graph $\mathcal{G}^\mathbf{O}$. 
\end{definition}

\begin{figure}[ht!]
\centering
\tikzstyle{W} = [color=white, line width=4pt]
\tikzstyle{vBlue} = [draw=white, text=blue, very thick]
\tikzstyle{vRed} = [draw=white, text=red, very thick]
\tikzstyle{eBlack} = [color=black, thick]
\tikzstyle{eBlackDash} = [color=black, thick, dashed]
\begin{tikzpicture}[scale=0.7]
\def\x{2}
\def\y{2}
\def\l{0.5}

\node[vRed] (empty) at (0*\x,-4*\y) { $(~)$};

\node[vRed] (12) at (0*\x,-3*\y) {$(12)$};

\node[vRed] (1234) at (0*\x,-2*\y) {$(12 \mmid 34)$};
\node[vBlue] (1423) at (-1*\x,-1*\y) {$(14 \mmid 23)$};
\node[vBlue] (1324) at (1*\x,-1*\y) {$(13 \mmid 24)$};

\node[vRed] (123456) at (0*\x,2*\y) {$(12 \mmid 34 \mmid 56)$};
\node[vBlue] (123546) at (-4*\x,4*\y) {$(12 \mmid 35 \mmid 46)$};
\node[vBlue] (152634) at (4*\x,0*\y) {$(15 \mmid 26 \mmid 34)$};
\node[vBlue] (162534) at (4*\x,4*\y) {$(16 \mmid 25 \mmid 34)$};
\node[vBlue] (123645) at (-4*\x,0*\y) {$(12 \mmid 36 \mmid 45)$};
\node[vRed] (142356) at (-1*\x,3*\y) {$(14 \mmid 23 \mmid 56)$};
\node[vBlue] (152346) at (-5*\x,5*\y) {$(15 \mmid 23 \mmid 46)$};
\node[vBlue] (142635) at (3*\x,1*\y) {$(14 \mmid 26 \mmid 35)$};
\node[vBlue] (162345) at (3*\x,5*\y) {$(16 \mmid 23 \mmid 45)$};
\node[vBlue] (142536) at (-5*\x,1*\y) {$(14 \mmid 25 \mmid 36)$};
\node[vRed] (132456) at (1*\x,3*\y) {$(13 \mmid 24 \mmid 56)$};
\node[vBlue] (132546) at (-3*\x,5*\y) {$(13 \mmid 25 \mmid 46)$};
\node[vBlue] (132645) at (5*\x,1*\y) {$(13 \mmid 26 \mmid 45)$}; 
\node[vBlue] (162435) at (5*\x,5*\y) {$(16 \mmid 24 \mmid 35)$};
\node[vBlue] (152436) at (-3*\x,1*\y) {$(15 \mmid 24 \mmid 36)$};

\draw (1423) edge [eBlack] (1234);
\draw (1234) edge [eBlack] (1324);

\draw (123546) edge [eBlack] (152346);
\draw (123546) edge [eBlack] (132546);
\draw (132546) edge [eBlack] (152346);
\draw (162534) edge [eBlack] (162345);
\draw (162534) edge [eBlack] (162435);
\draw (162345) edge [eBlack] (162435);
\draw (123645) edge [eBlack] (142536);
\draw (123645) edge [eBlack] (152436);
\draw (142536) edge [eBlack] (152436);
\draw (152634) edge [eBlack] (142635);
\draw (152634) edge [eBlack] (132645);
\draw (132645) edge [eBlack] (142635);

\draw[W] (123456) edge (123546); \draw (123456) edge [eBlack] (123546);
\draw[W] (123456) edge (152634); \draw (123456) edge [eBlack] (152634);
\draw[W] (123456) edge (123645); \draw (123456) edge [eBlack] (123645);
\draw[W] (123456) edge (162534); \draw (123456) edge [eBlack] (162534);

\draw[W] (142356) edge (152346); \draw (142356) edge [eBlack] (152346);
\draw[W] (142356) edge (142635); \draw (142356) edge [eBlack] (142635);
\draw[W] (142356) edge (162345); \draw (142356) edge [eBlack] (162345);
\draw[W] (142356) edge (142536); \draw (142356) edge [eBlack] (142536);

\draw[W] (132456) edge (132546); \draw (132456) edge [eBlack] (132546);
\draw[W] (132456) edge (132645); \draw (132456) edge [eBlack] (132645);
\draw[W] (132456) edge (162435); \draw (132456) edge [eBlack] (162435);
\draw[W] (132456) edge (152436); \draw (132456) edge [eBlack] (152436);

\draw[eBlackDash] (12) edge (empty);

\draw[eBlackDash] (1234) edge (12);

\draw[eBlackDash] (123456) edge (1234);
\draw[eBlackDash] (142356) edge (1423);
\draw[eBlackDash] (132456) edge (1324);
\end{tikzpicture}
\caption{The orthogonal Weingarten graph $\mathcal{G}^{\O}$, restricted to $\P_0 \sqcup \P_1 \sqcup \P_2 \sqcup \P_3$. Each solid edge represents two type $A$ edges, one in each direction. Each dashed edge represents one type $B$ edge, directed down the page. Each blue vertex should also have a directed loop, although these are not depicted to avoid cluttering the diagram.}
\label{fig:weingartengraph}
\end{figure}
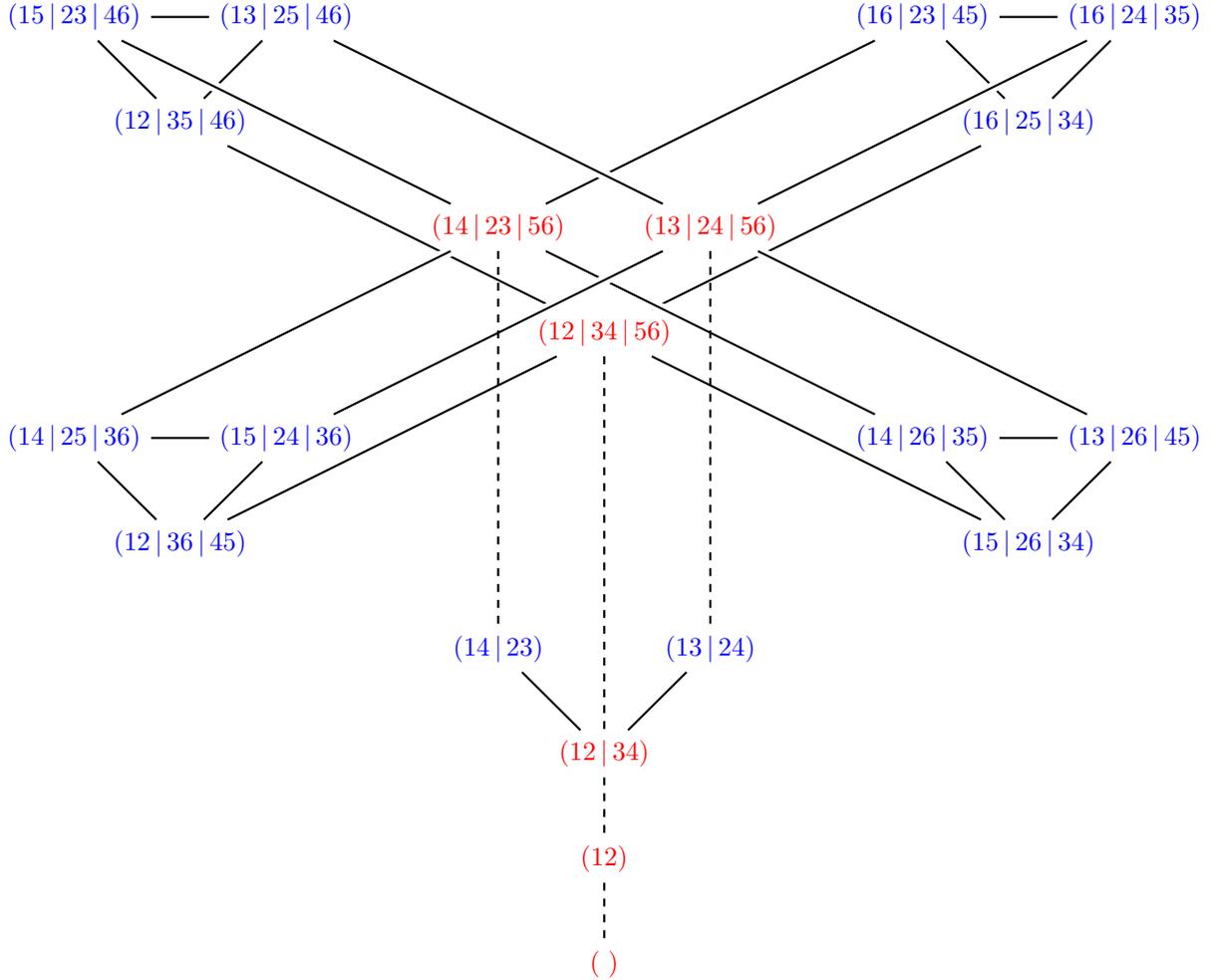

By virtue of the construction of the orthogonal Weingarten graph from the orthogonality relations, one obtains the following.

\begin{proposition}[Large $N$ expansion]
For each pair partition $\m$, we have the large $N$ expansion
\[
\wgo(\m) = \sum_{\bm{\rho} \in \mathrm{Path}^{\mathbf{O}}(\m)} \left( -\frac{1}{N} \right)^{\ell_A(\bm{\rho})} \left( \frac{1}{N} \right)^{\ell_B(\bm{\rho})},
\]
where the summation is over all paths in $\mathcal{G}^{\O}$ from $\m$ to the empty pair partition $(~)$. We use $\ell_A(\bm{\rho})$ to denote the number of type $A$ edges and $\ell_B(\bm{\rho})$ to denote the number of type $B$ edges on the path $\bm{\rho}$.
\end{proposition}

The orthogonality relations of \cref{thm:orthogonality_O} may be recast in a succinct and purely algebraic manner, using the Jucys--Murphy elements in the symmetric group algebra. Such a relation between Weingarten functions and Jucys--Murphy elements was initially realised by Matsumoto and Novak in the unitary setting~\cite{mat-nov13,nov10}.

\begin{proposition}[{\cite[Theorem~7.3]{mat11}}] \label{prop:ojucysmurphy}
For each positive integer $k$, we have the following equality in the vector space $\mathbb{C}[\P_k]$.
\[
\sum_{\m \in \P_k} \wgo(\m) \, \m = \prod_{i=1}^{k} \frac{1}{N + J_{2i-1}} \cdot \e_{k}
\]
\end{proposition}

\begin{remark} \label{rem:symplectic}
One can also consider the Weingarten calculus for symplectic groups~\cite{col-sni06}. However, it is known that the symplectic Weingarten function satisfies $\mathrm{Wg}^{\mathbf{Sp}}(\m) = \pm \left. \wgo(\m) \right|_{N \mapsto -2N}$, so it does not store information beyond what is already contained in the orthogonal Weingarten function~\cite{mat13,col-mat17}.
\end{remark}

\subsection{Jack functions} \label{subsec:jack}

The Jack functions form a one-parameter family of symmetric functions that generalises the Schur functions and the zonal polynomials~\cite{mac15}. In various problems that involve Schur functions, they provide a natural deformation that retains many of the key features of the original problem. Perhaps surprisingly, there remain many unresolved conjectures concerning Jack functions, such as those of Hanlon~\cite{han88}, Stanley~\cite{sta89}, Goulden--Jackson~\cite{gou-jac96}, and Alexandersson--Haglund--Wang~\cite{ale-hag-wan21}.

To define the family of Jack functions, we introduce some combinatorial constructions pertaining to Young diagrams. For a box $\Box$ in row $i$ and column $j$ of a Young diagram, define
\begin{itemize}
\item the {\em $b$-content} to be $c_b(\Box) = (b+1) (j-1) - (i-1)$,
\item the {\em arm length} $a(\Box)$ to be the number of boxes to the right of $\Box$ in the same row,
\item the {\em leg length} $\ell(\Box)$ to be the number of boxes below $\Box$ in the same column,
\item the {\em $b$-hook length} to be $h_b(\Box) = (b+1) a(\Box) + \ell(\Box) + 1$, and
\item the {\em $b$-hook$'$ length} to be $h_b'(\Box) = (b+1) a(\Box) + \ell(\Box) + b + 1$.
\end{itemize}

For a partition $\lambda$, which we equate with its Young diagram, let $\mathrm{cont}_b(\lambda)$ denote the multiset of $b$-contents of boxes in $\lambda$. Furthermore, define the hook products
\[
\mathrm{hook}_b(\lambda) = \prod_{\Box \in \lambda} h_b(\lambda) \qquad \text{and} \qquad \mathrm{hook}_b'(\lambda) = \prod_{\Box \in \lambda} h_b'(\lambda).
\]

\begin{figure} [ht!]
\centering
\begin{small}
\ytableausetup{mathmode, boxframe=0.3mm, boxsize=12mm}
\begin{ytableau}
0 & b+1 & 2b+2 & 3b+3 & 4b+4 \\
-1 & b & 2b+1 & 3b+2 \\
-2 & b-1 & 2b & 3b+1 \\
-3 & b-2
\end{ytableau}
\qquad \qquad \qquad
\begin{ytableau}
4b+8 & 3b+7 & 2b+5 & b+4 & 1 \\
3b+6 & 2b+5 & b+3 & 2 \\
3b+5 & 2b+4 & b+2 & 1 \\
b+2 & 1
\end{ytableau}
\end{small}
\caption{The $b$-contents (left) and the $b$-hook lengths (right) of the partition $(5, 4, 4, 2)$.}
\end{figure}

Note that at $b = 0$, one recovers the usual notions of content, hook length, and hook product, which appears in the hook length formula for the dimension of the Specht module labelled by $\lambda$.
\[
\dim \lambda = \frac{|\lambda|!}{\mathrm{hook}_0(\lambda)}
\]

\begin{definition} \label{def:jack}
Write $\bm{p}$ to denote the infinite sequence of commuting variables $p_1, p_2, p_3, \ldots$ and write $\partial_i$ as a shorthand for the differential operator $\frac{\partial}{\partial p_i}$, where $i \geq 1$. For each partition $\lambda$, the {\em Jack function} $J^{(b)}_\lambda(\bm{p}) \in \mathbb{C}(b)[\bm{p}]$ is the unique eigenvector of the {\em Laplace--Beltrami operator}
\[
D(b) = \frac{1}{2} \bigg[ (b+1) \sum_{i,j\geq 1} i j p_{i+j} \partial_i \partial_j + \sum_{i,j\geq 1} (i+j) p_i p_j \partial_{i+j} + b \sum_{i\geq 1} i (i-1) p_i \partial_i \bigg],
\]
satisfying the following conditions.
\begin{itemize}
\item The eigenvalue corresponding to $J^{(b)}_\lambda(\bm{p})$ is given by the equation
\[]
D(b) \cdot J^{(b)}_\lambda(\bm{p}) = \bigg( \sum_{\Box \in \lambda} c_b(\Box) \bigg) J^{(b)}_\lambda(\bm{p}).
\]
\item The Jack function $J^{(b)}_\lambda(\bm{p})$ is normalised so that
\[
J^{(b)}_\lambda(\bm{p}) = \mathrm{hook}_b(\lambda) \, m_\lambda(\bm{p}) + \sum_{\mu \prec \lambda} a_\lambda^\mu \, m_\mu(\bm{p}),
\]
for some $a_\lambda^\mu\in \mathbb{C}[b]$, where $m_\mu(\bm{p})$ is the monomial symmetric function labelled by $\mu$ expressed in terms of the power-sum symmetric functions $p_1, p_2, p_3, \ldots$ and $\prec$ denotes the dominance order. 
\end{itemize}
\end{definition}

By setting $p_k = x_1^k + x_2^k + x_3^k + \cdots$ to be power-sum symmetric functions in commuting variables $x_1, x_2, x_3, \ldots$, the Jack functions can be considered to be a family of symmetric functions. Indeed, they form a vector space basis for the algebra of symmetric functions.

We will commonly omit the arguments $\bm{p} = (p_1, p_2, p_3, \ldots)$ from the Jack function when it is reasonable to do so. There are various competing normalisations for the Jack functions and the one from the definition above is known as the $J$-normalisation. The Jack functions were originally introduced using the parameter $\alpha = b+1$, but an increasing volume of literature suggests that deep combinatorial significance can be ascribed to Jack functions when expressed in terms of $b$~\cite{jac70, gou-jac96, las09}. The reader should beware that our notation for Jack functions in this paper uses $J_\lambda^{(b)}$ to denote the Jack function expressed in terms of $b$, which may appear at odds with the classical literature. In particular, if $\mathbf{J}_\lambda^{(\alpha)}$ denotes the Jack function according to Jack~\cite{jac70}, then $\mathbf{J}_\lambda^{(\alpha)} = \mathbf{J}_\lambda^{(b+1)} = J_\lambda^{(b)}$. We hope that this clarifies, rather than creates, any potential confusion.

\begin{example} \label{ex:jackfunctions}
The first several Jack functions are as follows.
\begin{align*}
J_\emptyset^{(b)} &=1 & J_1^{(b)} &= p_1 & J_{2}^{(b)} &= (b+1) \, p_{2} + p_1^2 & J_{3}^{(b)} &= 2(b+1)^2 \, p_{3} + 3(b+1) \, p_2 p_1 + p_1^3 \\
 & & & & J_{11}^{(b)} &= - p_2 + p_1^2 & J_{21}^{(b)} &= - (b+1) \, p_3 + b \, p_2 p_1 + p_1^ 3 \\
 & & & & & & J_{111}^{(b)} &= 2 \, p_3 -3 \, p_2 p_1 + p_1^3
\end{align*}
\end{example}

Jack functions can alternatively be characterised using the fact that they provide an orthogonal basis for the ring of symmetric functions with respect to the inner product defined by
\begin{equation} \label{eq:innerproduct_b}
\langle p_\lambda, p_\mu \rangle_b = \delta_{\lambda, \mu} (b+1)^{\ell(\lambda)} z_\lambda,
\end{equation}
where $\ell(\lambda)$ denotes the number of parts of $\lambda$ and $z_\lambda = \prod_{i \geq 1} i^{m_i(\lambda)} \, m_i(\lambda)!$ for $\lambda = 1^{m_1(\lambda)} 2^{m_2(\lambda)} 3^{m_3(\lambda)} \cdots$.
With respect to this inner product, we have
\[
\langle J_\lambda, J_\mu \rangle_b = \delta_{\lambda, \mu} \, \mathrm{hook}_b(\lambda) \, \mathrm{hook}_b'(\lambda).
\]

Jack functions can be recovered from Macdonald polynomials under a particular limit, while specialising the Jack functions to $b = 0$ ($\alpha = 1$) and $b = 1$ ($\alpha = 2$) recovers the Schur functions and the zonal polynomials, via the equations
\[
J^{(0)}_\lambda(\bm{p}) = \frac{|\lambda|!}{ \dim \lambda} \, s_\lambda(\bm{p}) \qquad \text{and} \qquad J^{(1)}_\lambda(\bm{p}) = \frac{(2|\lambda|)!}{ \dim 2 \lambda} \, Z_\lambda(\bm{p}).
\]

\subsection{The \texorpdfstring{$b$-}{b-}monotone Hurwitz numbers} \label{subsec:bmonotone}

Monotone Hurwitz numbers enumerate factorisations of permutations into transpositions that satisfy a certain monotonicity constraint, as per the definition below. They were first defined by Goulden, Guay-Paquet and Novak, who proved that they arise as coefficients in the large $N$ asymptotic expansion of the Harish-Chandra--Itzykson--Zuber (HCIZ) matrix integral over the unitary group $\mathbf{U}(N)$~\cite{gou-gua-nov14}. It was subsequently shown that monotone Hurwitz numbers are governed by the Chekhov--Eynard--Orantin topological recursion~\cite{che-eyn06,eyn-ora07,do-dye-mat17}.

\begin{definition} \label{def:monotonehurwitz}
The {\em monotone Hurwitz number} $\vec{H}_{g,n}(\mu_1, \ldots, \mu_n)$ is equal to $\frac{1}{|\mu|!}$ multiplied by the number of tuples $(\tau_1, \tau_2, \ldots, \tau_r)$ of transpositions in the symmetric group $S_{|\mu|}$ such that
\begin{itemize}
\item $r = |\mu| + 2g-2+n$,
\item the cycles of $\tau_1 \circ \tau_2 \circ \cdots \circ \tau_m$ are labelled $1, 2, \ldots, n$ such that cycle $i$ has length $\mu_i$ for $i = 1, 2, \ldots, n$,
\item if we write each transposition $\tau_i = (a_i ~ b_i)$ with $a_i < b_i$, then $b_1 \leq b_2 \leq \cdots \leq b_m$, and
\item $\tau_1, \tau_2, \ldots, \tau_r$ generate a transitive subgroup of $S_{|\mu|}$.
\end{itemize}
\end{definition}

Monotone Hurwitz numbers can be alternatively defined via the following equality of generating functions~\cite{gua-har15}.
\begin{align} \label{eq:monotone}
Z(\bm{p}; \h, z) &= \exp \Bigg[ \sum_{k \geq 0} z^k \sum_{g \geq 0} \sum_{n \geq 1} \frac{\h^{2g-2+n}}{n!} \sum_{\mu_1, \ldots, \mu_n \geq 1} \vec{H}_{g,n}(\mu_1, \ldots, \mu_n) \, p_{\mu_1} \cdots p_{\mu_n} \Bigg] \notag \\
&= \sum_{k \geq 0} \frac{z^k}{\h^k} \sum_{\lambda \vdash k} \Bigg( \prod_{\Box \in \lambda} \frac{1}{1-\h c_0(\Box)} \Bigg) \, \frac{s_\lambda(\bm{p})}{\mathrm{hook}_0(\lambda)}.
\end{align}

\begin{remark}
The input $(\mu_1, \ldots, \mu_n)$ to the monotone Hurwitz number $\vec{H}_{g,n}(\mu_1, \ldots, \mu_n)$ can be taken to be either a partition or a composition of an integer. We typically adopt the latter convention, which requires us to declare that $\vec{H}_{g,n}(\mu_1, \ldots, \mu_n)$ is a symmetric function of its arguments. These considerations apply here and throughout, for all of the variations of the monotone Hurwitz numbers that we encounter.
\end{remark}

Bonzom, Chapuy and Do\l{}\k{e}ga introduce a $b$-deformed version of this partition function by upgrading Schur functions to Jack functions, contents to $b$-contents, and hook products to $b$-hook products in the following way~\cite{bon-cha-dol23}.
\begin{align} \label{eq:bmonotone}
Z^{(b)}(\bm{p}; \h, z) &= \exp \Bigg[ \sum_{k \geq 0} z^k \sum_{g \geq 0} \sum_{n \geq 1} \frac{\h^{2g-2+n}}{n!} \sum_{\mu_1, \ldots, \mu_n \geq 1} \Hb_{g,n}(\mu_1, \ldots, \mu_n) \, \frac{p_{\mu_1} \cdots p_{\mu_n}}{(b+1)^n} \Bigg] \notag \\
&= \sum_{k \geq 0} \frac{z^k}{\h^k} \sum_{\lambda \vdash k} \Bigg( \prod_{\Box \in \lambda} \frac{1}{1-\h c_b(\Box)} \Bigg) \, \frac{J^{(b)}_\lambda(\bm{p})}{\mathrm{hook}_b(\lambda) \, \mathrm{hook}_b'(\lambda)}.
\end{align}

They proved that the resulting $b$-monotone Hurwitz numbers $\Hb_{g,n}(\mu_1, \ldots, \mu_n)$ enumerate possibly non-orientable monotone Hurwitz maps and satisfy Virasoro constraints. More recently, it was shown that the $b$-monotone Hurwitz numbers satisfy the refined topological recursion~\cite{CDO24b}.

\begin{remark}
There is a significant difference between \cref{eq:monotone,eq:bmonotone}. The former has a summation over $g \in \mathbb{N} = \{0, 1, 2, \ldots\}$, while the latter requires the summation to be over $g \in \frac{1}{2} \mathbb{N} = \{0, \frac{1}{2}, 1, \frac{3}{2}, \ldots\}$. It is useful to consider all summations to be over $g \in \frac{1}{2} \mathbb{N}$, so that $\vec{H}_{g,n}(\mu_1, \ldots, \mu_n) = 0$ whenever $g$ is non-integral. One could argue that it is more natural for the summations to be over an integer parameter equal to $2g$ instead. However, our choice of notation here reflects one that is common in the literature, where $g$ represents the genus of a surface and a non-integral genus indicates that the surface is non-orientable~\cite{bon-cha-dol23}.
\end{remark}

\begin{theorem}[Virasoro constraints for $Z^{(b)}$ {\cite[Theorem~2.7]{bon-cha-dol23}}] \label{thm:virasoro_b}
The partition function $Z^{(b)}(\bm{p}; \h, z)$ satisfies $L_m^{(b)} \cdot Z^{(b)} = 0$, for $m = 1, 2, 3, \ldots$, where
\[
L_m^{(b)} = \frac{m \partial_m}{\h} - (1+b) \sum_{i + j = m} ij \partial_i \partial_j - \sum_{i} (i + m) p_i \partial_{i+m} - b m (m-1) \partial_m - \delta_{m,1} \frac{z}{\h^2(1+b)}.
\]
Here, we write $\partial_i$ as a shorthand for the differential operator $\frac{\partial}{\partial p_i}$, for $i \geq 1$. These constraints uniquely determine $Z^{(b)}$ from the initial condition $\left. Z^{(b)} \right|_{z=0} = 1$. Moreover, the operators $L_1^{(b)}, L_2^{(b)}, L_3^{(b)}, \ldots$ satisfy the Virasoro commutation relations
\[
[L_m^{(b)}, L_n^{(b)}] = (m-n) \, L_{m+n}^{(b)}, \qquad \text{for } m, n \geq 1.
\]
\end{theorem}

It will also be convenient to introduce the following expansion of the partition function, where the exponential in \cref{eq:bmonotone} has been a removed, the indexing is over $r$ rather than $g$, and a factor of $\frac{1}{\mu_1 \cdots \mu_n}$ has been introduced.
\begin{equation} \label{eq:bmonotonedisc}
Z^{(b)}(\bm{p}; \h, z) = \sum_{k \geq 0} \frac{z^k}{\h^k} \sum_{r \geq 0} \sum_{n \geq 1} \frac{\h^r}{n!} \sum_{\mu_1, \ldots, \mu_n \geq 1} \frac{\hb_r(\mu_1, \ldots, \mu_n)}{\mu_1 \cdots \mu_n} \, \frac{p_{\mu_1} \cdots p_{\mu_n}}{(b+1)^n}
\end{equation}
It follows that for the value $b = 0$, $\vec{h}^{(0)}_r(\mu_1, \ldots, \mu_n)$ is the enumeration of \cref{def:monotonehurwitz} without the final transitivity constraint, where $r$ is the number of transpositions and the result is multiplied by $\mu_1 \cdots \mu_n$.

The $b$-monotone Hurwitz numbers provide an interesting yet manageable example of the more general notion of $b$-deformation of enumerative problems, which often involve the inclusion of ``non-orientable'' contributions. Such deformations appear to fit into the more general framework of refinement in mathematical physics. At present, the picture is far from complete, but the $b$-monotone Hurwitz numbers provide a useful case study from which more general results may follow in the future.

\section{Weingarten calculus for real Grassmannians} \label{sec:weingarten-A} 

In this section, we consider integration on the the real Grassmannian $\mathrm{Gr}_\mathbb{R}(M,N)$ for $M < N$, where we realise the Grassmannian as the space of $N \times N$ idempotent real symmetric matrices of rank $M$. In fact, this space can be described in the following three equivalent ways.

\begin{definition} \label{def:grassmannian}
Let $I_M$ denote the $M \times M$ identity matrix and let $I_{M,N}$ denote the $N \times N$ matrix whose first $M$ diagonal entries are 1 and whose remaining entries are 0. For $M < N$, define the space
\begin{align*}
\mathbf{A}(M,N) &= \{ A \in \mathrm{Mat}_{N \times N}(\mathbb{R}) \mid A^2 = A, A = A^T \text{ and } \mathrm{rank}(A) = M \} \\
&= \{ A = B^TB \mid B \in \mathrm{Mat}_{M \times N}(\mathbb{R}) \text{ and } BB^T = I_M \} \\
&= \{ A = O I_{M,N} O^T \mid O \in \O(N) \}.
\end{align*}
\end{definition}

The orthogonal group $\O(N)$ acts transitively on $\mathbf{A}(M,N)$ by conjugation, endowing it with the structure of a compact homogeneous space. Thus, the Haar measure on $\O(N)$ induces an $\O(N)$-invariant normalised Haar measure on $\mathbf{A}(M,N)$, which we denote succinctly by $\dd \mu$. Since the stabiliser of $I_{M,N} \in \mathbf{A}(M,N)$ is $\O(M) \times \O(N-M)$, we may identify $\mathbf{A}(M,N)$ with the real Grassmannian
\[
\mathrm{Gr}_\mathbb{R}(M,N) \cong \faktor{\O(N)}{\O(M) \times \O(N-M)}.
\]

For $1 \leq i, j \leq N$, define the function $A_{ij}: \mathbf{A}(M,N) \to \mathbb{R}$ that records the $(i,j)$ matrix element. Our primary goal is to calculate integrals of the form
\[
\int_{\mathbf{A}(M,N)} A_{i(1) i(2)} A_{i(3) i(4)} \cdots A_{i(2k-1) i(2k)} \, \dd \mu,
\]
for a function $i: \{1, 2 \ldots, 2k\} \to \{1, 2, \ldots, N\}$. We impose the technical assumption that $k \leq N$ for future convenience. However, this assumption has little bearing on our work, since we will eventually consider these matrix integrals in the regime of large $N$ and fixed ratio $\frac{M}{N}$.

\begin{remark}
A systematic study of Weingarten calculus for matrix ensembles associated with compact symmetric spaces has been carried out by Matsumoto~\cite{mat13}. The three so-called {\em chiral ensembles} give rise to integration on spaces that can be naturally identified with Grassmannians. However, the matrix ensembles studied by Matsumoto differ from the one considered here, leading to theories that are at least superficially distinct. On the other hand, it may be the case that there exists a non-trivial relationship between the Weingarten function that we define below and those defined by Matsumoto.
\end{remark}

\subsection{Convolution formula and orthogonality relations} \label{subsec:orthogonality_A}

We commence by defining a relevant Weingarten function, which focuses on a certain class of ``elementary'' integrals over the space $\mathbf{A}(M,N)$.

\begin{definition} \label{def:weingarten_A}
Define the {\em Weingarten function} $\wga: \P_\bullet \to \mathbb{R}(M, N)$ by
\[ 
\wga(\m) = \int_{\mathbf{A}(M,N)} A_{i(1)i(2)} A_{i(3)i(4)} \cdots A_{i(2k-1)i(2k)} \, \dd \mu,
\]
where $i: \{1, 2, \ldots, 2k\} \to \{1, 2, \ldots, N\}$ is a strongly admissible function for $\m$.
\end{definition}

As in the case of \cref{def:weingartenfn} for the orthogonal Weingarten function, one does not know a priori that the integral is independent of the choice of $\m$ nor that the result is a rational function. However, these are both consequences of the following result, which relates the integral above to the orthogonal Weingarten function. Furthermore, observe that the integral in the definition of $\wga(\m)$ is invariant under the change $i \mapsto i \circ \rho$ for $\rho \in H_k$. This follows directly from the fact that $A^T = A$ for $A \in \mathbf{A}(M,N)$ and the fact that the matrix elements commute. It follows that $\wga(\m)$ depends only on the coset-type of the pair partition $\m$ and thus, descends to a function on the set of partitions.

\begin{proposition}[Convolution formula for $\mathbf{A}(M,N)$] \label{prop:convolution_A}
For any function $i: \{1, 2, \ldots, 2k\} \to \{1, 2, \ldots, N\}$,
\[
\int_{\mathbf{A}(M,N)} A_{i(1) i(2)} A_{i(3) i(4)} \cdots A_{i(2k-1) i(2k)} \, \dd \mu = \sum_{\m \in \P_{k}} \Delta_{\m}(i) \, \wga(\m).
\]
Here, we set $\Delta_\m(i) = 1$ if $i$ is admissible for $\m$ and $\Delta_\m(i) = 0$ otherwise.
\end{proposition}

\begin{proof}
We lift the integrals from $\mathbf{A}(M,N)$ to $\O(N)$, using \cref{def:grassmannian} and the compatibility of the Haar measures $\dd \mu$ on $\mathbf{A}(M,N)$ and $\dd \nu$ on $\O(N)$. Using the convolution formula for $\O(N)$ of \cref{thm:convolution_O}, the left side of the desired equality becomes
\begin{align*}
\text{LHS} &= \int_{\O(N)} [OI_{M,N}O^T]_{i(1)i(2)} \cdots [OI_{M,N}O^T]_{i(2k-1)i(2k)} \, \dd \nu \\
&= \sum_{j_1, \ldots, j_k = 1}^M \int_{\O(N)} O_{i(1) j_1} O_{i(2) j_1} \cdots O_{i(2k-1) j_k} O_{i(2k) j_k} \, \dd \nu \\
&= \sum_{j_1, \ldots, j_k = 1}^M \sum_{\m, \n \in \P_k} \Delta_\m(i)\, \Delta_\n(j) \, \wgo(\m^{-1} \n),
\end{align*}
where $(j(1), j(2), \ldots, j(2k)) = (j_1, j_1, j_2, j_2, \ldots, j_k, j_k)$.

The right side becomes
\begin{align*}
\text{RHS}&= \sum_{\m \in \P_{k}} \Delta_{\m}(i) \int_{\mathbf{A}(M,N)} A_{h(1) h(2)} \cdots A_{h(2k-1) h(2k)} \, \dd \mu \\
&= \sum_{j_1, \ldots, j_k = 1}^M \sum_{\m \in \P_{k}} \Delta_{\m}(i) \int_{\O(N)} O_{h(1) j_1} O_{h(2) j_1} \cdots O_{h(2k-1) j_k} O_{h(2k) j_k} \, \dd \nu \\
&= \sum_{j_1, \ldots, j_k = 1}^M \sum_{\m \in \P_{k}} \Delta_{\m}(i) \sum_{\mathfrak{l}, \n \in \P_k} \Delta_{\mathfrak{l}}(h) \, \Delta_\n(j) \, \wgo(\mathfrak{l}^{-1} \n) \\
&= \sum_{j_1, \ldots, j_k = 1}^M \sum_{\m, \n \in \P_{k}} \Delta_{\m}(i) \, \Delta_\n(j) \, \wgo(\m^{-1} \n),
\end{align*}
where $h$ is any strongly admissible function for $\m$ and again $(j(1), j(2), \ldots, j(2k)) = (j_1, j_1, j_2, j_2, \ldots, j_k, j_k)$. The third equality uses the convolution formula for $\O(N)$ of \cref{thm:convolution_O} and the fourth equality uses the fact that one only obtains a contribution when $\mathfrak{l} = \mathfrak{m}$.
\end{proof}

We now use the convolution formula of \cref{prop:convolution_A} to deduce orthogonality relations for the Weingarten function $\wga$.

\begin{theorem}[Orthogonality relations for $\wga$] \label{thm:orthogonality_A}
For each non-empty pair partition $\m \in \P_k$, the Weingarten function $\wga$ satisfies the relation
\begin{align*}
\wga(\m) = &-\frac{1}{N} \sum_{i=1}^{2k-2} \wga((i~2k-1) \cdot \m) + \frac{M}{N} \delta_{\{2k-1,2k\} \in \m} \wga(\m^\downarrow) \\
&+ \frac{1}{N} \sum_{i=1}^{2k-2} \delta_{\{i,2k\}\in\m} \wga([(i~2k-1)\cdot\m]^\downarrow).
\end{align*}
Here, we set $\delta_{\{i,j\} \in \m} = 1$ if $\{i,j\} \in \m$ and $\delta_{\{i,j\} \in \m} = 0$ otherwise. If $\{2k-1,2k\} \in \m$, we define $\m^\downarrow \in \P_{k-1}$ to be the pair partition obtained by removing the pair $\{2k-1, 2k\}$ from $\m$.
\end{theorem}

\begin{proof}
We consider the cases $\{2k-1,2k\} \in \m$ and $\{2k-1,2k\} \notin \m$ separately.

{\em Case 1.} Suppose $\{2k-1,2k\} \in \m$. \\
Let $j$ be a strongly admissible function for $\m^{\downarrow}$ with values in $\{1, 2, \ldots, k-1\}$ and consider the integral
\[
\sum_{i=1}^N \int_{\mathbf{A}(M,N)} A_{j(1)j(2)} \cdots A_{j(2k-3)j(2k-2)} A_{i i} \, \dd \mu.
\]
One can use $\sum_{i=1}^N A_{ii} = \mathrm{Tr}(A) = \mathrm{Tr}(O I_{M,N} O^T) = \mathrm{Tr}(I_{M,N}) = M$ and the definition of the Weingarten function $\wga$ to express the integral as $M \wga(\m^\downarrow)$.

Alternatively, the convolution formula of \cref{prop:convolution_A} can be applied directly to each summand as follows. For $k \leq i \leq N$, the sequence $(j(1), j(2), \ldots, j(2k-2), i, i)$ defines a strongly admissible function for $\m$ and is not admissible for any other pair partition. So the convolution formula implies that
\[
\sum_{i=k}^{N} \int_{\mathbf{A}(M,N)}A_{j(1)j(2)} \cdots A_{j(2k-3)j(2k-2)} A_{i i}\, \dd \mu = (N-k+1) \, \wga(\m).
\]

For $1\leq i \leq k-1$, the sequence $(j(1), j(2), \ldots, j(2k-2), i, i)$ defines a function that is admissible for $\m$ and for exactly two other pair partitions. These pair partitions arise from the three distinct ways to pair the four terms that are equal to $i$. They are given by $\m$, $(a~2k-1) \cdot \m$ and $(b~2k-1) \cdot \m$, where $\{a, b\} \in \m^{\downarrow}$ and $j(a) = j(b) = i$. Summing these contributions, one obtains
\[
\sum_{i=1}^{k-1} \int_{\mathbf{A}(M,N)} A_{j(1)j(2)} \cdots A_{j(2k-3)j(2k-2)} A_{i i} \, \dd \mu = (k-1) \, \wga(\m) + \sum_{i=1}^{2k-2} \wga((i ~2k-1) \cdot \m).
\]

These two sums combine to give
\begin{equation} \label{eq:wgacase1}
M \wga(\m^\downarrow) = N \wga(\m) + \sum_{i=1}^{2k-2} \wga((i ~2k-1) \cdot \m),
\end{equation}
which rearranges to yield the desired equality.

{\em Case 2.} Suppose $\{2k-1,2k\} \notin \m$. \\
Let $\{a, 2k-1\}$ and $\{b, 2k\}$ be the pairs in $\m$ containing $2k-1$ and $2k$, so that $1 \leq a, b\leq 2k-2$ with $a\neq b$. For ease of notation, suppose that $a$ is even, although the following argument applies equally when $a$ is odd.

Let $j: \{1,2,\ldots,2k\}\setminus\{a, 2k-1\} \to \{1, 2, \ldots, k-1\}$ be a function that is strongly admissible for the pair partition $\m \setminus \{\{a, 2k-1\}\}$. That is, $i$ assigns to each pair of $\m$, excluding $\{a, 2k-1\}$, a unique label from $1$ to $k-1$. Now consider the integral
\begin{equation} \label{eq:prooforthogonality_A}
\sum_{i=1}^{N} \int_{\mathbf{A}(M,N)} \left( A_{j(1)j(2)} \cdots A_{j(a-3)j(a-2)} \right) A_{j(a-1)i} \left( A_{j(a+1) j(a +2)} \cdots A_{j(2k-3)j(2k-2)} \right) A_{ij(2k) } \, \dd \mu.
\end{equation}

On the one hand, we can use the fact that $\sum_{i=1}^{N} A_{j(a-1)i} A_{ij(2k)} = A_{j(a-1)j(2k)}$ since $A^2 = A$ for any $A\in\mathbf{A}(M,N)$. Then commuting the sum and the integral yields
\[
\int_{\mathbf{A}(M,N)} \left( A_{j(1)j(2)} \cdots A_{j(a-3) j(a-2)} \right) A_{j(a-1)j(2k)} \left( A_{j(a+1)j(a +2)} \cdots A_{j(2k-3)j(2k-2)} \right) \, \dd \mu.
\]
The sequence of indices $(j(1), \ldots, j(a-1), j(2k), j(a+1), \ldots, j(2k-2))$ defines a strongly admissible function for $[(b~2k-1) \cdot \m]^{\downarrow}$. So the expression in \cref{eq:prooforthogonality_A} is simply equal to $\wga([(b~2k-1) \cdot \m]^{\downarrow})$.

On the other hand, we can apply the convolution formula of \cref{prop:convolution_A} directly to each summand of \cref{eq:prooforthogonality_A}. For a given $1 \leq i \leq N$, define the function $h$ by
\[
h(s) = \begin{cases}
j(s), & s \in \{1,2,\ldots, 2k\} \setminus \{a, 2k-1\}, \\
i, & s \in \{a, 2k-1\}.
\end{cases}
\]
Note that $h$ is the sequence of indices of the integrand of the $i$th summand of \cref{eq:prooforthogonality_A}.

For $k \leq i \leq N$, it is clear that $h$ is a strongly admissible function for the pair partition $\m$, so the summand is simply $\wga(\m)$.

For $1 \leq i \leq k-1$, there is a pair $\{a_1, a_2\} \in \m \setminus \{\{a, 2k-1\}\}$ such that $j(a_1) = j(a_2) = i$. 
Then $h$ is an admissible function precisely for the three pair partitions $\m$, $(a_1 ~ 2k-1) \cdot \m$ and $(a_2~2k-1) \cdot \m$. Thus, the $i$th summand of \cref{eq:prooforthogonality_A} is equal to
\[
\wga(\m) + \wga((a_1 ~ 2k-1) \cdot \m) + \wga((a_2 ~ 2k-1) \cdot \m).
\]

Summing over these contributions for $1 \leq i \leq N$, the integral of \cref{eq:prooforthogonality_A} can be expressed as
\[
N \wga(\m) + \sum_{\substack{i=1\\ i\neq a}}^{2k-2} \wga((i ~ 2k-1) \cdot \m) + \wga((2k-1 ~ 2k) \cdot \m).
\]
To obtain the desired formula, observe that the pair partitions $(2k-1 ~2k) \cdot \m$ and $(a ~ 2k-1) \cdot \m$ are of the same coset-type. Invoking the coset-type invariance of $\wga$ allows us to conclude that
\begin{equation} \label{eq:wgacase2}
\wga([(b~2k-1) \cdot \m]^{\downarrow}) = N \wga(\m) + \sum_{i=1}^{2k-2} \wga((i ~ 2k-1) \cdot \m).
\end{equation}

The desired result is simply the expression of \cref{eq:wgacase1} from Case 1 and \cref{eq:wgacase2} from Case 2 in a unified manner.
\end{proof}

\subsection{Large \texorpdfstring{$N$}{N} expansion of the Weingarten function} \label{subsec:expansion_A}

As in other contexts for Weingarten calculus, the orthogonality relations motivate the definition of a Weingarten graph, from which we derive a large $N$ expansion for the Weingarten function~\cite{col-mat17}. The coefficients of this expansion are weighted enumerations of monotone factorisations, naturally leading to a variation of the usual monotone Hurwitz numbers.

\begin{definition}
The {\em $t$-deformed orthogonal Weingarten graph} $\mathcal{G}^{\mathbf{A}}$ is the infinite directed graph with vertex set $\P_\bullet$ and edge set $E_A \sqcup E_B \sqcup E_C$, where
\begin{itemize}
\item the set $E_A$ comprises ``type $A$'' edges, which are of the form $\m \longrightarrow (i ~ 2k-1) \cdot \m$ for $\m \in \P_k$ and $1 \leq i \leq 2k-2$;
\item the set $E_B$ comprises ``type $B$'' edges, which are of the form $\m \longrightarrow \m^\downarrow$ for $\m \in \P_k$ with $\{2k-1, 2k\} \in \m$;
\item the set $E_C$ comprises ``type $C$'' edges, which are of the form $\m \longrightarrow [(i~2k-1) \cdot \m]^\downarrow$ for $\m \in \P_k$ with $\{i, 2k\} \in \m$ and $1 \leq i \leq 2k-2$.
\end{itemize}

Consistent with usual graph-theoretic terminology, we define a {\em path} in $\mathcal{G}^{\mathbf{A}}$ to be a tuple of pair partitions 
$(\m_0, \m_1, \m_2, \ldots, \m_\ell)$ such that $\m_{i-1} \longrightarrow \m_i$ is a directed edge of $\mathcal{G}^{\mathbf{A}}$ for $i = 1, 2, \ldots, \ell$.

Let $\mathrm{Path}^{\mathbf{A}}(\m)$ denote the set of paths from the pair partition $\m$ to the empty pair partition $(~)$ in the Weingarten graph $\mathcal{G}^\mathbf{A}$. 
\end{definition}

The three edge types in $\mathcal{G}^{\mathbf{A}}$ correspond to the three terms on the right side of the orthogonality relations of \cref{thm:orthogonality_A}. Thus, the Weingarten graph $\mathcal{G}^{\mathbf{A}}$ is obtained from the orthogonal Weingarten graph $\mathcal{G}^\O$ of \cref{fig:weingartengraph} by adding type $C$ edges. We omit a diagram of $\mathcal{G}^{\mathbf{A}}$, which would likely be too cluttered to be illustrative.

By construction, we can express the Weingarten function $\wga$ as a weighted enumeration of walks in $\mathcal{G}^{\mathbf{A}}$.

\begin{proposition}[Large $N$ expansion] \label{prop:graph-func-equiv}
For each pair partition $\m$, we have the large $N$ fixed $\frac{M}{N}$ expansion
\[
\wga(\m) = \sum_{\bm{\rho} \in \mathrm{Path}^{\mathbf{A}}(\m)} \left( -\frac{1}{N} \right)^{\ell_A(\bm{\rho})} \left( \frac{M}{N} \right)^{\ell_B(\bm{\rho})} \left( \frac{1}{N} \right)^{\ell_C(\bm{\rho})},
\]
where the summation is over all paths in $\mathcal{G}^{\mathbf{A}}$ from $\m$ to the empty pair partition $(~)$. We use $\ell_K(\bm{\rho})$ to denote the number of type $K$ edges for $K \in \{A, B, C\}$.
\end{proposition}

Given a path in $\mathcal{G}^{\mathbf{A}}$ from a pair partition $\m \in \P_k$ to the empty pair partition $(~)$, one can record the sequence of transpositions $(\tau_r, \tau_{r-1}, \ldots, \tau_1)$ associated to the edges of types $A$ and $C$. By construction, the reversed sequence $(\tau_1, \tau_2, \ldots, \tau_r)$ is a monotone sequence of {\em odd} transpositions -- that is, transpositions of the form $(a~b)$ with $a < b$ and $b$ odd. Furthermore, one has the property $\tau_1 \circ \tau_2 \circ \cdots \circ \tau_r \cdot \m = \e_k$, which motivates the following definition.

\begin{definition} \label{def:factorisation}
A {\em monotone factorisation} of a pair partition $\m \in \P_k$ is a sequence $\bm{\tau} = (\tau_1, \tau_2, \ldots, \tau_r)$ of transpositions in $S_k$ such that 
\begin{itemize}
\item if we write $\tau_i = (a_i~b_i)$ with $a_i < b_i$, then $b_1 \leq b_2 \leq \cdots \leq b_m$ are odd, and
\item $\tau_1 \circ \tau_2 \circ \cdots \circ \tau_r \cdot \m = \e_k$.
\end{itemize}

Define the {\em hive number} of such a monotone factorisation $\bm{\tau}$ to be the number of distinct elements in the multiset $\{b_1, b_2, \ldots, b_m\}$ and denote it by $\mathrm{hive}(\bm{\tau})$.

Such a monotone factorisation is said to be {\em connected} if $\langle \tau_1, \tau_2, \ldots, \tau_r, \iota \rangle$ is a transitive subgroup of $S_{2k}$, where~$\iota$ denotes the involution $(1~2)(3~4) \cdots (2k-1~2k)$.

Denote by $\mathrm{Mono}(\m)$ the set of monotone factorisations of the pair partition $\m$ and by $\mathrm{CMono}(\m)$ the set of connected monotone factorisations of the pair partition $\m$.
\end{definition}

The recording of transpositions along a path in the Weingarten graph $\mathcal{G}^{\mathbf{A}}$ gives rise to a map
\[
\mathrm{F} : \mathrm{Path}^{\mathbf{A}}(\m) \to \mathrm{Mono}(\m),
\]
for each pair partition $\m$. This construction does not provide a bijection, since there are many different paths that give rise to the same monotone factorisation. To realise the preimage $\mathrm{F}^{-1}(\bm{\tau})$ for an arbitrary monotone factorisation $\bm{\tau}$ of length $r$, observe that there is a unique path $\bm{\rho}$ in $\mathcal{G}^{\mathbf{A}}$ comprising only edges of types $A$ and $B$ such that $F(\bm{\rho}) = \bm{\tau}$. All other preimages of $\bm{\tau}$ can be obtained from $\bm{\rho}$ by exchanging a pair of consecutive edges of types $A$ and $B$ in order with the corresponding edge of type $C$. The number of such pairs of edges that can be exchanged is precisely $\mathrm{hive}(\bm{\tau})$, so we have $|F^{-1}(\bm{\tau})| = 2^{\mathrm{hive}(\bm{\tau})}$. At the level of the weighted enumeration of paths, for a monotone factorisation $\bm{\tau}$ of a pair partition $\m \in \P_k$, we have the generating function identity
\begin{align*}
\sum_{\bm{\rho} \in F^{-1}(\bm{\tau})} x^{\ell_A(\bm{\rho})} y^{\ell_B(\bm{\rho})} z^{\ell_C(\bm{\rho})} &= x^{\ell_A(\bm{\rho}) + \ell_C(\bm{\rho})} y^{\ell_B(\bm{\rho}) + \ell_C(\bm{\rho})} (1 + x^{-1} y^{-1} z)^{\mathrm{hive}(\bm{\tau})} \\
&= x^{\ell(\bm{\tau})} y^k (1 + x^{-1} y^{-1} z)^{\mathrm{hive}(\bm{\tau})}.
\end{align*}
Here, we have used $\ell(\bm{\tau})$ to denote the length of $\bm{\tau}$.

The equation can now be applied to the large $N$ expansion of \cref{prop:graph-func-equiv} to give
\[
\wga(\m) = \sum_{r=0}^\infty \left(-\frac{1}{N}\right)^r \left(\frac{M}{N}\right)^k \sum_{\substack{\bm{\tau} \in \mathrm{Mono}(\m) \\ \ell(\bm{\tau}) = r}} \left(1- \dfrac{N}{M}\right)^{\mathrm{hive}(\bm{\tau})},
\]
which motivates the following.

\begin{definition} \label{def:monotonecount}
For $\m$ a pair partition and $r \geq 0$, let $\vec{h}^{(t)}_r(\m) \in \mathbb{Z}[t]$ denote the weighted enumeration of monotone factorisations of $\m$ with length $r$, where the weight of a monotone factorisation $\bm{\tau}$ is $t^{\mathrm{hive}(\bm{\tau})}$.
\end{definition}

\begin{remark}
One of the interesting features of monotone Hurwitz numbers and their variants is that they can be defined in several ways. It is worth noting that in \cref{sec:prelim,sec:hurwitz}, the monotone Hurwitz numbers are defined as coefficients of a particular generating function and then shown to enumerate monotone factorisations. In this section, our path through Weingarten calculus motivates us to take the reverse approach and define the monotone Hurwitz numbers above as weighted enumerations of monotone factorisations.
\end{remark}

The previous discussion then leads directly to the following result, which expresses the large $N$ fixed $\frac{M}{N}$ expansion of the Weingarten function $\wga$ as a weighted enumeration of monotone factorisations.

\begin{theorem}[Large $N$ expansion] \label{thm:expansion_A}
For each pair partition $\m \in \P_k$, we have the large $N$ fixed $\frac{M}{N}$ expansion 
\[
\wga(\m) = \frac{1}{(1-t)^k} \sum_{r=0}^\infty \vec{h}^{(t)}_r(\m) \left(-\frac{1}{N}\right)^r,
\]
where $t = 1 - \frac{N}{M}$.
\end{theorem}

It follows from the fact that $\wga(\m)$ is coset-type invariant that $\vec{h}^{(t)}_r(\m)$ is as well. Thus, $\vec{h}^{(t)}_r$ descends to a function on the set of partitions and one can make sense of the expression $\vec{h}^{(t)}_r(\lambda)$ for $\lambda$ a partition. Naturally, one can also extend the definition to $\vec{h}^{(t)}_r(\mu_1, \ldots, \mu_n)$ for any positive integers $\mu_1, \ldots, \mu_n$ by declaring $\vec{h}^{(t)}_r$ to be symmetric.

As previously mentioned, there is a close relation between Weingarten functions and Jucys--Murphy elements in the symmetric group algebra~\cite{mat-nov13,nov10}. The following result is the analogue of \cref{prop:ojucysmurphy} for the Weingarten function $\wga$.

\begin{proposition} \label{prop:jucysmurphy_A}
For each positive integer $k$, we have the following equality in the vector space $\mathbb{C}[\P_k]$.
\[
\sum_{\m \in \P_k} \wga(\m) \, \m = \prod_{i=1}^k \frac{M + J_{2i-1}}{N + J_{2i-1}} \cdot \e_k
\]
\end{proposition}

\begin{proof}
We use an inductive approach, noting that the base case is straightforward to check. Now suppose that the equation is true for a positive integer $k-1$. Take the orthogonality relation of \cref{thm:orthogonality_A}, multiply by $\m \in \mathbb{C}[\P_k]$, and sum over all $\m \in \P_k$.
\begin{align*}
\sum_{\m \in \P_k} \wga(\m) \, \m = &-\frac{1}{N} \sum_{\m \in \P_k} \sum_{i=1}^{2k-2} \wga((i~2k-1) \cdot \m) \, \m + \frac{M}{N} \sum_{\m \in \P_k} \delta_{\{2k-1,2k\} \in \m} \wga(\m^\downarrow) \, \m \\
&+ \frac{1}{N} \sum_{\m \in \P_k} \sum_{i=1}^{2k-2} \delta_{\{i,2k\}\in\m} \wga([(i~2k-1)\cdot\m]^\downarrow) \, \m
\end{align*}

The two summations over $1 \leq i \leq 2k-2$ can be naturally expressed via the action of the odd Jucys--Murphy elements $J_1, J_3, \ldots, J_{2k-1} \in \mathbb{C}[S_k]$, leading to the equation below. For $\n \in \P_{k-1}$, we use $\n^\uparrow$ to denote the element of $\P_k$ obtained by adding the pair $\{2k-1, 2k\}$.
\[
\sum_{\m \in \P_k} \wga(\m) \, \m = -\frac{1}{N} J_{2k-1} \sum_{\m \in \P_k} \wga(\m) \, \m + \frac{M}{N} \sum_{\n \in \P_{k-1}} \wga(\n) \, \n^{\uparrow} + \frac{1}{N} J_{2k-1} \sum_{\n \in \P_{k-1}} \wga(\n) \, \n^{\uparrow}
\]

The following is obtained by rearranging the equation above and then invoking the inductive hypothesis, thus completing the induction.
\[
\sum_{\m \in \P_k} \wga(\m) \, \m = \frac{M+J_{2k-1}}{N+J_{2k-1}} \sum_{\n \in \P_{k-1}} \wga(\n) \, \n^{\uparrow} = \prod_{i=1}^k \frac{M + J_{2i-1}}{N + J_{2i-1}} \cdot \e_k \qedhere
\]
\end{proof}

Introduce the partition function
\[
Z^{\mathbf{A}}(\bm{p}; \h, z) = \sum_{k \geq 0} \frac{z^k}{\h^k} \sum_{r \geq 0} \sum_{n \geq 1} \frac{\h^r}{n!} \sum_{\mu_1, \ldots, \mu_n \geq 1} \frac{\vec{h}^{(t)}_r(\mu_1, \ldots, \mu_n)}{\mu_1 \cdots \mu_n} \, \frac{p_{\mu_1} \cdots p_{\mu_n}}{2^n},
\]
which stores the weighted enumeration of monotone factorisations of \cref{def:monotonecount}, or essentially equivalently, values of the Weingarten function $\wga$. It is a formal power series in the variables $\h$, $z$ and $p_1, p_2, p_3, \ldots$, and the coefficients are polynomials in $t$.

\begin{proposition}[Virasoro constraints for $Z^\mathbf{A}$] \label{thm:virasoro_A}
The partition function $Z^\mathbf{A}(\bm{p}; \h, z)$ satisfies $L_m^\mathbf{A} \cdot Z^\mathbf{A} = 0$, for $m = 1, 2, 3, \ldots$, where
\[
L_m^\mathbf{A} = \frac{m \partial_m}{\h} - \sum_{i} (i + m) p_i \partial_{i+m} - 2 \sum_{i + j = m} ij \partial_i \partial_j - (m-1) m \partial_m - (t-1) \frac{z}{\h} (m-1) \partial_{m-1} - \delta_{m,1} \frac{z}{2\h^2}.
\]
Here, we write $\partial_i$ as a shorthand for the differential operator $\frac{\partial}{\partial p_i}$, where $i \geq 1$. These constraints uniquely determine $Z^{\mathbf{A}}$ from the initial condition $[z^0] \, Z^{\mathbf{A}} = 1$. Moreover, the operators $L_1^\mathbf{A}, L_2^\mathbf{A}, L_3^\mathbf{A}, \ldots$ satisfy the Virasoro commutation relations
\[
[L_m^\mathbf{A}, L_n^\mathbf{A}] = (m-n) \, L_{m+n}^\mathbf{A}, \qquad \text{for } m, n \geq 1.
\]
\end{proposition}

\begin{proof}
The orthogonality relations of \cref{thm:orthogonality_A}, coupled with the large $N$ expansion of \cref{thm:expansion_A}, imply the following analogue of the cut-and-join recursion for monotone Hurwitz numbers~\cite{gou-gua-nov13}. For positive integers $\mu_1, \mu_2, \ldots, \mu_n$,
\begin{align*}
\vec{h}^{(t)}_{r}(\mu_1, \mu_S) = &\sum_{j=2}^n 2 \mu_j \, \vec{h}^{(t)}_{r-1}(\mu_1+\mu_j, \mu_{S\setminus \{j\}}) + \sum_{\alpha+\beta=\mu_1} \vec{h}^{(t)}_{r-1}(\alpha, \beta, \mu_S) \\
&+ (\mu_1-1) \, \vec{h}^{(t)}_{r-1}(\mu_1, \mu_S) + (t-1) \, \vec{h}^{(t)}_{r-1}(\mu_1-1, \mu_S) + \delta_{\mu_1,1} \, \vec{h}^{(t)}_r(\mu_S).
\end{align*}
Here, we set $S = \{2, 3, \ldots, n\}$ and $\mu_I = (\mu_{i_1}, \mu_{i_2}, \ldots \mu_{i_k})$ for $I = \{i_1, i_2, \ldots, i_k\}$.

The Virasoro constraint $L^\mathbf{A}_{\mu_1} \cdot Z = 0$ is obtained by multiplying both sides of the recursion by the monomial $\mu_1 p_{\mu_1} p_{\mu_2} \cdots p_{\mu_n} \h^{r-|\mu|} \frac{z^{|\mu|}}{|\mu|!}$ and summing over all non-negative integers $r$ and all positive integers $\mu_1, \mu_2, \ldots, \mu_n$.

The uniqueness of the solution to the Virasoro constraints up to scale is apparent as the recursion effectively computes every polynomial $\vec{h}^{(t)}_r(\mu_1, \ldots, \mu_n)$ from the base case $\vec{h}^{(t)}_0(\mu) = \delta_{\mu, \emptyset}$.

Finally, one can check the Virasoro commutation relations directly by a relatively straightforward brute force calculation.
\end{proof}

\cref{thm:expansion_A} demonstrates that the Weingarten function $\wga$ is essentially an enumeration of monotone factorisations. It is customary in the literature to refer to an enumeration of monotone factorisations as a monotone Hurwitz number~\cite{gou-gua-nov14}. In the current context, one could define the notion of a ``$t$-deformed orthogonal monotone Hurwitz number'' $\vec{H}^{(t),\mathrm{orth}}_{g,n}(\mu_1, \ldots, \mu_n)$ to be $\frac{1}{|\mu|!}$ multiplied by the weighted enumeration of connected monotone factorisations of a pair partition of coset-type $\mu$, with length $|\mu|+2g-2+n$. As in \cref{def:monotonecount}, the weight of a monotone factorisation $\bm{\tau}$ is $t^{\mathrm{hive}(\bm{\tau})}$.

We do not pursue the topic of $t$-deformed orthogonal monotone Hurwitz numbers here but instead fold that discussion into a more general one concerning $bt$-monotone Hurwitz numbers, which we introduce in the next section.

\begin{remark}
One can consider a symplectic version of the story told in this section. However, as suggested in \cref{rem:symplectic}, no new information is obtained as the symplectic Weingarten function can be recovered, up to sign, from the orthogonal Weingarten function by a simple substitution. In fact, the three cases of interest are essentially subsumed by the $b$-Weingarten calculus introduced in the next section, from which we recover the unitary case at $b = 0$, the orthogonal case at $b = 1$, and the symplectic case at $b = -\frac{1}{2}$.
\end{remark}

\section{The \texorpdfstring{$bt$}{bt}-monotone Hurwitz numbers} \label{sec:hurwitz}

In this section, we introduce and develop the notion of $b$-Weingarten calculus, which interpolates between unitary and orthogonal Weingarten calculus. and depends on a parameter $b$. These previously well-studied theories are recovered at $b = 0$ and $b = 1$, respectively. Our construction of $b$-Weingarten calculus relies crucially on a weight function that takes two pair partitions and returns an element of the set $\{0, 1, b\}$.

We furthermore show that the $b$-deformation introduced in this section and the $t$-deformation introduced in the previous section can be combined to produce a $bt$-Weingarten calculus and corresponding $bt$-monotone Hurwitz numbers. These can be interpreted as a weighted enumeration of monotone factorisations, in which the $b$-parameter records the ``flip number'' and the $t$-parameter records the ``hive number''. On the basis of extensive numerical evidence, we extend the real-rootedness and interlacing conjectures previously posited for $t$-monotone Hurwitz numbers to this more general setting.

\subsection{A \texorpdfstring{$b$}{b}-Weingarten calculus} \label{subsec:weingarten_B}

The symmetric group $S_k$ naturally embeds into the set of pair partitions $\P_k$ via the map
\[
\sigma \mapsto \{ \{2i-1,2\sigma(i)\} \mid i = 1, 2, \ldots, k\}.
\]
This leads to an embedding of the unitary Weingarten graph $\mathcal{G}^{\mathbf{U}}$, whose vertices correspond to permutations, into the orthogonal Weingarten graph $\mathcal{G}^{\O}$, whose vertices correspond to pair partitions. Our initial goal is to add $b$-dependent weights to the orthogonal Weingarten graph such that the (unweighted) orthogonal Weingarten graph is recovered at $b = 1$, while the unitary Weingarten graph is separated as a connected component at $b = 0$. Here, we consider edges with weight 0 to be removed.

The pivotal object underlying the constructions of this section, and indeed the following section as well, is the weight function $\omega^{(b)}$ defined below.

\begin{definition} \label{def:weightfunction}
For a parameter $b$, define the function $\omega^{(b)}: \P_\bullet \times \P_\bullet \to \{0,1,b\}$ such that $\omega^{(b)}(\m, \n)$ is given by the following procedure.
\begin{enumerate}
\item If $\m \in \P_m$ and $\n \in \P_n$ for $m \neq n$, then set $\omega^{(b)}(\m, \n) = 0$. If $\m \in \P_0$ and $\n \in \P_0$, then set $\omega^{(b)}(\m, \n) = 1$. Otherwise, assume that $\m, \n \in \P_k$ for $k \geq 1$ and proceed to Step 2.

\item If there does not exist a transposition $(i~j) \in S_{2k}$ such that $(i~j) \cdot \m = \n$, then set $\omega^{(b)}(\m, \n) = 0$. Otherwise, assume that there exists a transposition $(i~j) \in S_{2k}$ such that $(i~j) \cdot \m = \n$ and proceed to Step 3.

\item 
\begin{enumerate}
\item Consider the graph $\Gamma(\m)$ on the vertex set $\{1, 2, \ldots, 2k\}$, with an edge joining $u$ and $v$ for each pair $\{u, v\}$ of the identity pair partition $\e_k$ and an edge joining $u$ and $v$ for each pair $\{u, v\}$ of the pair partition $\m$. The graph $\Gamma(\m)$ necessarily comprises disjoint cycles of even lengths.

\item To each vertex $v \in \Gamma(\m)$, assign a {\em charge} $q(v) \in \{+, -\}$ such that the vertex with the largest label in each cycle is assigned $+$ and such that each edge in $\Gamma(\m)$ is incident to one vertex with positive change and one vertex with negative charge.

\item Set $\omega^{(b)}(\m,\n) = 1$ if $q(i) = q(j)$ and set $\omega^{(b)}(\m,\n) = b$ if $q(i) \neq q(j)$.
\end{enumerate}
\end{enumerate}
\end{definition}

Although Step 3 of \cref{def:weightfunction} requires one to choose a transposition $(i~j)$ such that $(i~j) \cdot \m = \n$, the value of $\omega^{(b)}(\m,\n)$ is independent of this choice. In the case $\m = \n$, we have $(i~j) \cdot \m = \n$ for any $\{i,j\}\in \m$. By construction, there is an edge of $\Gamma(\m)$ joining vertices $i$ and $j$, so by the definition of the charge, we have $q(i) \neq q(j)$ and it follows that $\omega^{(b)}(\m, \m) = b$. In the case $\m \neq \n$ with $(i~j) \cdot \m = \n$, there is exactly one other transposition $(i'~j') \neq (i~j)$ such that $(i'~j') \cdot \m = \n$, where $\{i, i'\}$ and $\{j, j'\}$ are pairs in $\m$. By construction, there is an edge of $\Gamma(\m)$ joining vertices $i$ and $i'$ as well as an edge of $\Gamma(\m)$ joining vertices $j$ and $j'$. So by the definition of the charge, we have $q(i) = q(j)$ if and only if $q(i') = q(j')$, and it follows that $\omega^{(b)}(\m, \n)$ is independent of the choice of transposition. Therefore, $\omega^{(b)}$ is indeed well-defined.

\begin{example}
Consider the pair partitions
\[
\m = (1~8 \mmid 2~6 \mmid 3~9 \mmid 4~10 \mmid 5~7) \in \P_5 \qquad \text{and} \qquad \n = (1~8 \mmid 2~6 \mmid 3~5 \mmid 4~10 \mmid 7~9) \in \P_5.
\]
Observe that $\n = (3~7) \cdot \m = (5~9) \cdot \m$. The graph $\Gamma(\m)$ is depicted in \cref{fig:charges} with the vertices labelled by the appropriate charges. Since the charges of vertices 3 and 7 are opposite to each other, we have $\omega^{(b)}(\m, \n) = \omega^{(b)}(\m, (3~7) \cdot \n) = b$.

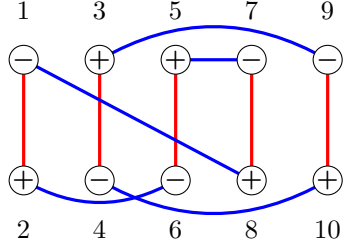
\begin{figure}[ht!]
\centering
\begin{tikzpicture}
\def\x{1}
\def\y{1.6}
\def\d{0.65}

\tikzstyle{M} = [draw=blue, very thick];
\tikzstyle{E} = [draw=red, very thick];

\draw[E] (0*\x,0) -- (0*\x,\y);
\draw[E] (1*\x,0) -- (1*\x,\y);
\draw[E] (2*\x,0) -- (2*\x,\y);
\draw[E] (3*\x,0) -- (3*\x,\y);
\draw[E] (4*\x,0) -- (4*\x,\y);

\draw[M] (0*\x,\y) -- (3*\x,0);
\draw[M] (0*\x,0) to[out=-30,in=-150] (2*\x,0);
\draw[M] (1*\x,\y) to[out=30,in=150] (4*\x,\y);
\draw[M] (1*\x,0) to[out=-30,in=-150] (4*\x,0);
\draw[M] (2*\x,\y) -- (3*\x,\y);

\node at (0*\x,\y+\d) {$1$};
\node at (1*\x,\y+\d) {$3$};
\node at (2*\x,\y+\d) {$5$};
\node at (3*\x,\y+\d) {$7$};
\node at (4*\x,\y+\d) {$9$};

\node at (0*\x,-\d) {$2$};
\node at (1*\x,-\d) {$4$};
\node at (2*\x,-\d) {$6$};
\node at (3*\x,-\d) {$8$};
\node at (4*\x,-\d) {$10$};

\node[circle, draw=black, fill=white, inner sep=0pt] at (0*\x,0) {\small{$\bm{+}$}};
\node[circle, draw=black, fill=white, inner sep=0pt] at (1*\x,0) {\small{$\bm{-}$}};
\node[circle, draw=black, fill=white, inner sep=0pt] at (2*\x,0) {\small{$\bm{-}$}};
\node[circle, draw=black, fill=white, inner sep=0pt] at (3*\x,0) {\small{$\bm{+}$}};
\node[circle, draw=black, fill=white, inner sep=0pt] at (4*\x,0) {\small{$\bm{+}$}};
\node[circle, draw=black, fill=white, inner sep=0pt] at (0*\x,\y) {\small{$\bm{-}$}};
\node[circle, draw=black, fill=white, inner sep=0pt] at (1*\x,\y) {\small{$\bm{+}$}};
\node[circle, draw=black, fill=white, inner sep=0pt] at (2*\x,\y) {\small{$\bm{+}$}};
\node[circle, draw=black, fill=white, inner sep=0pt] at (3*\x,\y) {\small{$\bm{-}$}};
\node[circle, draw=black, fill=white, inner sep=0pt] at (4*\x,\y) {\small{$\bm{-}$}};
\end{tikzpicture}
\caption{The graph $\Gamma(\m)$ corresponding to the pair partition $\m = (1~8 \mmid 2~6 \mmid 3~9 \mmid 4~10 \mmid 5~7) \in \P_5$. The edges corresponding to $\m$ are drawn in blue, while the edges corresponding to $\e_5$ are drawn in red. Each vertex is labelled by its charge.}
\label{fig:charges}
\end{figure}
\end{example}

We use the function $\omega^{(b)}$ to assign weights to the directed edges in the orthogonal Weingarten graph to obtain the $b$-Weingarten graph as follows.

\begin{definition} \label{def:bgraph}
The {\em $b$-Weingarten graph} $\mathcal{G}^{(b)}$ is the infinite weighted directed graph with vertex set $\P_\bullet$ and edge set $E_A \sqcup E_B$, where
\begin{itemize}
\item the set $E_A$ comprises ``type $A$'' edges, which are of the form $\m \longrightarrow (i~2k-1) \cdot \m$ for $\m \in \P_k$ and $1 \leq i \leq 2k-2$, where the weight is $\omega^{(b)}(\m, (i~2k-1) \cdot \m)$,
\item the set $E_B$ comprises ``type $B$'' edges, which are of the form $\m \longrightarrow \m^\downarrow$ for $\m \in \P_k$ with $\{2k-1, 2k\} \in \m$, where the weight is $1$.
\end{itemize}
In short, the graph $\mathcal{G}^{(b)}$ is simply the orthogonal Weingarten graph $\mathcal{G}^{\O}$ with type $A$ directed edges weighted by the function $\omega^{(b)}$.
\end{definition}

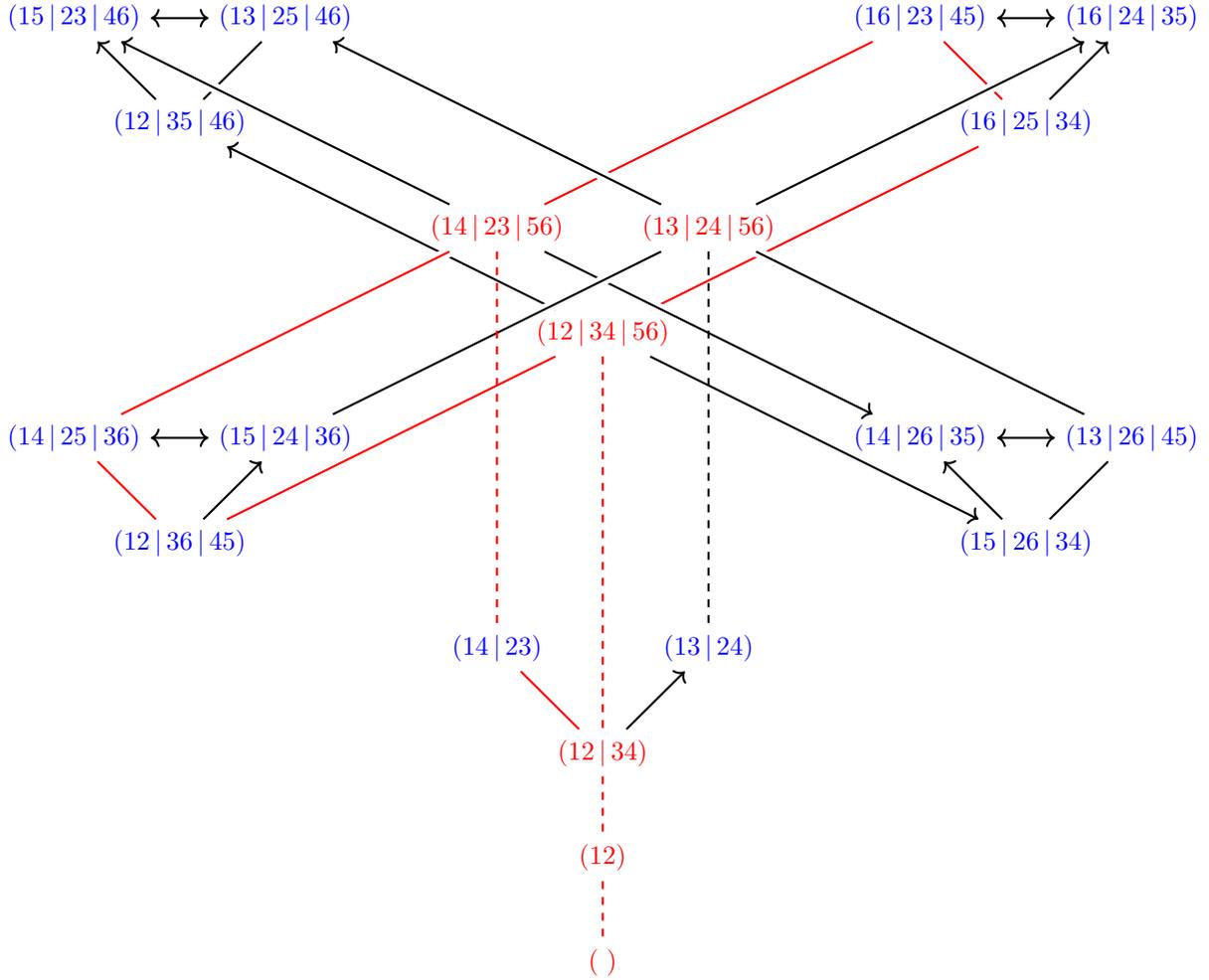
\begin{figure}[ht!]
\centering
\tikzstyle{W} = [color=white, line width=4pt]
\tikzstyle{vBlue} = [draw=white, text=blue, very thick]
\tikzstyle{vRed} = [draw=white, text=red, very thick]
\tikzstyle{eBlack} = [color = black, thick]
\tikzstyle{eRed} = [color = red, thick]
\tikzstyle{aBlack} = [->, thick]
\tikzstyle{aaBlack} = [<->, thick]
\tikzstyle{eBlackDash} = [color=black, thick, dashed]
\tikzstyle{eRedDash} = [color=red, thick, dashed]
\begin{tikzpicture}[scale=0.7]
\def\x{2}
\def\y{2}
\def\l{0.5}

\node[vRed] (empty) at (0*\x,-4*\y) { $(~)$};

\node[vRed] (12) at (0*\x,-3*\y) {$(12)$};

\node[vRed] (1234) at (0*\x,-2*\y) {$(12 \mmid 34)$};
\node[vBlue] (1423) at (-1*\x,-1*\y) {$(14 \mmid 23)$};
\node[vBlue] (1324) at (1*\x,-1*\y) {$(13 \mmid 24)$};

\node[vRed] (123456) at (0*\x,2*\y) {$(12 \mmid 34 \mmid 56)$};
\node[vBlue] (123546) at (-4*\x,4*\y) {$(12 \mmid 35 \mmid 46)$};
\node[vBlue] (152634) at (4*\x,0*\y) {$(15 \mmid 26 \mmid 34)$};
\node[vBlue] (162534) at (4*\x,4*\y) {$(16 \mmid 25 \mmid 34)$};
\node[vBlue] (123645) at (-4*\x,0*\y) {$(12 \mmid 36 \mmid 45)$};
\node[vRed] (142356) at (-1*\x,3*\y) {$(14 \mmid 23 \mmid 56)$};
\node[vBlue] (152346) at (-5*\x,5*\y) {$(15 \mmid 23 \mmid 46)$};
\node[vBlue] (142635) at (3*\x,1*\y) {$(14 \mmid 26 \mmid 35)$};
\node[vBlue] (162345) at (3*\x,5*\y) {$(16 \mmid 23 \mmid 45)$};
\node[vBlue] (142536) at (-5*\x,1*\y) {$(14 \mmid 25 \mmid 36)$};
\node[vRed] (132456) at (1*\x,3*\y) {$(13 \mmid 24 \mmid 56)$};
\node[vBlue] (132546) at (-3*\x,5*\y) {$(13 \mmid 25 \mmid 46)$};
\node[vBlue] (132645) at (5*\x,1*\y) {$(13 \mmid 26 \mmid 45)$}; 
\node[vBlue] (162435) at (5*\x,5*\y) {$(16 \mmid 24 \mmid 35)$};
\node[vBlue] (152436) at (-3*\x,1*\y) {$(15 \mmid 24 \mmid 36)$};

\draw (1423) edge [eRed] (1234);
\draw (1234) edge [aBlack] (1324);

\draw (123546) edge [aBlack] (152346);
\draw (123546) edge [eBlack] (132546);
\draw (132546) edge [aaBlack] (152346);
\draw (162534) edge [eRed] (162345);
\draw (162534) edge [aBlack] (162435);
\draw (162345) edge [aaBlack] (162435);
\draw (123645) edge [eRed] (142536);
\draw (123645) edge [aBlack] (152436);
\draw (142536) edge [aaBlack] (152436);
\draw (152634) edge [aBlack] (142635);
\draw (152634) edge [eBlack] (132645);
\draw (132645) edge [aaBlack] (142635);

\draw[W] (123456) edge (123546); \draw (123456) edge [aBlack] (123546);
\draw[W] (123456) edge (152634); \draw (123456) edge [aBlack] (152634);
\draw[W] (123456) edge (123645); \draw (123456) edge [eRed] (123645);
\draw[W] (123456) edge (162534); \draw (123456) edge [eRed] (162534);

\draw[W] (142356) edge (152346); \draw (142356) edge [aBlack] (152346); 
\draw[W] (142356) edge (142635); \draw (142356) edge [aBlack] (142635); 
\draw[W] (142356) edge (162345); \draw (142356) edge [eRed] (162345);
\draw[W] (142356) edge (142536); \draw (142356) edge [eRed] (142536);

\draw[W] (132456) edge (132546); \draw (132456) edge [aBlack] (132546);
\draw[W] (132456) edge (132645); \draw (132456) edge [eBlack] (132645);
\draw[W] (132456) edge (162435); \draw (132456) edge [aBlack](162435);
\draw[W] (132456) edge (152436); \draw (132456) edge [eBlack] (152436);

\draw[eRedDash] (12) edge (empty);

\draw[eRedDash] (1234) edge (12);

\draw[eRedDash] (123456) edge (1234);
\draw[eRedDash] (142356) edge (1423);
\draw[eBlackDash] (132456) edge (1324);
\end{tikzpicture}
\caption{The $b$-Weingarten graph $\mathcal{G}^{(b)}$, restricted to $\P_0 \sqcup \P_1 \sqcup \P_2 \sqcup \P_3$. Each solid edge represents two type $A$ edges, one in each direction. Each dashed edge represents one type $B$ edge, directed down the page. An arrowhead on an edge indicates that the weight $b$ is assigned to the corresponding directed edge, while all other directed edges are assigned weight $1$. Each blue vertex should also have a directed loop with weight $b$, although these are not depicted to avoid cluttering the diagram. The orthogonal Weingarten graph $\mathcal{G}^{\O}$ is obtained by ignoring the edge weights, while the unitary Weingarten graph $\mathcal{G}^{\mathbf{U}}$ appears as the subgraph comprising the red edges and their incident vertices.}
\label{fig:weingartengraph_B}
\end{figure}

\begin{remark}
The construction of the weight function $\omega^{(b)}$ in \cref{def:weightfunction} was motivated by the idea of interpolating between the orthogonal and unitary Weingarten graphs. However, we make no claims concerning the uniqueness of our construction, as there should be many possible weight functions that satisfy the desired properties. Furthermore, it may be the case that another construction of a suitable weight function makes the conjectures presented in \cref{sec:jucys-murphy} more apparent.
\end{remark}

In general, one recovers the large $N$ expansion of the Weingarten function from the enumeration of paths in the respective Weingarten graph. In the context of the $b$-Weingarten graph, we can form the analogue of the Weingarten function in the following way, even though we do not have a definition via integration with respect to Haar measure, as is usually the case in Weingarten calculus.

\begin{definition}
Define the {\em $b$-Weingarten function} $\wgb: \P_\bullet \to \mathbb{R}[b][[N^{-1}]]$ by the large $N$ expansion
\[
\wgb(\m) = \sum_{\bm{\rho} \in \mathrm{Path}^{(b)}(\m)} \left(-\frac{1}{N} \right)^{\ell_A(\bm{\rho})} \left( \frac{1}{N} \right)^{\ell_B(\bm{\rho})} w(\bm{\rho}), 
\]
where the summation is over all paths in $\mathcal{G}^{(b)}$ from $\m$ to the empty pair partition $(~)$. We use $\ell_A(\bm{\rho})$ to denote the number of type $A$ edges, $\ell_B(\bm{\rho})$ to denote the number of type $B$ edges, and $w(\bm{\rho})$ to denote the product of the weights of the directed edges in the path $\bm{\rho}$.
\end{definition}

From the adjacency structure of the graph $\mathcal{G}^{(b)}$, we obtain the following orthogonality relations for the $b$-Weingarten function. These simply arise from considering each path $\bm{\rho}$ from $\m$ to $(~)$ as an initial edge $\m \to \m'$ followed by a path from $\m'$ to $(~)$.

\begin{proposition}[Orthogonality relations for $\wgb$] \label{prop:orthogonality_B}
For each non-empty pair partition $\m \in \P_k$, the $b$-Weingarten function $\wgb$ satisfies the relation
\[
\wgb(\m) = -\frac{1}{N} \sum_{i=1}^{2k-2} \omega^{(b)}(\m,(i~2k-1) \cdot \m) \, \wgb((i~2k-1) \cdot \m) + \frac{1}{N} \delta_{\{2k-1,2k\} \in \m} \wgb(\m^\downarrow).
\]
\end{proposition}

The monotone Hurwitz numbers arise from the enumeration of paths in the unitary Weingarten graph~$\mathcal{G}^{\mathbf{U}}$. It is natural to surmise that the $b$-monotone Hurwitz numbers introduced by Bonzom, Chapuy and Do\l{}\k{e}ga via Jack functions and discussed in \cref{subsec:bmonotone} may similarly arise from the $b$-Weingarten graph $\mathcal{G}^{(b)}$. We will show that this is indeed the case by verifying that the Virasoro constraints for $b$-monotone Hurwitz numbers and the orthogonality relations for the $b$-Weingarten function are essentially equivalent.

\begin{proposition} \label{prop:expansion_B}
For $\m \in \P_k$, the $b$-Weingarten function $\wgb$ has the large $N$ expansion
\[
\wgb(\m) = \left( \frac{1}{N}\right)^k \sum_{r=0}^\infty \hb_r(\lambda(\m)) \left(-\frac{1}{N} \right)^r,
\]
where $\hb_r$ denotes the $b$-monotone Hurwitz number defined by \cref{eq:bmonotonedisc} and $\lambda(\m)$ is the coset-type of the pair partition $\m$.
\end{proposition}

\begin{proof}
For $\m \in \P_k$, the level structure of the $b$-Weingarten graph ensures that we can write
\[
\wgb(\m) = \left( \frac{1}{N}\right)^k \sum_{r=0}^\infty \wgb_r(\m) \left(-\frac{1}{N} \right)^r,
\]
for some coefficients $\wgb_r(\m)$. So we want to prove that for every pair partition $\m$ and $r\geq 0$, we have $\wgb_r(\m) = \hb_r(\lambda(\m))$, where the monotone Hurwitz number $\hb_r(\lambda)$ is defined by \cref{eq:bmonotonedisc}.

We proceed by induction on $\ell = r + k$. The base case $\ell = 0$ is simply the equality
\[
\wgb_0((~)) = 1 = \hb_0((~)).
\]

Now let $\ell \geq 0$ and suppose that for $r, k \geq 0$ with $r+k = \ell$, we have $\wgb_r(\m) = \hb_r(\lambda(\m))$, for all $\m \in \P_k$. Then consider $r, k \geq 0$ with $r+k = \ell+1$ and let $\m \in \P_k$. Suppose that $\lambda(\m) = (\lambda_1, \lambda_2, \ldots, \lambda_n)$, where $2k-1$ is contained in a cycle of length $2\lambda_1$ in the graph $\Gamma(\m)$. Here, we allow the parts of $\lambda(\m)$ to be in any order, not necessarily weakly decreasing.

Express the orthogonality relations of \cref{prop:orthogonality_B}
at the level of the coefficients $\wgb_r(\m)$ and apply the inductive hypothesis to obtain
\begin{align*}
\wgb_r(\m) &= \sum_{i=1}^{2k-2} \omega^{(b)}(\m,(i~2k-1) \cdot \m) \, \wgb_{r-1}((i~2k-1) \cdot \m) + \delta_{\{2k-1,2k\} \in \m} \wgb_r(\m^\downarrow)\\
&= \sum_{i=1}^{2k-2} \omega^{(b)}(\m,(i~2k-1) \cdot \m) \, \hb_{r-1}(\lambda((i~2k-1) \cdot \m)) + \delta_{\lambda_1,1} \, \hb_r(\lambda_S).
\end{align*}
Here, we set $S = \{2, 3, \ldots, n\}$ and $\lambda_I = (\lambda_{i_1}, \lambda_{i_2}, \ldots \lambda_{i_k})$ for $I = \{i_1, i_2, \ldots, i_k\}$.

Now evaluate the coset-type $\lambda((i~2k-1) \cdot \m)$ and weight $\omega^{(b)}(\m,(i~2k-1) \cdot \m)$ for $i = 1, 2, \ldots 2k-2$. First, construct the graph $\Gamma(\m)$ and consider the charges on its vertices, as per the construction of $\omega^{(b)}$ in \cref{def:weightfunction}. There are three possibilities, depending on the value of $i$.

\begin{itemize}
\item {\em Cut}. If $i$ is in the same cycle as $2k-1$ and is an even distance $2\alpha$ from $2k-1$ in $\Gamma(\m)$, the cycle will be cut and $(i~2k-1) \cdot \m$ will have coset-type $(\alpha, \lambda_1-\alpha, \lambda_S)$. In this case, $\omega^{(b)}(\m,(i~2k-1) \cdot \m) = 1$, so these terms contribute
\[
\sum_{\alpha=1}^{\lambda_1-1} \hb_{r-1}(\alpha,\lambda_1-\alpha, \lambda_S).
\]

\item {\em Flip}. If $i$ is in the same cycle as $2k-1$ and is an odd distance from $2k-1$ in $\Gamma(\m)$, the cycle will remain the same length and $(i~2k-1) \cdot \m$ will have coset-type $\lambda$. In this case, $\omega^{(b)}(\m,(i~2k-1) \cdot \m) = b$, so these terms contribute
\[
b(\lambda_1-1) \, \hb_{r-1}(\lambda).
\]
Note that the prefactor is $\lambda_1-1$ and not $\lambda_1$, as $2k-1$ is necessarily in the same cycle as $2k$ and we do not have a term corresponding to $i = 2k$.

\item {\em Join}. If $i$ is in a distinct cycle from $2k-1$ of length $2\lambda_j$, the two cycles will be joined and $(i~2k-1) \cdot \m$ will have coset-type $(\lambda_1+\lambda_j, \lambda_{S \setminus \{j\}})$. The weight $\omega^{(b)}(\m,(i~2k-1) \cdot \m)$ will depend on the charge of $i$. However, summing over all $2\lambda_j$ values of $i$ in the cycle, one obtains $\lambda_j$ values of $i$ with positive charge and $\lambda_j$ values of $i$ with negative charge. So these terms contribute
\[
\sum_{j=2}^n (1+b) \lambda_j \, \hb_{r-1}(\lambda_1+\lambda_j,\lambda_{S \setminus \{j\}}).
\]
\end{itemize}

So we have shown the equality
\[
\wgb_r(\m) = \sum_{\alpha=1}^{\lambda_1-1} \hb_{r-1}(\alpha, \lambda_1-\alpha, \lambda_S) + (1+b) \sum_{j=2}^n \lambda_j \, \hb_{r-1}(\lambda_1+\lambda_j, \lambda_{S \setminus \{j\}}) + b(\lambda_1-1) \, \hb_{r-1}(\lambda) + \delta_{\lambda_1,1} \, \hb_r(\lambda_S).
\]

On the other hand, one can apply the Virasoro constraint $L_{\lambda_1}^{(b)} \cdot Z^{(b)} = 0$ of \cref{thm:virasoro_b}, using the expression of \cref{eq:bmonotonedisc} for the partition function, and extract the coefficient of $\h^{r-1} p_{\lambda_2} \cdots p_{\lambda_n}$ to obtain
\[
\hb_r(\lambda) = \sum_{\alpha=1}^{\lambda_1-1} \hb_{r-1}(\alpha, \lambda_1-\alpha, \lambda_S) + (1+b) \sum_{j=2}^n \lambda_j \, \hb_{r-1}(\lambda_1+\lambda_j, \lambda_{S \setminus \{j\}}) + b(\lambda_1-1) \, \hb_{r-1}(\lambda) + \delta_{\lambda_1,1} \, \hb_r(\lambda_S).
\]

Thus, we immediately obtain the equality $\wgb_r(\m) = \hb_r(\lambda)$, which completes the induction.
\end{proof}

The previous result shows that the $b$-monotone Hurwitz numbers, originally defined by Bonzom, Chapuy and Do\l{}\k{e}ga via Jack functions, can be expressed as a weighted enumeration of paths in the $b$-Weingarten graph $\mathcal{G}^{(b)}$. We have a bijection
\[
\mathrm{F}: \mathrm{Path}^{(b)}(\m) \to \mathrm{Mono}(\m),
\]
between paths from $\m$ to $(~)$ in $\mathcal{G}^{(b)}$ and monotone factorisations of the pair partition $\m$. Since the weight of a path in $\mathcal{G}^{(b)}$ is a power of $b$, we can make the following definition.

\begin{definition} \label{def:flip}
Define the {\em flip number} of a monotone factorisation $\bm{\tau}$ of a pair partition $\m$ to be the unique integer $\mathrm{flip}(\bm{\tau})\geq 0$ such that $w(\bm{\rho}) = b^{\mathrm{flip}(\bm{\tau})}$, where $\bm{\rho} = \mathrm{F}^{-1}(\bm{\tau})$ is the unique path in $\mathcal{G}^{(b)}$ corresponding to~$\bm{\tau}$. 
\end{definition}

From this definition, it is clear that the $b$-Weingarten function can be expanded for $\m \in \P_k$ as 
\[
\wgb(\m) = \left( \frac{1}{N}\right)^k \sum_{r=0}^\infty \left(-\frac{1}{N} \right)^r \sum_{\substack{\bm{\tau}\in \mathrm{Mono}(\m) \\ \ell(\bm{\tau}) = r}} b^{\mathrm{flip}(\bm{\tau})}.
\]
By comparing this equation with \cref{prop:expansion_B}, the $b$-monotone Hurwitz numbers may be realised in the following way.

\begin{proposition} \label{prop:bmonotoneinterpretation}
For $\lambda$ a partition and $r \geq 0$, the $b$-monotone Hurwitz numbers satisfy
\[
\hb_r(\lambda) = \sum_{\substack{\bm{\tau}\in \mathrm{Mono}(\m)\\ \ell(\bm{\tau}) = r}} b^{\mathrm{flip}(\bm{\tau})},
\]
where $\m$ is any pair partition of coset-type $\lambda$. It follows that the ``connected'' $b$-monotone Hurwitz numbers defined by \cref{eq:bmonotone} satisfy 
\[
\Hb_{g,n}(\lambda) = \frac{1}{\prod \lambda_i} \sum_{\substack{\bm{\tau}\in \mathrm{CMono}(\m)\\ \ell(\bm{\tau}) = |\lambda|+2g-2+n}} b^{\mathrm{flip}(\bm{\tau})}.
\]
(Recall from \cref{def:factorisation} that $\mathrm{CMono}(\m)$ denotes the set of connected monotone factorisations of $\m$.)
\end{proposition}

A combinatorial interpretation for the $b$-monotone Hurwitz numbers was already provided by Bonzom, Chapuy and Do\l{}\k{e}ga in terms of monotone Hurwitz graphs enumerated with weight $b$ to the power of a certain ``measure of non-orientability''~\cite{bon-cha-dol23}. \cref{prop:bmonotoneinterpretation} provides a closely related, yet in some sense distinct, combinatorial interpretation in terms of monotone factorisations and the notion of ``flip number''.

\begin{remark}
It is not clear to the authors whether $\mathrm{flip}(\bm{\tau})$ can be directly and explicitly determined from an arbitrary monotone sequence of transpositions $\bm{\tau}$. A naive approach to calculating $\mathrm{flip}(\bm{\tau})$ would be to inductively apply the relation
\[
b^{\mathrm{flip}(\bm{\tau})} = \omega^{(b)}(\m, \tau_r \cdot \m) \, b^{\mathrm{flip}(\bm{\tau}')},
\]
where $\bm{\tau} = (\tau_1, \tau_2, \ldots, \tau_r) \in \mathrm{Mono}(\m)$ and $\bm{\tau}' = (\tau_1, \tau_2, \ldots, \tau_{r-1}) \in \mathrm{Mono}(\tau_r \cdot \m)$. This ``local'' approach to calculating $\mathrm{flip}(\bm{\tau})$ is effective yet only recursive, and does not preclude the existence of a more ``global'' approach to directly determine $\mathrm{flip}(\bm{\tau})$.
\end{remark}

\subsection{The \texorpdfstring{$bt$-}{bt-}monotone Hurwitz numbers} \label{subsec:weingarten_BT}

In previous sections, we have discussed two distinct deformations of the monotone Hurwitz numbers.

\begin{itemize}
\item The $b$-deformation of the monotone Hurwitz numbers is inspired by the combinatorics of Jack functions, as described in \cref{subsec:bmonotone}. The $b$-parameter records a certain ``measure of non-orientability'' for monotone Hurwitz maps~\cite{bon-cha-dol23}, or as described in \cref{subsec:weingarten_B}, a ``flip number'' for monotone factorisations of pair partitions. In \cref{sec:jucys-murphy}, we extend the notion of $b$-deformation to the Jucys--Murphy elements in the symmetric group algebra.

\item The $t$-deformation of the monotone Hurwitz numbers was originally introduced in previous work of the authors and Moskovsky by passing from Weingarten calculus for unitary groups to complex Grassmannians~\cite{cou-do-mos23}. An analogous story was told in \cref{sec:weingarten-A} in the context of Weingarten calculus for real Grassmannians. In both the previous and present works, the $t$-parameter records the ``hive number'' of a monotone sequence of transpositions.
\end{itemize}

It is natural to consider the common generalisation of these two deformations, in the context of Weingarten calculus and monotone Hurwitz numbers. The upshot of this section is that these deformations are compatible in the sense that one can carry both parameters in a natural way. We present the definitions without motivation and the results without proof, unless there are subtleties or differences from those of previous sections.

\begin{definition}
The $bt$-Weingarten graph $\mathcal{G}^{(bt)}$ is the infinite weighted directed graph with vertex set $\P_\bullet$ and edge set $E_A \sqcup E_B \sqcup E_C$, where
\begin{itemize}
\item the set $E_A$ comprises ``type $A$'' edges, which are of the form $\m \longrightarrow (i~2k-1) \cdot \m$ for $\m \in \P_k$ and $1 \leq i \leq 2k-2$, where the weight is $\omega^{(b)}(\m, (i~2k-1) \cdot \m)$,
\item the set $E_B$ comprises ``type $B$'' edges, which are of the form $\m \longrightarrow \m^\downarrow$ for $\m \in \P_k$ with $\{2k-1, 2k\} \in \m$, where the weight is $1$,
\item the set $E_C$ comprises ``type $C$'' edges, which are of the form $\m \longrightarrow [(i~2k-1) \cdot \m]^{\downarrow})$ for $\m \in \P_k$ with $\{i,2k\} \in \m$, where the weight is $1$.
\end{itemize}
In short, the graph $\mathcal{G}^{(bt)}$ is simply the $t$-deformed orthogonal Weingarten graph $\mathcal{G}^{\mathbf{A}}$ with type $A$ directed edges weighted by the function $\omega^{(b)}$.
\end{definition}

Perhaps the only subtlety here is that the type $C$ edges carry the weight $1$ rather than depending on the parameter $b$. That this is the correct choice of weight will not become apparent until the proof of \cref{prop:btmonotonecombinatorics}.

\begin{definition} \label{def:btweingarten}
Define the {\em $bt$-Weingarten function} $\wgbt: \P_\bullet \to \mathbb{R}[b, M][[N^{-1}]]$ by the large $N$ expansion
\[
\wgbt(\m) = \sum_{\bm{\rho} \in \mathrm{Path}^{(b)}(\m)} \left(-\frac{1}{N} \right)^{\ell_A(\bm{\rho})} \left( \frac{M}{N} \right)^{\ell_B(\bm{\rho})} \left( \frac{1}{N} \right)^{\ell_C(\bm{\rho})} w(\bm{\rho}), 
\]
where the summation is over all paths in $\mathcal{G}^{(bt)}$ from $\m$ to the empty pair partition $(~)$. We use $\ell_K(\bm{\rho})$ to denote the number of type $K$ edges for $K \in \{A, B, C\}$ and $w(\bm{\rho})$ to denote the product of the weights of the directed edges in the path $\bm{\rho}$.
\end{definition}

Again, from the adjacency structure of the graph $\mathcal{G}^{(bt)}$, we obtain the following orthogonality relations for the $bt$-Weingarten function.

\begin{proposition}[Orthogonality relations for $\wgbt$] \label{prop:orthogonality_BT}
For each non-empty pair partition $\m \in \P_k$, the $bt$-Weingarten function $\wgbt$ satisfies the relation
\begin{align*}
\wgbt(\m) ={}& -\frac{1}{N} \sum_{i=1}^{2k-2} \omega^{(b)}(\m,(i~2k-1) \cdot \m) \, \wgbt((i~2k-1) \cdot \m) + \frac{M}{N} \delta_{\{2k-1,2k\} \in \m} \wgbt(\m^\downarrow) \\
&+ \frac{1}{N} \sum_{i=1}^{2k-2} \delta_{\{i,2k\} \in \m} \wgbt([(i~2k-1) \cdot \m]^{\downarrow}).
\end{align*}
\end{proposition}

\begin{example}
From the base case $\wgbt((~)) = 1$, one immediately obtains $\wgbt((12)) = \frac{M}{N}$. The three orthogonality relations for $\wgbt(\m)$, where $\m \in \P_2$, read
\begin{align*}
\wgbt((12 \mmid 34)) &= -\frac{1}{N} \wgbt((14 \mmid 23)) - \frac{b}{N} \wgbt((13 \mmid 24)) + \frac{M}{N} \wgbt((12)) \\
\wgbt((14 \mmid 23)) &= -\frac{1}{N} \wgbt((12 \mmid 34)) - \frac{b}{N} \wgbt((14 \mmid 23)) + \frac{1}{N} \wgbt((12)) \\
\wgbt((13 \mmid 24)) &= -\frac{b}{N} \wgbt((13 \mmid 24)) -\frac{1}{N} \wgbt((12 \mmid 34)) + \frac{1}{N} \wgbt((12)).
\end{align*}
After substituting $\wgbt((12)) = \frac{M}{N}$, one obtains a non-degenerate linear system with solution
\[
\begin{aligned}
\wgbt((12 \mmid 34)) &= \frac{M(MN + bM - b-1)}{N (N+b+1) (N-1)} \\
\wgbt((14 \mmid 23)) &= \wgbt((13 \mmid 24)) = \frac{M(N-M)}{N(N+b+1) (N-1)}.
\end{aligned}
\]
In general, the orthogonality relations provide a linear system from which one can deduce any value of the Weingarten function from the base case $\wgbt((~)) = 1$. It follows that each such value is in fact a rational function of $N$, from which the corresponding series in $N^{-1}$ is obtained as a large $N$ expansion.
\end{example}

By direct analogy with \cref{eq:bmonotone,eq:bmonotonedisc} for the $b$-monotone Hurwitz numbers, introduce the partition function
\begin{align} \label{eq:btmonotone}
Z^{(bt)}(\bm{p}; \h, z) &= \sum_{k \geq 0} \frac{z^k}{\h^k} \sum_{\lambda \vdash k} \Bigg( \prod_{\Box \in \lambda} \frac{1- (1-t) \h c_b(\Box)}{1 - \h c_b(\Box)} \Bigg) \, \frac{J^{(b)}_\lambda(\bm{p})}{\mathrm{hook}_b(\lambda) \, \mathrm{hook}_b'(\lambda)} \notag \\
&= \exp \Bigg[ \sum_{k \geq 0} z^k \sum_{g \geq 0} \sum_{n \geq 1} \frac{\h^{2g-2+n}}{n!} \sum_{\mu_1, \ldots, \mu_n \geq 1} \Hbt_{g,n}(\mu_1, \ldots, \mu_n) \, \frac{p_{\mu_1} \cdots p_{\mu_n}}{(b+1)^n} \Bigg] \notag \\
&= \sum_{k \geq 0} \frac{z^k}{\h^k} \sum_{r \geq 0} \sum_{n \geq 1} \frac{\h^r}{n!} \sum_{\mu_1, \ldots, \mu_n \geq 1} \frac{\hbt_r(\mu_1, \ldots, \mu_n)}{\mu_1 \cdots \mu_n} \, \frac{p_{\mu_1} \cdots p_{\mu_n}}{(b+1)^n}.
\end{align}
This defines the $bt$-monotone Hurwitz numbers $\Hbt_{g,n}(\mu_1, \ldots, \mu_n)$ and $\hbt_r(\mu_1, \ldots, \mu_n)$, where we consider the latter a ``disconnected'' renormalised version of the former. In the summation over $g$ in the second line, one must allow $g$ to range over $\frac{1}{2} \mathbb{N} = \{0, \frac{1}{2}, 1, \frac{3}{2}, \ldots\}$, as in the case of \cref{eq:bmonotone} for the $b$-monotone Hurwitz numbers. See \cref{tab:bthurwitz1,tab:bthurwitz2,tab:bthurwitz3} for some examples of $bt$-monotone Hurwitz numbers.

We will prove that the partition function satisfies Virasoro constraints, which in some sense, align with the orthogonality relations for $\wgbt$. As a result, we deduce that the $bt$-Weingarten function $\wgbt$ stores $bt$-monotone Hurwitz numbers in its large $N$ expansion.

\begin{theorem}[Virasoro constraints for $Z^{(bt)}$] \label{thm:virasoro_bt}
The partition function $Z^{(bt)}(\bm{p}; \h, z)$ satisfies $L_m^{(bt)} \cdot Z^{(bt)} = 0$, for $m = 1, 2, 3, \ldots$, where
\[
L_m^{(bt)} = \frac{m \partial_m}{\h} - (1+b) \sum_{i + j = m} ij \partial_i \partial_j - \sum_{i} (i + m) p_i \partial_{i+m} - b m (m-1) \partial_m - \frac{z}{\h} (t-1) (m-1) \partial_{m-1} - \delta_{m,1} \frac{z}{\h^2(1+b)}.
\]
Here, we write $\partial_i$ as a shorthand for the differential operator $\frac{\partial}{\partial p_i}$, for $i \geq 1$. These constraints uniquely determine $Z^{(bt)}$ from the initial condition $\left. Z^{(bt)} \right|_{z=0} = 1$. Moreover, the operators $L_1^{(bt)}, L_2^{(bt)}, L_3^{(bt)}, \ldots$ satisfy the Virasoro commutation relations
\[
[L_m^{(bt)}, L_n^{(bt)}] = (m-n) \, L_{m+n}^{(bt)}, \qquad \text{for } m, n \geq 1.
\]
\end{theorem}

\begin{proof}
First, observe that the Virasoro commutation relation $[L_m^{(bt)}, L_n^{(bt)}] = (m-n) \, L_{m+n}^{(bt)}$ can be checked directly using a relatively straightforward brute force calculation.

The remainder of the proof can be divided into two parts. The first part of the proof is a direct application of \cite[Theorem~2.7]{CDO24a}, which gives a partial differential equation satisfied by the partition functions for the broad class of $b$-deformations of weighted Hurwitz numbers. In the case of $bt$-monotone Hurwitz numbers, this result asserts that $Z^{(bt)}$ is the unique solution of the equation
\begin{equation} \label{eq:evolution_bt}
\sum_{n \geq 1} p_n L_n \cdot Z^{(bt)} = 0,
\end{equation}
with initial condition $\left. Z^{(bt)} \right|_{z=0} = 1$.

The second part of the proof uses an argument based on the proof of~\cite[Lemma~2.6]{bon-cha-dol23}. It will be convenient for us to define $\tilde{L}_m^{(bt)} = \h L_m^{(bt)}$ and $x = \frac{z}{\h}$. We use \cref{eq:evolution_bt} to show by induction on $\ell$ that $[x^k \h^r] \, \tilde{L}_m^{(bt)} \cdot Z^{(bt)} = 0$ for any integers $k, r \geq 0$ with $k + r = \ell$ and for all $m \geq 1$, from which the desired claim follows.

Note that the operator $\tilde{L}_m^{(bt)}$ may be written as
\begin{equation} \label{eq:Lexpansion}
\tilde{L}_m^{(bt)} = m \partial_m + O(x,\h),
\end{equation}
where we use $O(x,\h) \in \mathbb{C}(b)[\bm{p}, \bm{\partial}, \h, x]$ to denote a sum of monomials of degree at least $1$ in $x$ or $\h$. As usual, we use the notation $\bm{p} = (p_1, p_2, p_3, \ldots)$ and we introduce the notation $\bm{\partial} = (\frac{\partial}{\partial p_1}, \frac{\partial}{\partial p_2}, \frac{\partial}{\partial p_3}, \ldots)$. The base case $\ell = 0$ is satisfied since $\left. Z^{(bt)} \right|_{z=0} = 1$ implies that for all $m \geq 1$,
\[
[x^0\h^0] \, \tilde{L}_m^{(bt)} \cdot Z^{(bt)} = [z^0\h^0] \, \tilde{L}_m^{(bt)} \cdot Z^{(bt)} = m \partial_m \cdot \left( [z^0\h^0] \, Z^{(bt)} \right) = 0.
\]

Now consider $\ell \geq 1$ and let $k, r \geq 0$ with $k + r = \ell$. Assume for the sake of induction that $[x^k \h^r] \, \tilde{L}_m \cdot Z^{(bt)} = 0$ for any $k, r \geq 0$ with $k + r < \ell$ and for all $m \geq 1$. Let $\mathcal{L}^{\geq 1}$ denote the space spanned by operators of the form $f \tilde{L}_m^{(bt)}$, where $m \geq 1$ and $f \in \mathbb{C}(b)[\bm{p}, \bm{\partial}, \h, x]$ is any monomial of degree at least $1$ in $x$ or $\h$. By the inductive hypothesis, we have $[x^k \h^r] \, X \cdot Z^{(bt)} = 0$ for any $X \in \mathcal{L}^{\geq 1}$.

For $m \geq 1$, consider the action of the commutator $\left[ \sum p_n \tilde{L}_n^{(bt)}, p_m \tilde{L}_m^{(bt)} \right]$ on the partition function $Z^{(bt)}$. On one hand, applying \cref{eq:evolution_bt} leads to
\[
\left[\sum_{n \geq 1} p_n \tilde{L}_n^{(bt)}, p_m \tilde{L}_m^{(bt)} \right] \cdot Z^{(bt)} = \sum_{n \geq 1} p_n \tilde{L}_n^{(bt)} p_m \tilde{L}_m^{(bt)} \cdot Z^{(bt)} = \left[ \sum_{n \geq 1} n p_n \partial_n p_m \tilde{L}_m^{(bt)} + X \right] \cdot Z^{(bt)},
\]
for some $X \in \mathcal{L}^{\geq 1}$. Thus, invoking the inductive hypothesis leads to
\[
[x^k\h^r]\left[\sum_{n \geq 1} p_n \tilde{L}_n^{(bt)}, p_m \tilde{L}_m^{(bt)} \right] \cdot Z^{(bt)} = [x^k\h^r] \left( \sum_{n \geq 1} n p_n \partial_n \right) p_m \tilde{L}_m^{(bt)} \cdot Z^{(bt)} = [x^k\h^r] \, k p_m \tilde{L}_m^{(bt)} \cdot Z^{(bt)}.
\]
The final equality here uses the fact that the power of $x$ in $Z^{(bt)}$ records the weighted homogeneous degree in $\bm{p}$, while the operator $p_m \tilde{L}_m^{(bt)}$ preserves the degree.

On the other hand, one can expand the commutator $\left[ \sum p_n \tilde{L}_n^{(bt)}, p_m \tilde{L}_m^{(bt)} \right]$ as
\begin{align*}
& \left[ \sum_{n \geq 1} p_n \tilde{L}_n^{(bt)}, p_m \tilde{L}_m^{(bt)} \right] = \sum_{n \geq 1} \left( p_n \left[ \tilde{L}_n^{(bt)}, p_m \right] \tilde{L}_m^{(bt)} + p_m \left[ p_n, \tilde{L}_m^{(bt)} \right] \tilde{L}_n^{(bt)} + p_n p_m \left[ \tilde{L}_n^{(bt)}, \tilde{L}_m^{(bt)} \right] \right) \\
&\qquad\qquad = \sum_{n \geq 1} \left( p_n \left[ n \partial_n + O(x, \h), p_m \right] \tilde{L}_m^{(bt)} + p_m \left[ p_n, m \partial_m + O(x, \h) \right] \tilde{L}_n^{(bt)} + p_n p_m (n-m) \h \tilde{L}_{n+m}^{(bt)} \right).
\end{align*}
Here, we have used \cref{eq:Lexpansion} on the first two terms and the Virasoro commutation relation on the third term. One can now recognise that the entire expression is an element of $\mathcal{L}^{\geq 1}$. The terms involving $O(x, \h)$ and the final term are evidently elements of $\mathcal{L}^{\geq 1}$, while the remaining terms actually cancel out to zero.

Finally, putting these pieces together yields the equality
\[
0 = [x^k\h^r] \left[ \sum_{n \geq 1} p_n \tilde{L}_n^{(bt)}, p_m \tilde{L}_m^{(bt)} \right] \cdot Z^{(bt)} = [x^k\h^r] \, k p_m \tilde{L}_m^{(bt)} \cdot Z^{(bt)}.
\]
If $k \neq 0$, we deduce from this equation that $[x^k\h^r] \, \tilde{L}_m^{(bt)} \cdot Z^{(bt)} = 0$, as desired. If $k = 0$, we use the initial condition $\left. Z^{(bt)} \right|_{z=0} = 1$ to deduce that $[x^k\h^r] \, \tilde{L}_m^{(bt)} \cdot Z^{(bt)} = 0$. This completes the induction, so the desired claim follows.
\end{proof}

At the level of the $bt$-monotone Hurwitz numbers, the Virasoro constraints of \cref{thm:virasoro_bt} imply the following generalisation of the cut-join recursion for the usual monotone Hurwitz numbers~\cite{gou-gua-nov13}. We call it the ``cut-join-flip recursion'' as it can be alternatively proved using the interpretation for the $bt$-monotone Hurwitz numbers of \cref{prop:btmonotonecombinatorics} as weighted enumerations of monotone factorisations. In that case, the recursion arises by considering the behaviour of a transposition on the coset-type of a pair partition, which falls into one of three types.

\begin{corollary}[Cut-join-flip recursion] \label{cor:cjfrecursion}
Apart from the initial condition $\Hbt_{0,1}(1) = 1$, the $bt$-monotone Hurwitz numbers satisfy 
\begin{align*}
\mu_1 \Hbt_{g,n}(\mu_1, \mu_S) &= (b+1) \sum_{i=2}^n (\mu_1+\mu_i) \Hbt_{g,n-1}(\mu_1+\mu_i, \mu_{S \setminus \{i\}}) \\
&+ \sum_{\alpha + \beta = \mu_1} \alpha \beta \left[ \Hbt_{g-1,n+1}(\alpha, \beta, \mu_S) + \sum_{\substack{g_1+g_2=g \\ I_1 \sqcup I_2 = S}} \Hbt_{g_1,|I_1|+1}(\alpha, \mu_{I_1}) \, \Hbt_{g_2,|I_2|+1}(\beta, \mu_{I_2}) \right] \\
&+ b \mu_1 (\mu_1-1) \Hbt_{g-\frac{1}{2},n}(\mu_1, \mu_S) + (t-1) (\mu_1-1) \Hbt_{g,n}(\mu_1-1, \mu_S),
\end{align*}
where we set $S = \{2, 3, \ldots, n\}$ and $\mu_I = \{\mu_{i_1}, \mu_{i_2}, \ldots, \mu_{i_k}\}$ for $I = \{i_1, i_2. \ldots, i_k\}$. In the summation over $g_1 + g_2 = g$, we allow $g_1$ and $g_2$ to range over $\frac{1}{2} \mathbb{N} = \{0, \frac{1}{2}, 1, \frac{3}{2}, \ldots\}$.
\end{corollary}

We now apply the methods of \cref{subsec:weingarten_B} to realise the coefficients of the partition function $Z^{(bt)}$ as enumerations of monotone factorisation of pair partitions, weighted by certain monomials in $b$ and $t$.

\begin{proposition}[Large $N$ expansion]\label{prop:expansion_BT}
For $\m \in \P_k$, the $bt$-Weingarten function $\wgbt$ has the large $N$ expansion
\[
\wgbt(\m) = \frac{1}{(1-t)^k} \sum_{r=0}^\infty \hbt_r(\lambda(\m)) \left(-\frac{1}{N} \right)^r,
\]
where $\hbt_r$ denotes the $bt$-monotone Hurwitz number defined by \cref{eq:btmonotone}, $\lambda(\m)$ is the coset-type of the pair partition $\m$, and $t = 1 - \frac{N}{M}$.
\end{proposition}

\begin{proof}
The method of proof is identical to that of \cref{prop:expansion_B}, inductively equating the coefficients of the $bt$-Weingarten function with the coefficients of $Z^{(bt)}$ via the orthogonality relations of \cref{prop:orthogonality_BT} and \cref{thm:virasoro_bt}. There is just one additional observation required: in the case when $\{i,2k\} \in \m$, we have $\lambda ([(i~2k-1) \cdot \m]^{\downarrow}) = (\lambda_1-1,\lambda_I)$, where $\lambda(\m) = (\lambda_1, \lambda_2, \ldots, \lambda_n)$ and $2k$ is assumed to be in the cycle of length $2\lambda_1$.
\end{proof}

As one would expect, the $bt$-monotone Hurwitz numbers defined by the partition function of \cref{eq:btmonotone} can be interpreted as weighted enumerations of monotone factorisations.

\begin{proposition} \label{prop:btmonotonecombinatorics}
For $\lambda$ a partition and $r \geq 0$, the $bt$-monotone Hurwitz numbers satisfy
\[
\hbt_r(\lambda) = \sum_{\substack{\bm{\tau} \in \mathrm{Mono}(\m) \\ \ell(\bm{\tau}) = r}} b^{\mathrm{flip}(\bm{\tau})} t^{\mathrm{hive}(\bm{\tau})}, 
\]
where $\m$ is any pair partition of coset-type $\lambda$. It follows that the ``connected'' $bt$-monotone Hurwitz numbers satisfy 
\[
\Hbt_{g,n}(\lambda) = \frac{1}{\prod \lambda_i} \sum_{\substack{\bm{\tau}\in \mathrm{CMono}(\m) \\ \ell(\bm{\tau}) = |\lambda|+2g-2+n}} b^{\mathrm{flip}(\bm{\tau})} t^{\mathrm{hive}(\bm{\tau})}.
\]
\end{proposition}

\begin{proof}
The recording of transpositions along a path in the Weingarten graph $\mathcal{G}^{(bt)}$ gives rise to a map
\[
\mathrm{F} : \mathrm{Path}^{(bt)}(\m) \to \mathrm{Mono}(\m),
\]
for each pair partition $\m$. As in the case of the $t$-deformed orthogonal Weingarten graph discussed in \cref{subsec:expansion_A}, this construction does not provide a bijection and we have the generating function identity
\begin{equation} \label{eq:gfidentity}
\sum_{\bm{\rho} \in F^{-1}(\bm{\tau})} x^{\ell_A(\bm{\rho})} y^{\ell_B(\bm{\rho})} z^{\ell_C(\bm{\rho})} = x^{\ell(\bm{\tau})} y^k (1 + x^{-1} y^{-1} z)^{\mathrm{hive}(\bm{\tau})}.
\end{equation}

By aligning \cref{def:btweingarten} for the $bt$-Weingarten function with its expression from \cref{prop:expansion_BT}, and using $\ell_A(\bm{\rho}) + \ell_C(\bm{\rho}) = r$ and $\ell_B(\bm{\rho}) + \ell_C(\bm{\rho}) = k$, it suffices to show that 
\[
\sum_{\bm{\rho} \in F^{-1}(\bm{\tau})} \left( -\frac{N}{M} \right)^{\ell_C(\bm{\rho})} w(\bm{\rho}) = b^{\mathrm{flip}(\bm{\tau})} t^{\mathrm{hive}(\bm{\tau})},
\]
for any monotone factorisation $\bm{\tau}$ of $\m \in \P_k$, where $t = 1 - \frac{N}{M}$. However, this would follow directly from \cref{eq:gfidentity}, as long as $w(\bm{\rho}) = b^{\mathrm{flip}(\bm{\tau})}$ for all $\bm{\rho} \in \mathrm{F}^{-1}(\bm{\tau})$. We devote the remainder of the proof to showing that this is indeed the case.

Recall that there is a unique path $\bm{\rho}_{AB} \in F^{-1}(\bm{\tau})$ comprising only edges of types $A$ and $B$. All of the $2^{\mathrm{hive}(\bm{\tau})}$ paths in $F^{-1}(\bm{\tau})$ can be obtained from $\bm{\rho}_{AB}$ by exchanging a pair of consecutive edges of types $A$ and $B$ in order with the corresponding edge of type $C$. We claim that such an exchange has no effect on the weight of the path --- that is, for all $\bm{\rho} \in F^{-1}(\bm{\tau})$,
\[
w(\bm{\rho}) = w(\bm{\rho}_{AB}) = b^{\mathrm{flip}(\bm{\tau})}.
\]
From this observation, the desired result follows under the identification $t = 1 - \frac{N}{M}$.

To prove the claimed equality, let $\bm{\rho}$ be an arbitrary path in $\mathrm{Path}^{(bt)}(\m)$ with at least one pair of consecutive edges of types $A$ and $B$ in order, given by the triple of pair partitions $\mathfrak{a} \longrightarrow \mathfrak{b} \longrightarrow \mathfrak{c}$. Since the edge $\mathfrak{b} \to \mathfrak{c}$ is of type $B$, the pair partition $\mathfrak{b} \in \P_j$ must contain the pair $\{2j-1, 2j\}$. Moreover, since $\mathfrak{b} = (i~2j-1) \cdot \mathfrak{a}$ for some $1\leq i \leq 2j-2$, it must be the case that $\mathfrak{a}$ contains the pair $\{i, 2j\}$. To evaluate the weight $\omega^{(b)}(\mathfrak{a}, \mathfrak{b})$, we note that in the graph $\Gamma(\mathfrak{a})$, vertices $i$ and $2j-1$ are both adjacent to vertex $2j$. Thus, the vertices $i$ and $2j-1$ have the same charge and $\omega^{(b)}(\mathfrak{a}, \mathfrak{b}) = 1$. Now suppose that $\bm{\rho}'$ is the path obtained from $\bm{\rho}$ by exchanging this pair of consecutive edges $\mathfrak{a} \longrightarrow \mathfrak{b} \longrightarrow \mathfrak{c}$ of types $A$ and $B$ in order with the corresponding edge $\mathfrak{a} \longrightarrow \mathfrak{c}$ of type $C$. Since all edges of type $C$ have weight $1$, it is clear that $w(\bm{\rho}) = w(\bm{\rho}')$ and the desired claim follows.

The analogous statement for $\Hbt_{g,n}(\lambda)$ follows immediately since the exponential in \cref{eq:btmonotone} passes from the possibly disconnected enumeration to the connected enumeration in the usual manner.
\end{proof}

\begin{remark}
The refined topological recursion is known to govern certain $b$-deformations of enumerative problems, with the $b$-monotone Hurwitz numbers providing one example~\cite{CDO24b,osu24}. The form of the partition function in \cref{eq:btmonotone} --- in particular, the fact that the product in the first line involves a rational function of the $b$-content ---implies that the $bt$-monotone Hurwitz numbers also satisfy the refined topological recursion. We note that the refined topological recursion is currently only defined for a restricted set of spectral curves. One expects that this definition can be broadened so that the refined topological recursion could be applied to a much larger class of $b$-deformations of enumerative problems.
\end{remark}

\begin{table}[pht!]
\caption{The $bt$-monotone Hurwitz numbers $\Hbt_{g,1}(\lambda)$ for $0 \leq g \leq \frac{3}{2}$ and $2 \leq \lambda \leq 6$.}
\label{tab:bthurwitz1}
\begin{tabularx}{\textwidth}{ccX} \toprule
$g$ & $\lambda$ & $\lambda \, \Hbt_{g,1}(\lambda)$ \\ \midrule
$0$ &$(2)$ & $t$ \\
 & $(3)$ & $t^2 + t$ \\
 & $(4)$ & $t^3 + 3t^2 + t$ \\
 & $(5)$ & $t^4 + 6t^3 + 6t^2 + t$ \\
 & $(6)$ & $t^5 + 10t^4 + 20t^3 + 10t^2 + t$ \\ \midrule
$\frac{1}{2}$ & $(2)$ & $bt$ \\
 & $(3)$ & $3bt^2 + 3bt$ \\
 & $(4)$ & $6bt^3 + 17bt^2 + 6bt$ \\
 & $(5)$ & $10bt^4 + 55bt^3 + 55bt^2 + 10bt$ \\
 & $(6)$ & $15bt^5 + 135bt^4 + 262bt^3 + 135bt^2 + 15bt$ \\ \midrule
$1$ & $(2)$ & $(b^2+b+1)t$ \\
 & $(3)$ & $(7b^2+5b+5)t^2 + (7b^2+5b+5)t$ \\
 & $(4)$ & $(25b^2+15b+15)t^3 + (68b^2+40b+40)t^2 + (25b^2+15b+15)t$ \\
 & $(5)$ & $(65b^2+35b+35) t^4 + (335b^2+175b+175) t^3 + (335b^2+175b+175) t^2 + (65b^2+35b+35) t$ \\ 
 & $(6)$ & $(140b^2+70b+70)t^5 + (1162b^2+560b+560)t^4 + (2202b^2+1050b+1050)t^3$ \\
 & & $\quad+ (1162b^2+560b+560)t^2 + (140b^2+70b+70)t$ \\ \midrule
$\frac{3}{2}$ & $(2)$ & $( b^3+2b^2+2b) t$ \\
 & $(3)$ & $(15b^3+24b^2+24b) t^2 + (15b^3+24b^2+24b) t$ \\
 & $(4)$ & $(90b^3+127b^2+127b) t^3 + (238b^3+332b^2+332b) t^2 + (90b^3+127b^2+127b) t$ \\
 & $(5)$ & $(350b^3+455b^2+455b) t^4 + (1720b^3+2195b^2+2195b) t^3 + (1720b^3+2195b^2+2195b) t^2$ \\ 
 & & $\quad+ (350b^3+455b^2+455b) t$ \\
 & $(6)$ & $(1050b^3+1288b^2+1288b)t^5 + (8196b^3+9823b^2+9823b) t^4 + (15246b^3+18148b^2+18148b) t^3$ \\
 & & $\quad+ (8196b^3+9823b^2+9823b) t^2 + (1050b^3+1288b^2+1288b)t$ \\ \bottomrule
\end{tabularx}
\end{table}

\begin{table}[pht!]
\caption{The $bt$-monotone Hurwitz numbers $\Hbt_{g,2}(\lambda_1, \lambda_2)$ for $0 \leq g \leq 1$ and $2 \leq |\lambda| \leq 4$.}
\label{tab:bthurwitz2}
\begin{tabularx}{\textwidth}{ccX} \toprule
$g$ & $\lambda$ & $\prod \lambda_i \, \Hbt_{g,n}(\lambda_1, \ldots, \lambda_n)$ \\ \midrule
$0$ & $(1,1)$ & $(b+1) t$ \\
 & $(2,1)$ & $(2b+2) t^2 + (2b+2) t$ \\
 & $(3,1)$ & $(3b+3) t^3 + (9b+9) t^2 + (3b+3) t$ \\
 & $(2,2)$ & $(4b+4) t^3 + (10b+10) t^2 + (4b+4) t$ \\ \midrule
$\frac{1}{2}$ & $(1,1)$ & $(b^2+b) t$ \\
 & $(2,1)$ & $(6b^2+6b) t^2 + (6b^2+6b) t$ \\
 & $(3,1)$ & $(18b^2+18b) t^3 + (51b^2+51b) t^2 + (18b^2+18b) t$ \\
 & $(2,2)$ & $(22b^2+22b) t^3 + (56b^2+56b) t^2 + (22b^2+22b) t$ \\ \midrule
$1$ & $(1,1)$ & $(b^3+2b^2+2b+1) t$ \\
 & $(2,1)$ & $(14b^3+24b^2+20b+10) t^2 + (14b^3+24b^2+20b+10)$ \\
 & $(3,1)$ & $(75b^3+120b^2+90b+45) t^3 + (204b^3+324b^2+240b+120) t^2 + (75b^3+120b^2+90b+45) t$ \\
 & $(2,2)$ & $(86b^3+136b^2+100b+50) t^3 + (220b^3+348b^2+256b+128) t^2$ \\
 & & $\quad + (86b^3+136b^2+100b+50) t$ \\ \bottomrule
\end{tabularx}
\end{table}

\begin{table}[ht!]
\caption{The $bt$-monotone Hurwitz numbers $\Hbt_{g,3}(\lambda_1, \lambda_2, \lambda_3)$ for $0 \leq g \leq 1$ and $3 \leq |\lambda| \leq 4$.}
\label{tab:bthurwitz3}
\begin{tabularx}{\textwidth}{ccX} \toprule
$g$ & $\lambda$ & $\prod \lambda_i \, \Hbt_{g,n}(\lambda_1, \ldots, \lambda_n)$ \\ \midrule
$0$ & $(1,1,1)$ & $(4b^2+8b+4) t^2 + (4b^2+8b+4) t$ \\
 & $(2,1,1)$ & $(10b^2+20b+10) t^3 + (28b^2+56b+28) t^2 + (10b^2+20b+10) t$ \\ \midrule
$\frac{1}{2}$ & $(1,1,1)$ & $(12b^3+24b^2+12b) t^2 + (12b^3+24b^2+12b) t$ \\
 & $(2,1,1)$ & $(58b^3+116b^2+58b) t^3 + (158b^3+316b^2+158b) t^2 + (58b^3+116b^2+58b) t$ \\ \midrule
$1$ & $(1,1,1)$ & $(28b^4+76b^3+88b^2+60b+20) t^2 + (28b^4+76b^3+88b^2+60b+20) t$ \\ 
 & $(2,1,1)$ & $(236b^4+612b^3+656b^2+420b+140) t^3 + (628b^4+1624b^3+1732b^2+1104b+368) t^2$ \\
 & & $\quad+ (236b^4+612b^3+656b^2+420b+140) t$ \\ \bottomrule
\end{tabularx}
\end{table}

\subsection{Real-rootedness and interlacing}

The $bt$-monotone Hurwitz numbers can be computed effectively using the recursion of \cref{cor:cjfrecursion}, and some of these appear in \cref{tab:bthurwitz1,tab:bthurwitz2,tab:bthurwitz3}. Such data suggests interesting structure concerning the coefficients and roots of the $bt$-monotone Hurwitz numbers, considered as polynomials in $t$.

For instance, it appears that $(b+1)^{n-1}$ is a factor of each coefficient of $\Hbt_{g,n}(\mu_1, \ldots, \mu_n)$. This is a direct consequence of the recursion of \cref{cor:cjfrecursion} and can be shown rigorously by induction. On the other hand, a combinatorial explanation for this factor using the combinatorial interpretation for $bt$-monotone Hurwitz numbers of \cref{prop:btmonotonecombinatorics} does not appear to be immediate.

In the previous work of the authors and Moskovsky, it was shown that the $t$-monotone Hurwitz numbers, which can be obtained by setting $b = 0$ in the current setting, are $\Lambda$-polynomials~\cite{cou-do-mos23}. A $\Lambda$-polynomial is a polynomial whose sequence of coefficients is non-negative, palindromic and unimodal~\cite{bre90}. The following results states that a similar statement holds for the $bt$-monotone Hurwitz numbers. We omit the proof, which is almost identical to that of~\cite[Proposition~3.9]{cou-do-mos23} and uses certain additive and multiplicative closure properties of the set of $\Lambda$-polynomials. The only caveat is that we require $b$ to be positive.

\begin{proposition}
For $g \geq 0$, $n \geq 1$, $|\mu| \geq 2$ and any positive value of $b$, the $bt$-monotone Hurwitz number $\Hbt_{g,n}(\mu_1, \ldots, \mu_n)$ is a $\Lambda$-polynomial in $t$ of degree $|\mu| - 1$.
\end{proposition}

In the previous work of the authors and Moskovsky, various conjectures were stated concerning the roots of the $t$-monotone Hurwitz numbers~\cite{cou-do-mos23}. Extensive numerical evidence suggests that these conjectures remain true for the $bt$-monotone Hurwitz numbers, as long as $b$ avoids the problematic values $0$ and $-1$. Note that these only cause an issue since $\Hbt_{g,n}(\mu_1, \ldots, \mu_n)$ can have global factors of $b$ (for $g$ non-integral) and $b+1$ (for $n \geq 2$), so the polynomial can vanish if $b$ is equal to $0$ or $-1$.

\begin{conjecture}[Real-rootedness] \label{con:btrealrooted}
For $g \geq 0$, $n \geq 1$, $\mu_1, \ldots, \mu_n \geq 1$ and any value of $b \in \mathbb{R} \setminus \{0, -1\}$, the $bt$-monotone Hurwitz number $\Hbt_{g,n}(\mu_1, \ldots, \mu_n)$ is a real-rooted polynomial in $t$.
\end{conjecture}

The evidence suggests that the roots of the $bt$-monotone Hurwitz numbers are not only real, but also possess interesting structure relative to each other. We say that a polynomial $P$ {\em interlaces} a polynomial $Q$ if
\begin{itemize}
\item the degree of $P$ is $n$ and the degree of $Q$ is $n + 1$ for some positive integer $n$;
\item $P$ has $n$ real roots $a_1 \leq a_2 \leq \cdots \leq a_n$ and $Q$ has $n+1$ real roots $b_1 \leq b_2 \leq \cdots \leq b_{n+1}$, allowing for multiplicity; and
\item $b_1 \leq a_1 \leq b_2 \leq a_2 \leq \cdots \leq b_n \leq a_n \leq b_{n+1}$.
\end{itemize}
By convention, we also say that a polynomial $P$ interlaces a polynomial $Q$ if $P$ is constant and $Q$ is affine.

\begin{conjecture}[Interlacing] \label{con:btinterlacing}
For $g \geq 0$, $n \geq 1$, $\mu_1, \ldots, \mu_n \geq 1$ and any value of $b \in \mathbb{R} \setminus \{0, -1\}$, the $bt$-monotone Hurwitz number $\Hbt_{g,n}(\mu_1, \ldots, \mu_n)$ interlaces each of the $n$ polynomials
\[
\Hbt_{g,n}(\mu_1+1, \mu_2, \ldots, \mu_n), \quad \Hbt_{g,n}(\mu_1, \mu_2+1, \ldots, \mu_n), \quad \ldots, \quad \Hbt_{g,n}(\mu_1, \mu_2, \ldots, \mu_n+1).
\]
\end{conjecture}

The real-rootedness and interlacing properties of \cref{con:btrealrooted,con:btinterlacing} have been checked computationally for $\Hbt_{g,n}(\mu_1, \ldots, \mu_n)$ in many cases, including the following.
\begin{itemize}
\item $g \in \{0, \frac{1}{2}, 1, \frac{3}{2}\}$, ~ $1 \leq n \leq 5$, ~ $1 \leq |\mu| \leq 16$, ~ $b \in \{-5, -4, -3, -2, 1, 2, 3, 4, 5\}$
\item $g \in \{2, \frac{5}{2}, 3, \frac{7}{2}\}$, ~ $1 \leq n \leq 5$, ~ $1 \leq |\mu| \leq 14$, ~ $b \in \{-5, -4, -3, -2, 1, 2, 3, 4, 5\}$
\item $g \in \{4, \frac{9}{2}, 5, \frac{11}{2}\}$, ~ $1 \leq n \leq 5$, ~ $1 \leq |\mu| \leq 12$, ~ $b \in \{-5, -4, -3, -2, 1, 2, 3, 4, 5\}$
\end{itemize}
This constitutes over 4000 independent checks, with each such check involving the 9 values of $b$ listed above.

The fact that the real-rootedness and interlacing conjectures for $t$-monotone Hurwitz numbers observed previously extend to the $bt$-monotone Hurwitz numbers appears rather remarkable. As claimed in our previous work with Moskovsky, these conjectural properties might not be isolated and may extend to broad classes of enumerative problems~\cite{cou-do-mos23}. In light of the data presented here, such classes of enumerative problems could include $b$-deformations and it may also be the case that the extra structure provided by $b$-deformations may shed light on the original conjectures.

\section{Introducing the \texorpdfstring{$b$}{b}-deformed Jucys--Murphy operators} \label{sec:jucys-murphy}

In this section, we introduce a $b$-deformation of the Jucys--Murphy elements in the symmetric group algebra, motivated by the $b$-deformation of the Weingarten graph appearing in the previous section.

\subsection{Definition and conjectural properties}

One can think of the Jucys--Murphy elements, introduced in \cref{eq:jucysmurphy}, as encapsulating the adjacency structure of the unitary Weingarten graph, which governs the Weingarten calculus for unitary groups~\cite{col-mat17}. A similar statement holds for the odd Jucys--Murphy elements $J_1, J_3, J_5, \ldots$ and the orthogonal Weingarten graph. By analogy, we introduce below the $b$-deformed Jucys--Murphy operators, which encapsulate the adjacency structure of the $b$-Weingarten graph $\mathcal{G}^{(b)}$ of \cref{subsec:weingarten_B}.

\begin{definition} \label{def:bJM}
For $k$ a positive integer, let $\V_k = \mathbb{C}(b)[\P_k]$ be the vector space with basis the set of pair partitions of $\{1, 2, \ldots, 2k\}$. Define the {\em $b$-deformed Jucys--Murphy operators} $\J_1, \J_2, \ldots, \J_k: \V_k \to \V_k$ by
\[
\J_i(\m) = \sum_{a=1}^{2i-2} \omega^{(b)}((a~2i-1) \cdot \m, \m) \, (a~2i-1) \cdot \m,
\]
where $\m \in \P_k$ and $\omega^{(b)}$ is the weight function of \cref{def:weightfunction}. We interpret the formula for $i = 1$ as $\J_1 = 0$ and refer to these operators collectively as $\J$-operators.
\end{definition}

The operator $\J_k$ acting on $\V_k$ encapsulates the adjacency structure of the type $A$ edges in the $b$-Weingarten graph at level $k$ in the following way. For any $\n, \m \in \P_k$, the entry of the matrix representing $\J_k$ in the row labelled $\n$ and the column labelled $\m$ is non-zero if and only if there is a type $A$ edge $\n \to \m$ in the $b$-Weingarten graph $\mathcal{G}^{(b)}$. Moreover, the entry of the matrix is equal to the weight $\omega^{(b)}(\n, \m)$ of the edge.

Note that due to the natural inclusions $\P_j \subseteq \P_k$ and $S_{2j} \subseteq S_{2k}$ for $j < k$, we need not distinguish between the operator $\J_i: \V_j \to \V_j$ and the restriction of the operator $\J_i: \V_k \to \V_k$ to $\V_j$.

It is natural to ask which properties of the usual Jucys--Murphy elements lift to their $b$-deformed counterparts introduced in \cref{def:bJM}. In the following, we enumerate some fundamental results concerning Jucys--Murphy elements~\cite{oko-ver96}.

\begin{proposition} ~ \label{prop:jm}
\begin{enumerate} [label=({\alph*})]
\item The Jucys--Murphy elements commute --- that is, $J_m J_n = J_n J_m$ for positive integers $m$ and $n$. They generate a maximal commutative subalgebra $X(k) = \langle J_1, J_2, \ldots, J_k \rangle$ of the group algebra $\mathbb{C}[S_k]$, known as the {\em Gelfand--Tsetlin algebra}.

\item There exists a basis $\{w_\mathsf{T}: \mathsf{T} \in \Tab(k)\}$ of $X(k)$, known as the {\em Gelfand--Tsetlin basis}, where $\Tab(k)$ denotes the set of standard Young tableaux with $k$ boxes. The Jucys--Murphy elements $J_1, J_2, \ldots, J_k$ act diagonally on $X(k)$, according to the formula
\[
J_i \cdot w_\mathsf{T} = c_0(\mathsf{T}_i) \, w_\mathsf{T}, \qquad \text{for } i = 1, 2, \ldots, k,
\]
where $\mathsf{T}_i$ is the box labelled $i$ in the tableau $\mathsf{T}$ and $c_0$ denotes the content.

\item For any tableau $\mathsf{S} \in \Tab(j)$ with $j \leq k$, the vector $w_\mathsf{S} \in X(j) \subseteq X(k)$ satisfies
\[
w_\mathsf{S} = \sum_{\substack{\mathsf{T} \in \Tab(k) \\ \mathsf{S} \subseteq \mathsf{T}}} w_\mathsf{T}.
\]

\item For any tableau $\mathsf{T} \in \Tab(k)$, the vector $w_\mathsf{T} \in X(k)$ satisfies the recursive equation
\[
w_\mathsf{T} = \Bigg(\prod_{\substack{\mathsf{S} \in \Tab(k) \\ \mathsf{S} \neq \mathsf{T}, \overline{\mathsf{S}} = \overline{\mathsf{T}}}} \frac{J_k - c_0(\mathsf{S}_k)}{c_0(\mathsf{T}_k) - c_0(\mathsf{S}_k)} \Bigg) \cdot w_{\overline{\mathsf{T}}},
\]
where $\overline{\mathsf{T}}$ denotes the tableau obtained from $\mathsf{T}$ by removing the box labelled $k$.
\end{enumerate}
\end{proposition}

A naive generalisation of \cref{prop:jm} cannot hold since the $\J$-operators do not commute in general. For example, one can check that $\J_2 \J_3(\m) \neq \J_3 \J_2(\m)$, for $\m = (13 \mmid 24 \mmid 56) \in \V_3$. However, it appears that the $\J$-operators exhibit better behaviour when restricted to the following analogue of the Gelfand--Tsetlin algebra.

\begin{definition}
For $k$ a positive integer, let 
\[
\X(k) = \langle \J_1, \J_2, \ldots, \J_k \rangle \cdot \e_k \subseteq \V_k.
\]
That is, $\X(k)$ is the orbit of $\e_k$ under the action of the algebra of $\J$-operators.
\end{definition}

Motivated by the role of the Jucys--Murphy elements in the representation theory of the symmetric groups, we form the following conjectures. These are analogous to the results of \cref{prop:jm}, which correspond to the case $b = 0$.

\begin{conjecture} \label{con:jucys-murphy} ~
\begin{enumerate} [label=({\alph*})]
\item The $\J$-operators commute when restricted to $\X(k)$ --- that is, $\J_m \J_n (v) = \J_n \J_m (v)$ for $1 \leq m \leq n \leq k$ and for all $v \in \X(k)$.

\item There exists a basis $\{\w_\mathsf{T}: \mathsf{T} \in \Tab(k)\}$ of $\X(k)$. The $\J$-operators $\J_1, \J_2, \ldots, \J_k$ act diagonally on $\X(k)$, according to the formula
\[
\J_i \cdot \w_\mathsf{T} = c_b(\mathsf{T}_i) \, \w_\mathsf{T}, \qquad \text{for } i = 1, 2, \ldots, k,
\]
where $c_b$ denotes the $b$-content.

\item For any tableau $\mathsf{S} \in \Tab(j)$ with $j \leq k$, the vector $\w_\mathsf{S} \in \X(j) \subseteq \X(k)$ satisfies
\[
\w_\mathsf{S} = \sum_{\substack{\mathsf{T} \in \Tab(k) \\ \mathsf{S} \subseteq \mathsf{T}}} \w_\mathsf{T}.
\]
In particular, the special case $j = 0$ leads to the statement
\[
\e_k = \sum_{\mathsf{T} \in \Tab(k)} \w_\mathsf{T}.
\]

\item For any tableau $\mathsf{T} \in \Tab(k)$, the vector $\w_\mathsf{T} \in \X(k)$ satisfies the recursive equation
\[
\w_\mathsf{T} = \Bigg(\prod_{\substack{\mathsf{S} \in \Tab(k) \\ \mathsf{S} \neq \mathsf{T}, \overline{\mathsf{S}} = \overline{\mathsf{T}}}} \frac{\J_k - c_b(\mathsf{S}_k)}{c_b(\mathsf{T}_k) - c_b(\mathsf{S}_k)} \Bigg) \cdot \w_{\overline{\mathsf{T}}},
\]
where $\overline{\mathsf{T}}$ denotes the tableau obtained from $\mathsf{T}$ by removing the box labelled $k$.
\end{enumerate}
\end{conjecture}

\cref{con:jucys-murphy} has been verified directly for $1 \leq k \leq 5$ using SageMath. Since the vector space $\V_k$ has dimension $(2k-1)!!$, extending these calculations to the case $k = 6$ naively involves working with $10395 \times 10395$ matrices, beyond the computational power readily available to us. It is likely that a more efficient implementation should be able to confirm the conjecture for larger values of $k$, although the verifications for $1 \leq k \leq 5$ already provide reasonably compelling evidence.

\cref{prop:jm} follows from an analysis of the branching structure of the irreducible representations of the symmetric groups~\cite{oko-ver96}. More generally, variations on this theme arise for certain multiplicity-free sequences of finite-dimensional algebras~\cite{DLS18}. However, the current context does not immediately suggest an obvious candidate for a family of algebraic structures whose representation theory would lead to \cref{con:jucys-murphy}. A naive approach would be to seek a family of algebras on the spaces $\V_1, \V_2, \V_3, \ldots$, but we were unable to accomplish this. On the other hand, it appears that the subspaces $\X(1), \X(2), \X(3), \ldots$ are strongly analogous to the Gelfand--Tsetlin algebras that arise in the representation theory of symmetric groups.

\begin{example} \label{ex:jmconjecture}
The $k = 3$ case of \cref{con:jucys-murphy} is sufficiently small to be computed by hand, yet sufficiently large to be illustrative. Consider the four standard Young tableaux with three boxes.

\ytableausetup{mathmode, boxframe=0.3mm,boxsize=6mm}
\[
\begin{ytableau}
\none[\mathsf{A}] & \none[=~~] & \mathsf{1} \\
\none & \none & \mathsf{2} \\
\none & \none & \mathsf{3}
\end{ytableau}
\qquad \qquad
\begin{ytableau}
\none[\mathsf{B}] & \none[=~~] & \mathsf{1} & \mathsf{2} \\
\none & \none & \mathsf{3}
\end{ytableau}
\qquad \qquad
\begin{ytableau}
\none[\mathsf{C}] & \none[=~~] & \mathsf{1} & \mathsf{3} \\ 
\none & \none & \mathsf{2}
\end{ytableau}
\qquad \qquad
\begin{ytableau}
\none[\mathsf{D}] & \none[=~~] & \mathsf{1} & \mathsf{2} & \mathsf{3}
\end{ytableau}
\]

The corresponding basis vectors of $\X(3)$ can be recursively computed via part (d) of \cref{con:jucys-murphy} to produce the following.
\begin{align*}
\w_{\mathsf{A}} &= \frac{(b+1) - \J_2}{b+2} \cdot \frac{(b+1) - \J_3}{b+3} \cdot \e_3 \\
\w_{\mathsf{B}} &= \frac{\J_2 + 1}{b+2} \cdot \frac{2(b+1) - \J_3}{2b+3} \cdot \e_3 \\
\w_{\mathsf{C}} &= \frac{(b+1) - \J_2}{b+2} \cdot \frac{\J_3 + 2}{b+3} \cdot \e_3 \\
\w_{\mathsf{D}} &= \frac{\J_2 + 1}{b+2} \cdot \frac{\J_3 + 1}{2b+3} \cdot \e_3
\end{align*}

One can then verify that the operators $\J_1, \J_2, \J_3$ are mutually diagonalised by these basis vectors. Furthermore, one can compute the eigenvalues to be the following, thus verifying parts (a) and (b) of \cref{con:jucys-murphy}.
\begin{align*}
\J_1 \cdot \w_\mathsf{A} &= 0 &
\J_2 \cdot \w_\mathsf{A} &= - \w_\mathsf{A} &
\J_3 \cdot \w_\mathsf{A} &= -2 \, \w_\mathsf{A} \\
\J_1 \cdot \w_\mathsf{B} &= 0 &
\J_2 \cdot \w_\mathsf{B} &= (b+1) \, \w_\mathsf{B} &
\J_3 \cdot \w_\mathsf{B} &= - \w_\mathsf{B} \\
\J_1 \cdot \w_\mathsf{C} &= 0 &
\J_2 \cdot \w_\mathsf{C} &= - \w_\mathsf{C} &
\J_3 \cdot \w_\mathsf{C} &= (b+1) \, \w_\mathsf{C} \\
\J_1 \cdot \w_\mathsf{D} &= 0 &
\J_2 \cdot \w_\mathsf{D} &= (b+1) \, \w_\mathsf{D} &
\J_3 \cdot \w_\mathsf{D} &= 2(b+1) \, \w_\mathsf{D}
\end{align*}

Finally, part (c) of \cref{con:jucys-murphy} can then be checked in a straightforward manner --- for example, the fact that $\w_{\mathsf{A}} + \w_{\mathsf{B}} + \w_{\mathsf{C}} + \w_{\mathsf{D}} = \e_3$ follows from the algebraic identity
\[
\frac{(b+1) - \J_2}{b+2} \cdot \frac{(b+1) - \J_3}{b+3} +
\frac{\J_2 + 1}{b+2} \cdot \frac{2(b+1) - \J_3}{2b+3} +
\frac{(b+1) - \J_2}{b+2} \cdot \frac{\J_3 + 2}{b+3} +
\frac{\J_2 + 1}{b+2} \cdot \frac{\J_3 + 1}{2b+3} = 1.
\]
\end{example}

Recall that the orthogonality relations for the Weingarten function $\wga$ of \cref{thm:orthogonality_A} led to the succinct formula of \cref{prop:jucysmurphy_A} for the Weingarten function in terms of Jucys--Murphy elements. An analogous, more general, statement can be made for the Weingarten function $\wgb$ for which \cref{prop:orthogonality_BT} plays the role of the orthogonality relations. The precise statement is provided below without proof, which would be almost identical to that of \cref{prop:jucysmurphy_A}.

\begin{proposition} \label{prop:jucysmurphy_b}
For each positive integer $k$, we have the following equality in $\V_k$.
\[
\sum_{\m \in \P_k} \wgb(\m) \, \m = \frac{M + \J_k}{N + \J_k} \cdots \frac{M + \J_2}{N + \J_2} \cdot \frac{M + \J_1}{N + \J_1} \cdot \e_k
\]
\end{proposition}

\subsection{Relations with symmetric functions}

As previously mentioned, there remain many unresolved conjectures concerning Jack functions, such as those of Hanlon~\cite{han88}, Stanley~\cite{sta89}, Goulden--Jackson~\cite{gou-jac96}, and Alexandersson--Haglund--Wang~\cite{ale-hag-wan21}. Furthermore, the notion of $b$-deformation that is inspired by Jack functions appears to fit into the more general framework of refinement in mathematical physics, although this picture is far from complete~\cite{CDO24b}. Thus, it would be desirable to have a deeper understanding of Jack functions and $b$-deformations. The main obstruction is a lack of the representation-theoretic foundations that are known when 
\begin{itemize}
\item $b = 0$: corresponding to the theory of Schur functions, symmetric groups $S_k$, and general linear groups $GL_n(\mathbb{C})$, and 
\item $b = 1$: corresponding to the analogous theory of zonal polynomials, Gelfand pairs $(S_{2k}, H_k)$, and Gelfand pairs $(GL_n(\mathbb{R}), O_n)$.
\end{itemize}

In this section, we consider some conjectural relations that link the $b$-deformed Jucys--Murphy operators to the theory of symmetric functions. They support the notion that the $b$-deformation obtained by passing from Schur functions to Jack functions is aligned with the $b$-deformation of the Jucys--Murphy elements that we introduced in \cref{def:bJM}. Furthermore, one can speculate that these conjectures hint at a deeper underlying algebraic theory that may help to unlock some of the unresolved problems concerning Jack functions and $b$-deformations.

Drawing analogy from the representation theory underlying the cases $b=0$ and $b=1$, we introduce the following concepts. Denote by $\Z \V_k \subseteq \V_k$ the subspace of vectors whose coefficients with respect to the basis of pair partitions are constant on coset-type. For $\lambda \vdash k$, define
\begin{equation} \label{eq:pwbases}
\p_\lambda = \sum_{\substack{\m \in \P_k \\ \lambda(\m) = \lambda}} \m \qquad \qquad \text{and} \qquad \qquad \w_\lambda = \frac{1}{\mathrm{hook}_b(\lambda) \, \mathrm{hook}_b'(\lambda)} \sum_{\mu \vdash k} \big\langle J_\lambda^{(b)}(\bm{p}), p_\mu \big\rangle_b \, \p_\mu,
\end{equation}
where the $b$-deformed inner product is defined by \cref{eq:innerproduct_b}. It is immediate that $\{\p_\lambda: \lambda \vdash k\}$ is a basis for $\Z \V_k$ and that these basis vectors are analogous to the conjugacy classes $\{C_\lambda: \lambda \vdash k\}$ in the centre of the symmetric group algebra $Z\mathbb{C}[S_k]$. Since the Jack functions form a basis for the algebra of symmetric functions, it follows that $\{\w_\lambda: \lambda \vdash k\}$ is also a basis for $\Z \V_k$. These basis vectors are analogous to the orthogonal idempotents $\{\epsilon_\lambda: \lambda \vdash k\}$ in the centre of the symmetric group algebra $Z\mathbb{C}[S_k]$. The Gelfand--Tsetlin basis $\{w_\mathsf{T}: \mathsf{T} \in \Tab(k)\}$ is related to the orthogonal idempotents via the equation
\[
\sum_{\mathsf{T} \in \Tab(\lambda)} w_\mathsf{T} = \epsilon_\lambda,
\]
where $\Tab(\lambda)$ denotes the set of standard Young tableaux with shape $\lambda$. Furthermore, for any symmetric function $f$, we have the relation
\[
f(J_1, J_2, \ldots, J_k) \cdot \epsilon_\lambda = f(\mathrm{cont}_0(\lambda)) \, \epsilon_\lambda.
\]
These observations motivate the following conjectures.

\begin{conjecture} ~ \label{con:bJMsymmetric}
\begin{enumerate} [label=({\alph*})]
\item For each partition $\lambda \vdash k$,
\[
\w_\lambda = \sum_{\mathsf{T} \in \Tab(\lambda)} \w_\mathsf{T}.
\]

\item For any symmetric function $f$ in $k$ variables and any partition $\lambda \vdash k$,
\[
f(\J_1, \J_2, \ldots, \J_k) \cdot \w_\lambda = f(\mathrm{cont}_b(\lambda)) \, \w_\lambda,
\]
where $\mathrm{cont}_b(\lambda)$ is the multiset of $b$-contents of boxes in $\lambda$, as defined in \cref{subsec:jack}. As a corollary, any symmetric function of the $\J$-operators preserves the space $\Z \V_k$.
\end{enumerate}
\end{conjecture}

Observe that part (b) of \cref{con:bJMsymmetric} --- in particular, the expression $f(\J_1, \J_2, \ldots, \J_k)$ --- only makes sense in light of part (a) of \cref{con:jucys-murphy}, which asserts that the $\J$-operators commute on the space $\X(k)$.

Introduce the following $b$-deformation of the Frobenius characteristic map.
\begin{align} \label{eq:char_b}
\mathrm{ch}^{(b)}: \Z \V_k &\to \mathbb{C}(b)[p_1,p_2,\ldots]_k \notag \\
\p_\mu &\mapsto \frac{1}{(1+b)^{\ell(\mu)} z_\mu} \, p_\mu,
\end{align}
This defines a natural vector space isomorphism to the symmetric functions of homogeneous degree $k$, expressed in terms of power-sum symmetric functions. By construction, applying the $b$-deformed characteristic map to the ``orthogonal idempotent'' $\w_\lambda$ recovers the Jack polynomial $J_\lambda^{(b)}$ via the formula
\[
\mathrm{ch}^{(b)} (\w_\lambda) = \frac{J_\lambda^{(b)}(\bm{p})}{\mathrm{hook}_b(\lambda) \, \mathrm{hook}_b'(\lambda)}.
\]
Again, by analogy with the known behaviour of Jucys--Murphy elements and Schur functions, the pull-back of the sum $\J_1 + \J_2 + \cdots + \J_k$ should give rise to the Laplace--Beltrami operator, which appears in the definition of the Jack functions --- see \cref{def:jack}.

\begin{conjecture} \label{con:b-laplace}
For all $\lambda \vdash k$,
\[
\mathrm{ch}^{(b)} \big( (\J_1 + \J_2 + \cdots + \J_k) \cdot \p_\lambda \big) = D(b) \cdot \mathrm{ch}^{(b)}(\p_\lambda).
\]
\end{conjecture}

Recall that the definition of the $\J$-operators relies heavily on the weight function $\omega^{(b)}$ of \cref{def:weightfunction}. The motivation for the construction of the weight function was primarily to recover the Virasoro constraints for the $b$-monotone Hurwitz numbers. However, there should be many such choices of the weight function and it would be desirable to understand to what extent the properties of the corresponding $\J$-operators would depend on this choice.

Observe that these conjectures concerning the $\J$-operators are not entirely independent. In particular, part~(b) of \cref{con:bJMsymmetric} and \cref{con:b-laplace} would follow as consequences of \cref{con:jucys-murphy} and part (a) of \cref{con:bJMsymmetric}.

\cref{con:bJMsymmetric,con:b-laplace} have been verified directly for $1 \leq k \leq 5$ using SageMath. We conclude with a result that constitutes evidence of a more general nature towards part (b) of \cref{con:bJMsymmetric}. Since we have not yet proven that the $\J$-operators commute on the space $\X(k)$, let us define $f^{\geq}(\J_1, \J_2, \ldots, \J_k)$ for a symmetric function $f$ in $k$ variables as the result of ordering each monomial $\J_{i_1} \J_{i_2} \cdots \J_{i_r}$ such that $i_1 \geq i_2 \geq \cdots \geq i_r$.

\begin{proposition} \label{prop:bJucysMurphy}
For $k\geq 1$ and $r\geq 0$, the elementary and homogeneous symmetric polynomials in the $\J$-operators obey the equations
\begin{align}
e_r^{\geq}(\J_1, \J_2, \ldots, \J_k) \cdot \e_k &= \sum_{\substack{\lambda \vdash k \\ \ell(\lambda) = k-r}} \p_\lambda, \label{eq:bJucysMurphy1} \\ 
h_r^{\geq}(\J_1, \J_2, \ldots, \J_k) \cdot \e_k &= \sum_{\lambda \vdash k} h_r(\mathrm{cont}_b(\lambda)) \, \w_\lambda. \label{eq:bJucysMurphy2}
\end{align}
\end{proposition}

\begin{proof}
For the proof of \cref{eq:bJucysMurphy1}, observe that if $\m \in \P_j$ is a pair partition containing the pair $\{2j-1, 2j\}$, then $\omega^{(b)}((i ~ 2j-1) \cdot \m, \m) = 1$ for all $1 \leq i \leq 2j-2$. It then follows that, for $j \leq k$ and $\m \in \P_j \subseteq \P_k$ containing the pair $\{2j-1, 2j\}$,
\[
\J_j \cdot \m = J_{2j-1} \cdot \m.
\]
Note that this follows from the particular construction of the weight function $\omega^{(b)}$ appearing in \cref{def:weightfunction}. Since $\J_j \cdot \m \in \P_{j}$ for any $\m \in \P_j$, we have that for any $i_1 > i_2 > \cdots > i_r$,
\[
\J_{i_1} \J_{i_2} \cdots \J_{i_r} \cdot \e_k = J_{2i_1-1} J_{2i_2 -1} \cdots J_{2i_r-1} \cdot \e_k.
\]
This leads to the relation
\[
e_r^{\geq}(\J_1, \J_2, \ldots, \J_k) \cdot \e_k = e_r(J_1, J_3, \ldots, J_{2k-1}) \cdot \e_k = \sum_{\substack{\lambda \vdash k \\ \ell(\lambda) = k-r}} \p_\lambda,
\]
where the second equality follows from~\cite[Proposition~3.1]{mat11}. This completes the proof of \cref{eq:bJucysMurphy1}.

For the proof of \cref{eq:bJucysMurphy2}, we actually prove the stronger claim that for the sum of monomial symmetric functions
\[
g_{r,\ell} = \sum_{\substack{\nu \vdash r \\ \ell(\nu) = \ell}} m_\nu,
\]
we have
\begin{align} \label{eq:bJucysMurphy3}
g_{r,\ell}^{\geq}(\J_1, \J_2, \ldots, \J_k) \cdot \e_k &= \sum_{\lambda \vdash k} g_{r,\ell}(\mathrm{cont}_b(\lambda)) \, \w_\lambda.
\end{align}
Then \cref{eq:bJucysMurphy2} would follow from \cref{eq:bJucysMurphy3} by summing over $\ell = 1, 2, \ldots, r$.

Now for the proof of \cref{eq:bJucysMurphy3}, we draw on results from previous sections. We start with the statement of \cref{prop:jucysmurphy_b} and expand the right side using $\h = -\frac{1}{N}$ and $t = 1 - \frac{N}{M}$.
\begin{align} \label{eq:Jsymproof1}
\sum_{\m \in \P_k} \wgb(\m) \, \m &= \frac{M + \J_k}{N + \J_k} \cdots \frac{M + \J_2}{N + \J_2} \cdot \frac{M + \J_1}{N + \J_1} \cdot \e_k \notag \\
&= \frac{1}{(1-t)^k} \bigg( 1 + \frac{\h t \J_k}{1 - \h \J_k} \bigg) \cdots \bigg( 1 + \frac{\h t \J_2}{1 - \h \J_2} \bigg) \bigg( 1 + \frac{\h t \J_1}{1 - \h \J_1} \bigg) \cdot \e_k \notag \\
&= \frac{1}{(1-t)^k} \sum_{r=0}^\infty \h^r \sum_{\nu \vdash r} t^{\ell(\nu)} m_\nu^{\geq}(\J_1, \J_2, \ldots, \J_k) \cdot \e_k
\end{align}
On the other hand, \cref{prop:expansion_BT} allows us to express the expansion as follows.
\begin{equation} \label{eq:Jsymproof2}
\sum_{\m \in \P_k} \wgb(\m) \, \m = \frac{1}{(1-t)^k} \sum_{r=0}^\infty \h^r \sum_{\mu \vdash k } \hbt_r(\mu) \, \p_\mu.
\end{equation}
Equating the expressions in \cref{eq:Jsymproof1} and \cref{eq:Jsymproof2} leads to
\begin{equation} \label{eq:Jsymproof3}
\sum_{\mu \vdash k} \hbt_r(\mu) \, \p_\mu = \sum_{\nu \vdash r} t^{\ell(\nu)} m_\nu^{\geq}(\J_1, \J_2, \ldots, \J_k) \cdot \e_k.
\end{equation}

Now recall the expressions for the partition function given by \cref{eq:btmonotone}.
\begin{align*}
Z^{(bt)}(\bm{p}; \h, z) &= \sum_{k=0}^\infty \frac{z^k}{\h^k} \sum_{\lambda\vdash k} \Bigg( \prod_{\Box \in \lambda} \frac{1- (1-t) \h c_b(\Box)}{1 - \h c_b(\Box)} \Bigg) \frac{J_\lambda^{(b)}(\bm{p})}{\mathrm{hook}_b(\lambda) \, \mathrm{hook}_b'(\lambda)} \\
&= \sum_{k \geq 0} \frac{z^k}{\h^k} \sum_{r \geq 0} \sum_{n \geq 1} \frac{\h^r}{n!} \sum_{\mu_1, \ldots, \mu_n \geq 1} \frac{\hbt_r(\mu_1, \ldots, \mu_n)}{\mu_1 \cdots \mu_n} \, \frac{p_{\mu_1} \cdots p_{\mu_n}}{(b+1)^n}
\end{align*}

Under the assumption $|\mu| = k$, this allows us to calculate $\hbt_r(\mu)$ as a coefficient in the following way.
\begin{align*}
\hbt_r(\mu) &= [\h^r p_\mu] \sum_{\lambda\vdash k} \Bigg( \prod_{\Box\in \lambda}\frac{1 - (1-t) \h c_b(\Box)}{1- \h c_b(\Box)} \Bigg) \frac{J_\lambda^{(b)}(\bm{p})}{\mathrm{hook}_b(\lambda) \, \mathrm{hook}_b'(\lambda)} \\
&= \sum_{\lambda \vdash k} \sum_{\nu \vdash r} t^{\ell(\nu)} m_\nu(\mathrm{cont}_b(\lambda)) \frac{\langle J_\lambda^{(b)}(\bm{p}), p_\mu \rangle}{\mathrm{hook}_b(\lambda) \, \mathrm{hook}_b'(\lambda)}
\end{align*}

It follows from this equation and the definition of $\w_\lambda$ that
\begin{align} \label{eq:Jsymproof4}
\sum_{\mu \vdash k} \hbt_r(\mu) \, \p_\mu &= \sum_{\mu \vdash k} \sum_{\lambda \vdash k} \sum_{\nu \vdash r} t^{\ell(\nu)} m_\nu(\mathrm{cont}_b(\lambda)) \frac{\langle J_\lambda^{(b)}(\bm{p}), p_\mu \rangle}{\mathrm{hook}_b(\lambda) \, \mathrm{hook}_b'(\lambda)} \, \p_\mu \notag \\
&= \sum_{\lambda \vdash k} \sum_{\nu \vdash r} t^{\ell(\nu)} m_\nu(\mathrm{cont}_b(\lambda)) \, \w_\lambda
\end{align}

Finally, equate the expressions appearing in \cref{eq:Jsymproof3,eq:Jsymproof4} and extract the coefficient of $t^\ell$ to obtain
\[
\sum_{\substack{\nu \vdash r \\ \ell(\nu) = \ell}} m_\nu^{\geq}(\J_1, \J_2, \ldots, \J_k) \cdot \e_k = \sum_{\lambda \vdash k} \sum_{\substack{\nu \vdash r \\ \ell(\nu) = \ell}} m_\nu(\mathrm{cont}_b(\lambda)) \, \w_\lambda.
\]
Then \cref{eq:bJucysMurphy3} follows directly from the definition of $g_{r,\ell}$, which completes the proof.
\end{proof}

\bibliographystyle{plain}
\bibliography{bt-monotone-hurwitz}

\end{document}